%% file: CPII-AMS.tex
\documentclass[oneside]{amsart}
\usepackage[T1]{fontenc}
\usepackage[latin9]{inputenc}
\usepackage{amsthm}
\usepackage{amstext}
\usepackage{amssymb}
\usepackage{esint}
\usepackage[all]{xy}

\makeatletter

\providecommand{\tabularnewline}{\\}

\numberwithin{equation}{section}
\numberwithin{figure}{section}
\theoremstyle{plain}
\newtheorem{thm}{\protect\theoremname}[section]
  \theoremstyle{definition}
  \newtheorem{defn}[thm]{\protect\definitionname}
  \theoremstyle{plain}
  \newtheorem{prop}[thm]{\protect\propositionname}
  \theoremstyle{plain}
  \newtheorem{cor}[thm]{\protect\corollaryname}
  \theoremstyle{plain}
  \newtheorem{lem}[thm]{\protect\lemmaname}
  \theoremstyle{remark}
  \newtheorem{rem}[thm]{\protect\remarkname}

\makeatother

  \providecommand{\corollaryname}{Corollary}
  \providecommand{\definitionname}{Definition}
  \providecommand{\lemmaname}{Lemma}
  \providecommand{\propositionname}{Proposition}
  \providecommand{\remarkname}{Remark}
\providecommand{\theoremname}{Theorem}

\allowdisplaybreaks

\begin{document}

\title{Crossed Products of Banach Algebras. II.}
\begin{abstract}
In earlier work a crossed product of a Banach algebra was constructed
from a Banach algebra dynamical system $(A,G,\alpha)$ and a class
$\mathcal{R}$ of continuous covariant representations, and its representations
were determined. In this paper the theory is developed further. We
consider the dependence of the crossed product on the class $\mathcal{R}$
and its essential uniqueness. Next we study generalized Beurling algebras:
weighted Bochner spaces of $A$-valued functions on $G$ with a continuous
multiplication. Though not Banach algebras in general, they are isomorphic
to a crossed product of a Banach algebra, and the earlier work therefore
predicts the structure of their representations. Classical results
for the usual Beurling (Banach) algebras of scalar valued functions
are then retrieved as special cases. We also show how, e.g., an anti-covariant
pair of anti-representations of $A$ and $G$ can be viewed as a covariant
pair for a related Banach algebra dynamical system, so that the earlier
work becomes applicable to classes of such other pairs. After including
material on the representations of the projective tensor product of
Banach algebras, we combine this idea with the results already obtained
and describe the two-sided modules over the generalized Beurling algebras,
where again specializing to the scalars gives a classical result.
\end{abstract}

\author{Marcel de Jeu}
\address{Marcel de Jeu, Mathematical Institute, Leiden University, P.O. Box
9512, 2300 RA Leiden, The Netherlands}
\email{mdejeu@math.leidenuniv.nl}

\author{Miek Messerschmidt}
\address{Miek Messerschmidt, Mathematical Institute, Leiden University, P.O.
Box 9512, 2300 RA Leiden, The Netherlands}
\email{mmesserschmidt@gmail.com}

\author{Marten Wortel}
\address{Marten Wortel, School of Mathematics, Statistics \& Actuarial Science, Cornwallis Building,
University of Kent, Canterbury, Kent CT2 7NF, United Kingdom}
\email{marten.wortel@gmail.com}



\maketitle

\def\three_hack{(3)}

\input{CPII-CONTENT}

\section*{Acknowledgements}

Messerschmidt's research was supported by a Vrije Competitie grant
of the Netherlands Organisation for Scientific Research (NWO).

\bibliographystyle{amsplain}
\bibliography{bibliography}

\end{document}

%% file: CPII-CONTENT.tex
\global\long\def\crossedprod{(A\rtimes_{\alpha}G)^{\mathcal{R}}}

\global\long\def\crossedprodA{\crossedprod}

\global\long\def\crossedprodB{(B\rtimes_{\beta}H)^{\mathcal{S}}}

\global\long\def\crossedprodtensor{(A\hat{\otimes}B\rtimes_{\alpha\otimes\beta}G\times H)^{\mathcal{R}\circlearrowleft\mathcal{S}}}

\global\long\def\leftcent{\mathcal{M}_{l}}

\hyphenation{Beur-ling}

\section{Introduction and overview}

This paper is an analytical continuation of \cite{2011arXiv1104.5151D}
where, motivated by the theory of crossed products of $C^{*}$-algebras
and its relevance for the theory of unitary group representations,
a start was made with the theory of crossed products of Banach algebras.
General Banach algebras lack the convenient rigidity of $C^{*}$-algebras
where, e.g., morphisms are automatically continuous and even contractive,
and this makes the task of developing the basics more laborious than
it is for crossed products of $C^{*}$-algebras. Apart from some first
applications, including the usual description of the non-degenerate
(involutive) representations of the crossed product associated with
a $C^{*}$-dynamical system (cf.\,\cite[Theorem 9.3]{2011arXiv1104.5151D}),
\cite{2011arXiv1104.5151D} is basically concerned with one theorem,
the General Correspondence Theorem \cite[Theorem 8.1]{2011arXiv1104.5151D},
most of which is formulated as Theorem \ref{thm:General-Correspondence-Theorem}
below. If $\mathcal{R}$ is a non-empty class of non-degenerate continuous
covariant representations of a Banach algebra dynamical system $(A,G,\alpha)$
-- all notions will be reviewed in Section \ref{sec:Preliminaries-and-recapitulation}
-- then Theorem \ref{thm:General-Correspondence-Theorem} gives a
bijection between the non-degenerate $\mathcal{R}$-continuous covariant
representations of $(A,G,\alpha)$ and the non-degenerate bounded
representations of the crossed product $\crossedprod$, provided that
$A$ has a bounded approximate left identity. In the current paper,
the basic theory is developed further and, in addition, a substantial
part is concerned with generalized Beurling algebras $L^{1}(G,A,\omega;\alpha)$
and their representations. These are weighted Banach spaces of (equivalence
classes) of $A$-valued functions that are also associative algebras
with a multiplication that is continuous in both variables, but they
are not Banach algebras in general, since the norm need not be submultiplicative.
If $A$ equals the scalars, they reduce to the ordinary Beurling algebras
$L^{1}(G,\omega)$ (which are true Banach algebras) for a not necessarily
abelian group $G$. We will describe the non-degenerate bounded representations
of generalized Beurling algebras as a consequence of the General Correspondence
Theorem, which is thus seen to be a common underlying principle for
(at least) both crossed products of $C^{*}$-algebras and generalized
Beurling algebras.

We will now briefly describe the contents of the paper.

In Section \ref{sec:Preliminaries-and-recapitulation} we review the
relevant definitions and results of \cite{2011arXiv1104.5151D}. In
Section \ref{sec:VaryingR} it is investigated how the crossed product
$\crossedprod$ depends on $\mathcal{R}$, and it is also shown that
there exists an isometric representation of this algebra on a Banach
space. The latter result is used in Section \ref{sec:Uniqueness-of-the-crossed-prod}.
Loosely speaking, $\crossedprod$ ``generates'' all non-degenerate
$\mathcal{R}$-continuous covariant representations of $(A,G,\alpha)$,
and under two mild additional hypotheses it is shown to be the unique
such algebra, up to isomorphism (cf.\,Theorem \ref{thm:universal-property}).
This result parallels work of Raeburn's \cite{RaeburnOriginalUniversalPaper}.
It is also shown (cf.\,Proposition \ref{prop:left-regular-is-topological-embedding})
that the left regular representation of $\crossedprod$ is a topological
embedding into its left centralizer algebra $\leftcent(\crossedprod)$.
Since $\crossedprod$ need not have a bounded approximate right identity,
this is not automatic. 

Next, in Section \ref{sec:Applications_L1_and_Beurling_algebras}
the generalized Beurling algebras $L^{1}(A,G,\omega;\alpha)$ make
their appearance. These algebras can be defined for any Banach algebra
dynamical system $(A,G,\alpha)$ and weight $\omega$ on $G$, provided
that $\alpha$ is uniformly bounded. If $A$ has a bounded approximate
right identity, then it can be shown that $L^{1}(A,G,\omega;\alpha)$
is isomorphic to $\crossedprod$, for a suitably chosen class $\mathcal{R}$
(cf.\,Theorem \ref{thm:Choosing-R-correctly-Crossed-Products-are-beurling}).
Via this isomorphism the General Correspondence Theorem therefore
predicts, if $A$ has a bounded two-sided approximate identity, what
the non-degenerate bounded representations of $L^{1}(A,G,\omega;\alpha)$
are, in terms of the non-degenerate continuous covariant representations
of $(A,G,\alpha)$ (cf.\,Theorem \ref{thm:continuous-non-deg-covars-are-R-continuous}),
and some classical results are thus seen to be obtainable from the
General Correspondence Theorem. As the easiest example, we retrieve
the usual description of the non-degenerate left $L^{1}(G)$-modules
in terms of the uniformly bounded strongly continuous representations
of $G$. Naturally, there is a similar description of the non-degenerate
right $L^{1}(G)$-modules, but an intermediate procedure is needed
to obtain such a result from the General Correspondence Theorem, where
one always ends up with left modules over the crossed product. This
is taken up in Section \ref{sec:Other-types}, where we investigate
all ``reasonable'' variations on the theme that $\pi:A\to B(X)$ and
$U:G\to B(X)$ should be multiplicative, and that $U_{r}\pi(a)U_{r}^{-1}=\pi(\alpha_{r}(a))$
should hold for all $a\in A$ and $r\in G$. We argue that there are
only three more ``reasonable'' requirements (cf.\,Table \ref{tab:table1}).
One of these is, e.g., that $\pi$ and $U$ are anti-multiplicative
and that $U_{r}\pi(a)U_{r}^{-1}=\pi(\alpha_{r^{-1}}(a))$ for all
$a\in A$ and $r\in G$; for $A=\mathbb{K}$ and $\alpha=\textup{triv}$
this covers the case of right $G$-modules. Moreover, we show that
a pair $(\pi,U)$ of each of the other three types can be reinterpreted
as a covariant representation in the usual sense for a suitable ``companion''
Banach algebra dynamical system. The example $(\pi,U)$ given above,
where there are three ``flaws'' in the properties of $(\pi,U)$, is
a covariant representation for the opposite Banach algebra dynamical
system $(A^{o},G^{o},\alpha^{o})$. Therefore, if one seeks a Banach
algebra of which the non-degenerate bounded (multiplicative) representations
``encode'' a family of such pairs $(\pi,U)$, then a crossed product
of type $\crossedprod$ is not what one should look at, but $(A^{o}\rtimes_{\alpha^{o}}G^{o})^{\mathcal{R}^{o}}$
is to be considered.

Section \ref{sec:Several-BADS} shows, as a particular case of Theorem
\ref{thm:sun-product-of-tensor-crossed-products}, how the encoding
for various types can be collected in one Banach algebra. For example,
the non-degenerate bounded representations of $\crossedprod\hat{\otimes}(A^{o}\rtimes_{\alpha^{o}}G^{o})^{\mathcal{R}^{o}}$
correspond to commuting non-degenerate bounded representations of
$\crossedprod$ and $(A^{o}\rtimes_{\alpha^{o}}G^{o})^{\mathcal{R}^{o}}$.
These representations can then be respectively related to a usual
covariant representation of $(A,G,\alpha)$ and a thrice ``flawed''
pair $(\pi,U)$ as above, which again commute.

In the final Section \ref{sec:Beurling-Right-and-bimodules} we combine
the results from Sections \ref{sec:Applications_L1_and_Beurling_algebras},
\ref{sec:Other-types} and \ref{sec:Several-BADS}. Using the procedure
from Section \ref{sec:Other-types} and the results from Section \ref{sec:Applications_L1_and_Beurling_algebras},
the relation between thrice ``flawed'' $(\pi,U)$ as above and the
non-degenerate bounded representations of $L^{1}(G^{o},A^{o},\omega^{o};\alpha^{o})$
is easily established. Since coincidentally $L^{1}(G^{o},A^{o},\omega^{o};\alpha^{o})$
turns out to be anti-isomorphic to $L^{1}(A,G,\omega;\alpha)$, these
pairs $(\pi,U)$ can then also be related to the non-degenerate right
$L^{1}(A,G,\omega;\alpha)$-modules (cf.\,Theorem \ref{thm:anti-correspondence}).
It is then easy to describe the simultaneous left $L^{1}(A,G,\omega;\alpha)$--
and right $L^{1}(B,H,\eta;\beta)$-modules, where the actions commute
(cf.\,Theorem \ref{thm:Beurling-bimodules}). In particular this
describes the bimodules over a generalized Beurling algebra $L^{1}(A,G,\omega;\alpha)$.
Specializing to the case where $A$ equals the scalars yields a description
of the non-degenerate bimodules over an ordinary Beurling algebra
$L^{1}(G,\omega)$ in terms of $G$-bimodules. Specializing still
further to $\omega=1$ the classical description of the non-degenerate
$L^{1}(G)$-bimodules in terms of a uniformly bounded $G$-bimodule
is retrieved as the simplest case in the general picture.

\section{\label{sec:Preliminaries-and-recapitulation}Preliminaries and recapitulation}

For the sake of self-containment we provide a brief recapitulation
of definitions and results from earlier papers \cite{2009arXiv0904.3268D,2011arXiv1104.5151D}. 

Throughout this paper $X$ and $Y$ will denote Banach spaces. The
algebra of bounded linear operators on $X$ will be denoted by $B(X)$.
By $A$ and $B$ we will denote Banach algebras, not necessarily unital,
and by $G$ and $H$ locally compact groups (which are always assumed
to be Hausdorff). We will always use the same symbol $\lambda$ to
denote the left regular representation of various Banach algebras
instead of distinguishing between them, as the context will always
make precise what is meant. If $A$ is a Banach algebra, $X$ a Banach
space, and $\pi:A\to B(X)$ is a Banach algebra representation, when
confusion could arise, we will write $X_{\pi}$ instead of $X$ to
make clear that the Banach space $X$ is related to the representation
$\pi$. We do not assume that Banach algebra representations of unital
Banach algebras are unital. Representations of algebras and groups
are always multiplicative (so that we are considering left modules),
unless explicitly stated otherwise.

Let $A$ be a Banach algebra, $G$ a locally compact Hausdorff group
and $\alpha:G\to\mbox{Aut}(A)$ a strongly continuous representation
of $G$ on $A$. Then the triple $(A,G,\alpha)$ is called a \emph{Banach
algebra dynamical system.} 

Let $(A,G,\alpha)$ be a Banach algebra dynamical system, $X$ a Banach
space with $\pi:A\to B(X)$ and $U:G\to B(X)$ representations of
the algebra $A$ and group $G$ on $X$ respectively. If $(\pi,U)$
satisfies 
\[
\pi(\alpha_{s}(a))=U_{s}\pi(a)U_{s}^{-1},
\]
for all $a\in A$ and $s\in G$, the pair $(\pi,U)$ is called a \emph{covariant
representation} of $(A,G,\alpha)$ on $X$. The pair $(\pi,U)$ is
said to be \emph{continuous} if $\pi$ is norm-bounded and $U$ is
strongly continuous. The pair $(\pi,U)$ is called \emph{non-degenerate}
if $\pi$ is non-degenerate (i.e., the span of $\pi(A)X$ lies dense
in $X$).

Integrals of compactly supported continuous Banach space valued functions
are, as in \cite{2011arXiv1104.5151D}, defined by duality, following
\cite[Section 3]{Rudin}. Let $C_{c}(G,A)$ denote the space of all
continuous compactly supported $A$-valued functions. For any $f,g\in C_{c}(G,A)$
and $s\in G$ defining the twisted convolution 
\[
[f*g](s):=\int_{G}f(r)\alpha_{r}(g(r^{-1}s))\, dr
\]
gives $C_{c}(G,A)$ the structure of an associative algebra, where
integration is with respect to a fixed left Haar measure on $G$. 

If $(\pi,U)$ is a continuous covariant representation of $(A,G,\alpha)$
on $X$, then, for $f\in C_{c}(G,A)$, we define $\pi\rtimes U(f)\in B(X)$,
as in \cite[Section 3]{2011arXiv1104.5151D}, by 

\[
\pi\rtimes U(f)x:=\int_{G}\pi(f(s))U_{s}x\, ds\quad(x\in X).
\]
The map $\pi\rtimes U:C_{c}(G,A)\to B(X)$ is a representation of
the algebra $C_{c}(G,A)$ on $X$, and is called the \emph{integrated
form }of $(\pi,U)$. 

Let $\mathcal{R}$ be a class of covariant representations of $(A,G,\alpha)$.
Then $\mathcal{R}$ is called a\emph{ uniformly bounded class of continuous
covariant representations} if there exist a constant $C\geq0$ and
function $\nu:G\to[0,\infty)$ which is bounded on compact sets, such
that, for any $(\pi,U)\in\mathcal{R}$, we have that $\|\pi\|\leq C$
and $\|U_{r}\|\leq\nu(r)$ for all $r\in G$. We will always tacitly
assume that such a class $\mathcal{R}$ is non-empty. With $\mathcal{R}$
as such, it follows that $\|\pi\rtimes U(f)\|\leq C\left(\sup_{r\in\mbox{supp}(f)}\nu(r)\right)\|f\|_{1}$
for all $(\pi,U)\in\mathcal{R}$ and $f\in C_{c}(G,A)$ \cite[Remark 3.3]{2011arXiv1104.5151D}. 

We define the algebra seminorm $\sigma^{\mathcal{R}}$ on $C_{c}(G,A)$
by
\[
\sigma^{\mathcal{R}}(f):=\sup_{(\pi,U)\in\mathcal{R}}\|\pi\rtimes U(f)\|\quad(f\in C_{c}(G,A)),
\]
and denote the completion of the quotient $C_{c}(G,A)/\ker\sigma^{\mathcal{R}}$
by $\crossedprod$, with $\|\cdot\|^{\mathcal{R}}$ denoting the norm
induced by $\sigma^{\mathcal{R}}$. The Banach algebra $\crossedprod$
is called the \emph{crossed product} corresponding to $(A,G,\alpha)$
and $\mathcal{R}$. The quotient homomorphism is denoted by $q^{\mathcal{R}}:C_{c}(G,A)\to\crossedprod$.

A covariant representation of $(A,G,\alpha)$ is called \emph{$\mathcal{R}$-continuous
}if it is continuous and its integrated form is bounded with respect
to the seminorm $\sigma^{\mathcal{R}}$. For any Banach space $X$
and linear map $T:C_{c}(G,A)\to X$, if $T$ is bounded with respect
to the $\sigma^{\mathcal{R}}$ seminorm, we will denote the canonically
induced linear map on $\crossedprod$ by $T^{\mathcal{R}}:\crossedprod\to X$,
as detailed in \cite[Section 3]{2011arXiv1104.5151D}.

If $A$ has a bounded approximate left (right) identity, then it can
be shown that $\crossedprod$ also has a bounded approximate left
(right) identity, with estimates for its bound, \cite[Theorem 4.4 and Corollary 4.6]{2011arXiv1104.5151D}. 

We will denote the left centralizer algebra of a Banach algebra $B$
by $\leftcent(B)$. Assuming $B$ has a bounded approximate left identity
$(u_{i})$, any non-degenerate bounded representation $T:B\to B(X)$
induces a non-degenerate bounded representation $\overline{T}:\leftcent(B)\to B(X)$,
by defining $\overline{T}(L):=\mbox{SOT-lim}_{i}T(Lu_{i})$ for all
$L\in\leftcent(B)$, so that the following diagram commutes (cf. \cite[Theorem 4.1]{2009arXiv0904.3268D}):

\[
\xymatrix{B\ar[r]^{T}\ar[dr]^{\lambda} & B(X)\\
 & \leftcent(B)\ar[u]_{\overline{T}}
}
\]
Moreover, $\overline{T}(L)T(a)=T(La)$ for all $a\in B$ and $L\in\leftcent(B)$.
We will often use this fact. 

With $(A,G,\alpha)$ a Banach algebra dynamical system and $\mathcal{R}$
a uniformly bounded class of continuous covariant representations,
we define the homomorphisms $i_{A}:A\to\mbox{End}(C_{c}(G,A))$ and
$i_{G}:G\to\mbox{End}(C_{c}(G,A))$ by 
\[
\begin{array}{ccl}
(i_{A}(a)f)(s) & := & af(s),\\
(i_{G}(r)f)(s) & := & \alpha_{r}(f(r^{-1}s)),
\end{array}
\]
for all $a\in A$, $f\in C_{c}(G,A)$ and $r,s\in G$. For each $a\in A$
and $r\in G$, the maps 
\[
i_{A}(a),i_{G}(r):(C_{c}(G,A),\sigma^{\mathcal{R}})\to(C_{c}(G,A),\sigma^{\mathcal{R}})
\]
are bounded \cite[Lemma 6.3]{2011arXiv1104.5151D}, and 
\[
\begin{array}{ccl}
\|i_{A}(a)\|^{\mathcal{R}} & \leq & \sup_{(\pi,U)\in\mathcal{R}}\|\pi(a)\|,\\
\|i_{G}(r)\|^{\mathcal{R}} & \leq & \sup_{(\pi,U)\in\mathcal{R}}\|U_{r}\|.
\end{array}
\]
Defining $i_{A}^{\mathcal{R}}(a)q^{\mathcal{R}}(f):=q^{\mathcal{R}}(i_{A}(a)f)$
and $i_{G}^{\mathcal{R}}(r)q^{\mathcal{R}}(f):=q^{\mathcal{R}}(i_{G}(r)f)$
for all $a\in A$, $r\in G$ and $\in C_{c}(G,A)$, we obtain bounded
maps 
\[
i_{A}^{\mathcal{R}}(a),i_{G}^{\mathcal{R}}(r):\crossedprod\to\crossedprod.
\]
Moreover, the maps $a\mapsto i_{A}^{\mathcal{R}}(a)$ and $r\mapsto i_{G}^{\mathcal{R}}(r)$
map $A$ and $G$ into $\leftcent(\crossedprod)$. If $A$ has a bounded
approximate left  identity and $\mathcal{R}$ is a uniformly bounded
class of non-degenerate continuous covariant representations, then
$(i_{A}^{\mathcal{R}},i_{G}^{\mathcal{R}})$ is a non-degenerate $\mathcal{R}$-continuous
covariant representation of $(A,G,\alpha)$ on $\crossedprod$ \cite[Section 6]{2011arXiv1104.5151D}
and the integrated form $(i_{A}^{\mathcal{R}}\rtimes i_{G}^{\mathcal{R}})^{\mathcal{R}}$
equals the left regular representation of $\crossedprod$ \cite[Theorem 7.2]{2011arXiv1104.5151D}. 

The main theorem from \cite{2011arXiv1104.5151D} establishes, amongst
others, a bijective relationship between the non-degenerate $\mathcal{R}$-continuous
covariant representations of $(A,G,\alpha)$ and the non-degenerate
bounded representations of $\crossedprod$, by letting $(\pi,U)$
correspond to $(\pi\rtimes U)^{\mathcal{R}}$. This result will play
a fundamental role throughout the rest of this paper, and the relevant
part of \cite[Theorem 8.1]{2011arXiv1104.5151D} can be stated as
follows:
\begin{thm}
\textup{(}General Correspondence Theorem, cf.\,\cite[Theorem 8.1]{2011arXiv1104.5151D}\textup{)}
\label{thm:General-Correspondence-Theorem}Let $(A,G,\alpha)$ be
a Banach algebra dynamical system, where $A$ has a bounded approximate
left identity. Let $\mathcal{R}$ be a uniformly bounded class of
non-degenerate continuous covariant representations of $(A,G,\alpha)$.
Then the map $(\pi,U)\mapsto(\pi\rtimes U)^{\mathcal{R}}$ is a bijection
between the non-degenerate $\mathcal{R}$-continuous covariant representations
of $(A,G,\alpha)$ and the non-degenerate bounded representations
of $\crossedprod$.

More precisely:
\begin{enumerate}
\item If $(\pi,U)$ is a non-degenerate $\mathcal{R}$-continuous covariant
representation of $(A,G,\alpha)$, then $(\pi\rtimes U)^{\mathcal{R}}$
is a non-degenerate bounded representation of $\crossedprod$, and
\[
(\overline{(\pi\rtimes U)^{\mathcal{R}}}\circ i_{A}^{\mathcal{R}},\overline{(\pi\rtimes U)^{\mathcal{R}}}\circ i_{G}^{\mathcal{R}})=(\pi,U),
\]
where $\overline{(\pi\rtimes U)^{\mathcal{R}}}$ is the representation
of \textup{$\leftcent(\crossedprod)$ }as described above, cf. \cite[Section 7]{2011arXiv1104.5151D}\textup{.} 
\item If $T$ is a non-degenerate bounded representation of $\crossedprod$,
then $(\overline{T}\circ i_{A}^{\mathcal{R}},\overline{T}\circ i_{G}^{\mathcal{R}})$
is a non-degenerate $\mathcal{R}$-continuous covariant representation
of $(A,G,\alpha)$, and
\[
(\overline{T}\circ i_{A}^{\mathcal{R}}\rtimes\overline{T}\circ i_{G}^{\mathcal{R}})^{\mathcal{R}}=T.
\]

\end{enumerate}
\end{thm}

\section{Varying $\mathcal{R}$\label{sec:VaryingR}}

For a given Banach algebra dynamical system $(A,G,\alpha)$, one may
ask what relationship exists between the crossed products $(A\rtimes_{\alpha}G)^{\mathcal{R}_{1}}$
and $(A\rtimes_{\alpha}G)^{\mathcal{R}_{2}}$ for two uniformly bounded
classes $\mathcal{R}_{1}$ and $\mathcal{R}_{2}$ of possibly degenerate
continuous covariant representations on Banach spaces. This section
investigates this question.

Since uniformly bounded classes of covariant representations might
be proper classes, we must take some care in working with them. Nevertheless,
we can always choose a set from a uniformly bounded class $\mathcal{R}$
of covariant representations of a Banach algebra dynamical system
$(A,G,\alpha)$ so that this set determines $\sigma^{\mathcal{R}}$.
Indeed for every $f\in C_{c}(G,A)$, looking at the subset $\{\|\pi\rtimes U(f)\|:(\pi,U)\in\mathcal{R}\}$
of $\mathbb{R}$ (subclasses of sets are sets), we may choose a sequence
from $\{\|\pi\rtimes U(f)\|:(\pi,U)\in\mathcal{R}\}$ converging to
$\sigma^{\mathcal{R}}(f)$ and regard only those corresponding covariant
representations. In this way, we can chose a set $S$ from $\mathcal{R}$
of cardinality at most $|C_{c}(G,A)\times\mathbb{N}|$ such that $\sigma^{S}(f)=\sup_{(\pi,U)\in S}\|\pi\rtimes U(f)\|=\sigma^{\mathcal{R}}(f)$
for all $f\in C_{c}(G,A)$. Hence the following definition is meaningful;
it will be required in Definition \ref{def:lp-direct-sum-realization}
and Proposition \ref{thm:isometric_realization}.
\begin{defn}
\label{def:equivalent_uniformly_bounded_sets}Let $\mathcal{R}$ be
a uniformly bounded class of possibly degenerate continuous covariant
representations of $(A,G,\alpha)$. We define $[\mathcal{R}]$ to
be the collection of all uniformly bounded classes $\mathcal{S}$
that are actually sets and satisfy $\sigma^{\mathcal{R}}=\sigma^{\mathcal{S}}$
on $C_{c}(G,A)$. Elements of some $[\mathcal{R}]$ will be called
\emph{uniformly bounded sets of continuous covariant representations.}
\end{defn}
Before addressing the question laid out in the first paragraph, we
consider the following aside which will play a key role in Section
\ref{sec:Uniqueness-of-the-crossed-prod}.
\begin{defn}
Let $I$ be an index set and $\{X_{i}:i\in I\}$ a family of Banach
spaces. For $1\leq p\leq\infty$, we will denote the $\ell^{p}$-direct
sum of $\{X_{i}:i\in I\}$ by $\ell^{p}\{X_{i}:i\in I\}$.
\end{defn}
\
\begin{defn}
\label{def:lp-direct-sum-realization}Let $(A,G,\alpha)$ be a Banach
algebra dynamical system and $\mathcal{R}$ a uniformly bounded class
of continuous covariant representations. For $S\in[\mathcal{R}]$
and $1\leq p<\infty$, suppressing the $p$-dependence in the notation,
we define the representations $(\oplus_{S}\pi):A\to B(\ell^{p}\{X_{\pi}:(\pi,U)\in S\})$
and $(\oplus_{S}U):G\to B(\ell^{p}\{X_{\pi}:(\pi,U)\in S\})$ by $(\oplus_{S}\pi)(a):=\bigoplus_{(\pi,U)\in S}\pi(a)$
and $(\oplus_{S}U)_{r}:=\bigoplus_{(\pi,U)\in S}U_{r}$ for all $a\in A$
and $r\in G$ respectively. 
\end{defn}
It is easily seen that $((\oplus_{S}\pi),(\oplus_{S}U))$ is a continuous
covariant representation, that 
\[
((\oplus_{S}\pi)\rtimes(\oplus_{S}U))(f)=\bigoplus_{(\pi,U)\in S}\pi\rtimes U(f),
\]
and that $\|((\oplus_{S}\pi)\rtimes(\oplus_{S}U))(f)\|=\sigma^{S}(f)=\sigma^{\mathcal{R}}(f)$,
for all $f\in C_{c}(G,A)$. 

We hence obtain the following (where the statement concerning non-degeneracy
is an elementary verification).
\begin{prop}
\label{thm:isometric_realization}Let $(A,G,\alpha)$ be a Banach
algebra dynamical system and $\mathcal{R}$ a uniformly bounded class
of continuous covariant representations. For any $S\in[\mathcal{R}]$
and $1\leq p<\infty,$ there exists an $\mathcal{R}$-continuous covariant
representation of $(A,G,\alpha)$ on $\ell^{p}\{X_{\pi}:(\pi,U)\in S\}$,
denoted $((\oplus_{S}\pi),(\oplus_{S}U))$, such that its integrated
form satisfies $\|((\oplus_{S}\pi)\rtimes(\oplus_{S}U))(f)\|=\sigma^{\mathcal{R}}(f)$
for all $f\in C_{c}(G,A)$ and hence induces an isometric representation
of $\crossedprod$ on $\ell^{p}\{X_{\pi}:(\pi,U)\in S\}$.

If every element of $S$ is non-degenerate, then $((\oplus_{S}\pi),(\oplus_{S}U))$
is non-degenerate.
\end{prop}
The previous theorem shows, in particular, that crossed products can
always be realized isometrically as closed subalgebras of bounded
operators on some (rather large) Banach space. 

We now return to the original question. The following results examine
relations that may exist between crossed products defined by using
two different uniformly bounded classes of continuous covariant representations
of a Banach algebra dynamical system.
\begin{prop}
\label{thm:R-relations}Let $(A,G,\alpha)$ be a Banach algebra dynamical
system. Let $\mathcal{R}_{1}$ and $\mathcal{R}_{2}$ be uniformly
bounded classes of possibly degenerate continuous covariant representations
of $(A,G,\alpha)$ and $M\geq1$ a constant. Then the following are
equivalent:
\begin{enumerate}
\item There exists a homomorphism $h:(A\rtimes_{\alpha}G)^{\mathcal{R}_{2}}\to(A\rtimes_{\alpha}G)^{\mathcal{R}_{1}}$
such that $\|h\|\leq M$ and $h\circ q^{\mathcal{R}_{2}}(f)=q^{\mathcal{R}_{1}}(f)$
for all $f\in C_{c}(G,A)$.
\item The seminorms $\sigma^{\mathcal{R}_{1}}$ and $\sigma^{\mathcal{R}_{2}}$
satisfy $\sigma^{\mathcal{R}_{1}}(f)\leq M\sigma^{\mathcal{R}_{2}}(f)$
for all $f\in C_{c}(G,A)$.
\item There exist uniformly bounded sets of continuous covariant representations
$\mathcal{R}_{1}'\in[\mathcal{R}_{1}]$, $\mathcal{R}_{2}'\in[\mathcal{R}_{2}]$
and $\mathcal{R}_{3}'$ such that $\mathcal{R}_{1}'\cup\mathcal{R}_{2}'\subseteq\mathcal{R}_{3}'$
and $\sigma^{\mathcal{R}_{2}'}(f)\leq\sigma^{\mathcal{R}_{3}'}(f)\leq M\sigma^{\mathcal{R}_{2}'}(f)$
for all $f\in C_{c}(G,A)$.
\item If $(\pi,U)$ is an $\mathcal{R}_{1}$-continuous covariant representation
of $(A,G,\alpha)$ and $M'\geq0$ is such that $\|\pi\rtimes U(f)\|\leq M'\sigma^{\mathcal{R}_{1}}(f)$
for all $f\in C_{c}(G,A)$, then $(\pi,U)$ is an $\mathcal{R}_{2}$-continuous
covariant representation of $(A,G,\alpha)$, and $\|\pi\rtimes U(f)\|\leq M'M\sigma^{\mathcal{R}_{2}}(f)$
for all $f\in C_{c}(G,A)$.
\item For any bounded representation $T:(A\rtimes_{\alpha}G)^{\mathcal{R}_{1}}\to B(X)$
there exists a bounded representation $S:(A\rtimes_{\alpha}G)^{\mathcal{R}_{2}}\to B(X)$
such that $T\circ q^{\mathcal{R}_{1}}(f)=S\circ q^{\mathcal{R}_{2}}(f)$
for all $f\in C_{c}(G,A)$ and $\|S\|\leq M\|T\|$. 
\end{enumerate}
\end{prop}
\begin{proof}
We prove that (1) implies (5). Let $T:(A\rtimes_{\alpha}G)^{\mathcal{R}_{1}}\to B(X)$
be a bounded representation. Then $S:=T\circ h:(A\rtimes_{\alpha}G)^{\mathcal{R}_{2}}\to B(X)$
satisfies $T\circ q^{\mathcal{R}_{1}}(f)=T\circ h\circ q^{\mathcal{R}_{2}}(f)=S\circ q^{\mathcal{R}_{2}}(f)$
for all $f\in C_{c}(G,A)$, and $\|S\|\leq\|T\|\|h\|\leq M\|T\|$.

We prove that (5) implies (4). Let $(\pi,U)$ be $\mathcal{R}_{1}$-continuous
and $M'\geq0$ be such that $\|\pi\rtimes U(f)\|\leq M'\sigma^{\mathcal{R}_{1}}(f)$
for all $f\in C_{c}(G,A)$. Then, for the bounded representation $(\pi\rtimes U)^{\mathcal{R}_{1}}:(A\rtimes_{\alpha}G)^{\mathcal{R}_{1}}\to B(X_{\pi})$,
there exists a bounded representation $S:(A\rtimes_{\alpha}G)^{\mathcal{R}_{2}}\to B(X_{\pi})$
such that 
\[
\pi\rtimes U(f)=(\pi\rtimes U)^{\mathcal{R}_{1}}\circ q^{\mathcal{R}_{1}}(f)=S\circ q^{\mathcal{R}_{2}}(f)
\]
for all $f\in C_{c}(G,A)$, and $\|S\|\leq M\|(\pi\rtimes U)^{\mathcal{R}_{1}}\|\leq MM'.$
Hence, $(\pi,U)$ is $\mathcal{R}_{2}$-continuous, and $\|\pi\rtimes U(f)\|=\|S\circ q^{\mathcal{R}_{2}}(f)\|\leq MM'\sigma^{\mathcal{R}_{2}}(f)$
holds for all $f\in C_{c}(G,A)$.

We prove that (4) implies (2). Every $(\pi,U)\in\mathcal{R}_{1}$
is $\mathcal{R}_{1}$-continuous and satisfies $\|\pi\rtimes U(f)\|\leq\sigma^{\mathcal{R}_{1}}(f)$
for all $f\in C_{c}(G,A)$. Then, by hypothesis, $(\pi,U)$ is $\mathcal{R}_{2}$-continuous
and 
\[
\|\pi\rtimes U(f)\|\leq M\sigma^{\mathcal{R}_{2}}(f)
\]
 for all $f\in C_{c}(G,A)$. Taking the supremum over all $(\pi,U)\in\mathcal{R}_{1}$,
we obtain $\sigma^{\mathcal{R}_{1}}(f)\leq M\sigma^{\mathcal{R}_{2}}(f)$
for all $f\in C_{c}(G,A)$.

We prove that (2) implies (1). Since $\ker\sigma^{\mathcal{R}_{2}}\subseteq\ker\sigma^{\mathcal{R}_{1}}$,
a homomorphism 
\[
h:C_{c}(G,A)/\ker\sigma^{\mathcal{R}_{2}}\to C_{c}(G,A)/\ker\sigma^{\mathcal{R}_{1}}
\]
 can be defined by $h(q^{\mathcal{R}_{2}}(f)):=q^{\mathcal{R}_{1}}(f)$
for all $f\in C_{c}(G,A)$, and then satisfies $\|h\|\leq M$. The
map $h$ therefore extends to a homomorphism $h:(A\rtimes_{\alpha}G)^{\mathcal{R}_{2}}\to(A\rtimes_{\alpha}G)^{\mathcal{R}_{1}}$
with the same norm.

We prove that (2) implies (3). Let $\mathcal{R}_{1}'\in[\mathcal{R}_{1}]$
and $\mathcal{R}_{2}'\in[\mathcal{R}_{2}]$ and define $\mathcal{R}_{3}':=\mathcal{R}_{1}'\cup\mathcal{R}_{2}'$.
By construction we have that $\sigma^{\mathcal{R}_{2}'}(f)\leq\sigma^{\mathcal{R}_{3}'}(f)$
for all $f\in C_{c}(G,A)$. By hypothesis we have that $\sigma^{\mathcal{R}_{1}'}(f)\leq M\sigma^{\mathcal{R}_{2}'}(f)$
for all $f\in C_{c}(G,A)$, as well as $M\geq1$. Therefore, 
\[
\sigma^{\mathcal{R}_{2}'}(f)\leq\sigma^{\mathcal{R}_{3}'}(f)=\max\{\sigma^{\mathcal{R}_{1}'}(f),\sigma^{\mathcal{R}_{2}'}(f)\}\leq\max\{M\sigma^{\mathcal{R}_{2}'}(f),\sigma^{\mathcal{R}_{2}'}(f)\}=M\sigma^{\mathcal{R}_{2}'}(f).
\]

We prove that (3) implies (2). Let $\mathcal{R}_{1}'\in[\mathcal{R}_{1}]$,
$\mathcal{R}_{2}'\in[\mathcal{R}_{2}]$ and $\mathcal{R}_{3}'$ be
such that $\mathcal{R}_{1}'\cup\mathcal{R}_{2}'\subseteq\mathcal{R}_{3}'$
and $\sigma^{\mathcal{R}_{2}'}(f)\leq\sigma^{\mathcal{R}_{3}'}(f)\leq M\sigma^{\mathcal{R}_{2}'}(f)$
for all $f\in C_{c}(G,A)$. Then 
\[
\sigma^{\mathcal{R}_{1}}(f)=\sigma^{\mathcal{R}_{1}'}(f)\leq\sigma^{\mathcal{R}_{3}'}(f)\leq M\sigma^{\mathcal{R}_{2}'}(f)\leq M\sigma^{\mathcal{R}_{2}}(f).
\]
\end{proof}
We can now describe the relationship between $\mathcal{R}$ and the
isomorphism class of the pair $(\crossedprod,q^{\mathcal{R}})$.
\begin{cor}
Let $(A,G,\alpha)$ be a Banach algebra dynamical system and $\mathcal{R}_{1}$
and $\mathcal{R}_{2}$ be uniformly bounded classes of possibly degenerate
continuous covariant representations of $(A,G,\alpha)$. Then the
following are equivalent:
\begin{enumerate}
\item There exists a topological algebra isomorphism $h:(A\rtimes_{\alpha}G)^{\mathcal{R}_{1}}\to(A\rtimes_{\alpha}G)^{\mathcal{R}_{2}}$
such that the following diagram commutes: 
\[
\xymatrix{ & (A\rtimes_{\alpha}G)^{\mathcal{R}_{1}}\ar[dd]^{h}\\
C_{c}(G,A)\ar[ur]^{q^{\mathcal{R}_{1}}}\ar[dr]_{q^{\mathcal{R}_{2}}}\\
 & (A\rtimes_{\alpha}G)^{\mathcal{R}_{2}}
}
\]

\item The seminorms $\sigma^{\mathcal{R}_{1}}$ and $\sigma^{\mathcal{R}_{2}}$
on $C_{c}(G,A)$ are equivalent.
\item There exist uniformly bounded sets of possibly degenerate continuous
covariant representations $\mathcal{R}_{1}'\in[\mathcal{R}_{1}]$,
$\mathcal{R}_{2}'\in[\mathcal{R}_{2}]$ and $\mathcal{R}_{3}'$ with
$\mathcal{R}_{1}'\cup\mathcal{R}_{2}'\subseteq\mathcal{R}_{3}'$ and
constants $M_{1},M_{2}\geq0$, such that
\[
\begin{array}{c}
\sigma^{\mathcal{R}_{1}'}(f)\leq\sigma^{\mathcal{R}_{3}'}(f)\leq M_{1}\sigma^{\mathcal{R}_{1}'}(f),\\
\sigma^{\mathcal{R}_{2}'}(f)\leq\sigma^{\mathcal{R}_{3}'}(f)\leq M_{2}\sigma^{\mathcal{R}_{2}'}(f),
\end{array}
\]
 for all $f\in C_{c}(G,A)$.
\item The $\mathcal{R}_{1}$-continuous covariant representations of $(A,G,\alpha)$
coincide with the $\mathcal{R}_{2}$-continuous covariant representations
of $(A,G,\alpha)$. Moreover, there exist constants $M_{1},M_{2}\geq0$,
with the property that, if $M'\geq0$ and $(\pi,U)$ is $\mathcal{R}_{1}$-continuous,
such that $\|\pi\rtimes U(f)\|\leq M'\sigma^{\mathcal{R}_{1}}(f)$
for all $f\in C_{c}(G,A)$, then $\|\pi\rtimes U(f)\|\leq M_{1}M'\sigma^{\mathcal{R}_{2}}(f)$
for all $f\in C_{c}(G,A)$, and likewise for the indices $1$ and
$2$ interchanged.
\item There exist constants $M_{1},M_{2}\geq0$ with the property that,
for every bounded representation $T:(A\rtimes_{\alpha}G)^{\mathcal{R}_{1}}\to B(X)$
there exists a bounded representation $S:(A\rtimes_{\alpha}G)^{\mathcal{R}_{2}}\to B(X)$
with $\|S\|\leq M_{1}\|T\|$, such that the diagram
\[
\xymatrix{ & (A\rtimes_{\alpha}G)^{\mathcal{R}_{1}}\ar[dr]^{T}\\
C_{c}(G,A)\ar[ur]^{q^{\mathcal{R}_{1}}}\ar[dr]_{q^{\mathcal{R}_{2}}} &  & B(X)\\
 & (A\rtimes_{\alpha}G)^{\mathcal{R}_{2}}\ar[ur]_{S}
}
\]
commutes, and likewise with the indices $1$ and $2$ interchanged.
\end{enumerate}
\end{cor}
\begin{proof}
This follows from Proposition \ref{thm:R-relations}.\end{proof}

\section{Uniqueness of the crossed product\label{sec:Uniqueness-of-the-crossed-prod}}

Theorem \ref{thm:General-Correspondence-Theorem} asserts, amongst
others, that all non-degenerate $\mathcal{R}$-continuous covariant
representations of $(A,G,\alpha)$ can be generated from the non-degenerate
bounded representations of $\crossedprod$, with the aid of $\leftcent(\crossedprod)$
and the pair $(i_{A}^{\mathcal{R}},i_{G}^{\mathcal{R}})$. In this
section we show that, under mild additional hypotheses, $\crossedprod$
is the unique Banach algebra with this generating property. These
results are similar in nature as Raeburn's for the crossed product
of a $C^{*}$-algebra, see \cite{RaeburnOriginalUniversalPaper} or
\cite[Theorem 2.61]{Williams}.

We start with the general framework of how to generate many non-degenerate
$\mathcal{R}$-continuous covariant representations from a suitable
basic one, on a Banach space that is a Banach algebra.
\begin{lem}
\label{lem:existence_of_generating_pair}Let $(A,G,\alpha)$ be a
Banach algebra dynamical system, and let $\mathcal{R}$ be a uniformly
bounded class of continuous covariant representations of $(A,G,\alpha)$.
Let $C$ be a Banach algebra with a bounded approximate left identity,
and let $(k_{A},k_{G})$ be a non-degenerate $\mathcal{R}$-continuous
covariant representation of $(A,G,\alpha)$ on the Banach space $C$,
such that $k_{A}(A),k_{G}(G)\subseteq\leftcent(C)$. Suppose $T:C\to B(X)$
is a non-degenerate bounded representation of $C$ on a Banach space
$X$. Let $\overline{T}:\leftcent(C)\to B(X)$ be the non-degenerate
bounded representation of $\leftcent(C)$ such that the following
diagram commutes:
\[
\xymatrix{C\ar[dr]^{\lambda}\ar[r]^{T} & B(X)\\
 & \leftcent(C)\ar[u]_{\overline{T}}
}
\]
Then the pair $(\overline{T}\circ k_{A},\overline{T}\circ k_{G})$
is a non-degenerate $\mathcal{R}$-continuous covariant representation
of $(A,G,\alpha)$, and $(\overline{T}\circ k_{A})\rtimes(\overline{T}\circ k_{G})=\overline{T}\circ(k_{A}\rtimes k_{G})$.\end{lem}
\begin{proof}
It is clear that $\overline{T}\circ k_{A}$ is a continuous representation
of $A$ on $X$. Since $\overline{T}$ is unital \cite[Theorem 4.1]{2009arXiv0904.3268D},
$\overline{T}\circ k_{G}$ is a representation of $G$ on $X$. Using
that $\overline{T}(L)T(c)=T(Lc)$ for $L\in\leftcent(C)$ and $c\in C$,
(cf.\,\cite[Theorem 4.1]{2009arXiv0904.3268D}), we find, for $r\in G$,
$c\in C$ and $x\in X$, that $(\overline{T}\circ k_{G}(r))T(c)x=T(k_{G}(r)c)x$.
Since $k_{G}$ is strongly continuous and $T$ is continuous, we see
that 
\[
\lim_{r\to e}(\overline{T}\circ k_{G}(r))T(c)x=\lim_{r\to e}T(k_{G}(r)c)x=T(c)x,
\]
for all $c\in C$ and $x\in X$. The non-degeneracy of $T$, together
with \cite[Corollary 2.5]{2011arXiv1104.5151D} then imply that $T\circ k_{G}$
is strongly continuous. It is a routine verification that $(\overline{T}\circ k_{A},\overline{T}\circ k_{G})$
is covariant, so that $(\overline{T}\circ k_{A},\overline{T}\circ k_{G})$
is a continuous covariant representation of $(A,G,\alpha)$ on $C$.

We claim that $k_{A}\rtimes k_{G}:C_{c}(G,A)\to B(C)$ has its image
in $\leftcent(C)$, and that $(\overline{T}\circ k_{A})\rtimes(\overline{T}\circ k_{G})=\overline{T}\circ(k_{A}\rtimes k_{G})$.
The $\mathcal{R}$-continuity of $(k_{A},k_{G})$ and the continuity
of $\overline{T}$ then show that $(\overline{T}\circ k_{A},\overline{T}\circ k_{G})$
is $\mathcal{R}$-continuous. As to this claim, note that, for $f\in C_{c}(G,A)$,
the integrand in $k_{A}\rtimes k_{G}(f)=\int_{G}k_{A}(f(r))k_{G}(r)\, dr$
takes values in the SOT-closed subspace $\leftcent(C)$ of $B(C)$,
hence the integral is likewise in this subspace. Hence $\overline{T}\circ(k_{A}\rtimes k_{G}):C_{c}(G,A)\to B(X)$
is a meaningfully defined map. Using that that continuous operators
can be pulled through the integral \cite[Ch. 3, Exercise 24]{Rudin}
and the definition of operator valued integrals \cite[Proposition 2.19]{2011arXiv1104.5151D},
we then have for all $x\in X$:
\begin{eqnarray*}
\overline{T}\left(k_{A}\rtimes k_{G}(f)\right)T(c)x & = & T\left(k_{A}\rtimes k_{G}(f)c\right)x\\
 & = & T\left(\int_{G}k_{A}(f(r))k_{G}(r)\, dr\ c\right)x\\
 & = & T\left(\int_{G}k_{A}(f(r))k_{G}(r)c\, dr\right)x\\
 & = & \int_{G}T\left(k_{A}(f(r))k_{G}(r)c\right)x\, dr\\
 & = & \int_{G}\overline{T}\left(k_{A}(f(r))k_{G}(r)\right)T(c)x\, dr\\
 & = & \int_{G}\overline{T}\circ k_{A}(f(r))\overline{T}\circ k_{G}(r)T(c)x\, dr\\
 & = & \left(\int_{G}\overline{T}\circ k_{A}(f(r))\overline{T}\circ k_{G}(r)\, dr\right)T(c)x\\
 & = & \left((\overline{T}\circ k_{A})\rtimes(\overline{T}\circ k_{G})(f)\right)T(c)x.
\end{eqnarray*}
Since $T$ is non-degenerate, this establishes the claim.

It remains to show that $\overline{T}\circ k_{A}$ is non-degenerate.
Let $x\in X$ and $\varepsilon>0$ be arbitrary. Since $\overline{T}$
is non-degenerate \cite[Theorem 4.1]{2009arXiv0904.3268D}, there
exist finite sets $\{c_{i}\}_{i=1}^{n}\subseteq C$ and $\{x_{i}\}_{i=1}^{n}\subseteq X$
such that $\|\sum_{i=1}^{n}T(c_{i})x_{i}-x\|<\varepsilon/2$. Since
$k_{A}$ is non-degenerate, for every $i\in\{1,\ldots,n\}$, there
exist finite sets $\{a_{i,j}\}_{j=1}^{m_{i}}\subseteq A$ and $\{d_{i,j}\}_{j=1}^{m_{i}}\subseteq C$
such that $\|T\|\|x_{i}\|\|c_{i}-\sum_{j=1}^{m_{i}}k_{A}(a_{i,j})d_{i,j}\|<\varepsilon/2n$.
Then\\
\begin{eqnarray*}
 &  & \left\Vert x-\sum_{i=1}^{n}\sum_{j=1}^{m_{i}}\left(\overline{T}\circ k_{A}(a_{i,j})\right)T(d_{i,j})x_{i}\right\Vert \\
 & = & \left\Vert x-\sum_{i=1}^{n}\sum_{j=1}^{m_{i}}T\left(k_{A}(a_{i,j})d_{i,j}\right)x_{i}\right\Vert \\
 & \leq & \left\Vert x-\sum_{i=1}^{n}T(c_{i})x_{i}\right\Vert +\left\Vert \sum_{i=1}^{n}T(c_{i})x_{i}-\sum_{i=1}^{n}\sum_{j=1}^{m_{i}}T\left(k_{A}(a_{i,j})d_{i,j}\right)x_{i}\right\Vert \\
 & \leq & \left\Vert x-\sum_{i=1}^{n}T(c_{i})x_{i}\right\Vert +\sum_{i=1}^{n}\|T\|\left\Vert c_{i}-\sum_{j=1}^{m_{i}}k_{A}(a_{i,j})d_{i,j}\right\Vert \|x_{i}\|\\
 & < & \frac{\varepsilon}{2}+\frac{\varepsilon}{2}.
\end{eqnarray*}
We conclude that $\overline{T}\circ k_{A}$ is non-degenerate.\end{proof}
Naturally any Banach algebra $C'$ isomorphic to $C$ as in the previous
lemma has a similar ``generating pair'' $(k_{A}',k_{G}')$. The details
are in the following result, the routine verification of which is
left to the reader.
\begin{lem}
\label{lem:generating-pair-translation}Let $(A,G,\alpha)$, $\mathcal{R}$,
$C$ and $(k_{A},k_{G})$ be as in Lemma \ref{lem:existence_of_generating_pair}.
Suppose $C'$ is a Banach algebra and $\psi:C\to C'$ is a topological
isomorphism. Then:
\begin{enumerate}
\item $\psi_{l}:\leftcent(C)\to\leftcent(C')$, defined by $\psi_{l}(L):=\psi L\psi^{-1}$
for $L\in\leftcent(C)$, is a topological isomorphism.
\item The pair $(k_{A}',k_{G}'):=(\psi_{l}\circ k_{A},\psi_{l}\circ k_{G})$
is a non-degenerate $\mathcal{R}$-continuous covariant representation
of $(A,G,\alpha)$ on $C'$, such that $k_{A}'(A),k_{G}'(G)\subseteq\leftcent(C')$.
\item If $T:C\to B(X)$ is a non-degenerate bounded representation, then
so is $T':C'\to B(X)$, where $T':=T\circ\psi^{-1}$.
\item If $T:C\to B(X)$ is a non-degenerate bounded representation, and
$\overline{T'}:\leftcent(C')\to B(X)$ is the non-degenerate bounded
representation of $\leftcent(C')$ such that the diagram
\[
\xymatrix{C'\ar[dr]^{\lambda}\ar[r]^{T'} & B(X)\\
 & \leftcent(C')\ar[u]_{\overline{T'}}
}
\]
commutes, then $\overline{T}\circ k_{A}=\overline{T'}\circ k_{A}'$
and $\overline{T}\circ k_{G}=\overline{T'}\circ k_{G}'$.
\end{enumerate}
\end{lem}
Now let $\mathcal{R}$ be a uniformly bounded class of non-degenerate
continuous covariant representations of $(A,G,\alpha)$, where $A$
has a bounded approximate left identity. Then, according to \cite[Theorem 7.2]{2011arXiv1104.5151D},
$\crossedprod$ has a bounded approximate left identity, and the maps
$i_{A}^{\mathcal{R}}:A\to B(\crossedprod)$ and $i_{G}^{\mathcal{R}}:G\to B(\crossedprod)$
form a non-degenerate $\mathcal{R}$-continuous covariant representation
of $(A,G,\alpha)$ on the Banach space $\crossedprod$, with images
in $\leftcent(\crossedprod)$. According to Lemma \ref{lem:existence_of_generating_pair},
the triple $(\crossedprod,i_{A}^{\mathcal{R}},i_{G}^{\mathcal{R}})$
can be used to produce non-degenerate $\mathcal{R}$-continuous covariant
representations of $(A,G,\alpha)$ from non-degenerate bounded representations
of $\crossedprod$, and, according to Theorem \ref{thm:General-Correspondence-Theorem},
all non-degenerate $\mathcal{R}$-continuous covariant representations
are thus obtained. According to Lemma \ref{lem:generating-pair-translation},
any Banach algebra isomorphic to $\crossedprod$ has the same property.
We will now proceed to show the converse: If $(B,k_{A},k_{G})$ is
a triple generating all non-degenerate $\mathcal{R}$-continuous covariant
representations of $(A,G,\alpha)$, then it can be obtained from $(\crossedprod,i_{A}^{\mathcal{R}},i_{G}^{\mathcal{R}})$
as in Lemma \ref{lem:generating-pair-translation}.

We start with a preliminary observation that is of some interest in
its own right.
\begin{prop}
\label{prop:left-regular-is-topological-embedding}Let $(A,G,\alpha)$
be a Banach algebra dynamical system where $A$ has a bounded approximate
left identity. Let $\mathcal{R}$ be a uniformly bounded class of
non-degenerate continuous covariant representations. Then the left
regular representation $\lambda:\crossedprod\to\leftcent(\crossedprod)$
is a topological embedding.\end{prop}
\begin{proof}
According to Proposition \ref{thm:isometric_realization}, there exists
a non-degenerate $\mathcal{R}$-continuous covariant representation
$(\pi,U)$ such that $(\pi\rtimes U)^{\mathcal{R}}$ is a non-degenerate
isometric representation of $\crossedprod$. According to Theorem
\ref{thm:General-Correspondence-Theorem}, $\pi=\overline{(\pi\rtimes U)^{\mathcal{R}}}\circ i_{A}^{\mathcal{R}}$
and $U=\overline{(\pi\rtimes U)^{\mathcal{R}}}\circ i_{G}^{\mathcal{R}}$.
Furthermore, according to Lemma \ref{lem:existence_of_generating_pair},
\[
(\overline{(\pi\rtimes U)^{\mathcal{R}}}\circ i_{A}^{\mathcal{R}})\rtimes(\overline{(\pi\rtimes U)^{\mathcal{R}}}\circ i_{G}^{\mathcal{R}})=\overline{(\pi\rtimes U)^{\mathcal{R}}}\circ(i_{A}^{\mathcal{R}}\rtimes i_{G}^{\mathcal{R}}).
\]
We recall \cite[Theorem 7.2]{2011arXiv1104.5151D} that $(i_{A}^{\mathcal{R}}\rtimes i_{G}^{\mathcal{R}})^{\mathcal{R}}=\lambda$.
Combining all this, we see, with $M$ denoting an upper bound for
an approximate left identity of $\crossedprod$ \cite[Corollary 4.6]{2011arXiv1104.5151D},
that, for $f\in C_{c}(G,A)$:
\begin{eqnarray*}
\|q^{\mathcal{R}}(f)\|^{\mathcal{R}} & = & \|(\pi\rtimes U)^{\mathcal{R}}(q^{\mathcal{R}}(f))\|\\
 & = & \|\pi\rtimes U(f)\|\\
 & = & \|(\overline{(\pi\rtimes U)^{\mathcal{R}}}\circ i_{A}^{\mathcal{R}})\rtimes(\overline{(\pi\rtimes U)^{\mathcal{R}}}\circ i_{G}^{\mathcal{R}})(f)\|\\
 & = & \|\overline{(\pi\rtimes U)^{\mathcal{R}}}\circ(i_{A}^{\mathcal{R}}\rtimes i_{G}^{\mathcal{R}})(f)\|\\
 & = & \|\overline{(\pi\rtimes U)^{\mathcal{R}}}\circ(i_{A}^{\mathcal{R}}\rtimes i_{G}^{\mathcal{R}})^{\mathcal{R}}(q^{\mathcal{R}}(f))\|\\
 & = & \|\overline{(\pi\rtimes U)^{\mathcal{R}}}(\lambda(q^{\mathcal{R}}(f)))\|\\
 & \leq & M\|(\pi\rtimes U)^{\mathcal{R}}\|\|\lambda(q^{\mathcal{R}}(f))\|\\
 & = & M\|\lambda(q^{\mathcal{R}}(f))\|.
\end{eqnarray*}
Since the inequality $\|\lambda(q^{\mathcal{R}}(f))\|\leq\|q^{\mathcal{R}}(f)\|$
is trivial, the result follows from the density of $q^{\mathcal{R}}(C_{c}(G,A))$
in $\crossedprod$.\end{proof}
Since $q^{\mathcal{R}}(C_{c}(G,A))$ is dense in $\crossedprod$,
we conclude that $\lambda\circ q^{\mathcal{R}}(C_{c}(G,A))$ is dense
in $\lambda(\crossedprod)$, i.e., that $(i_{A}^{\mathcal{R}}\rtimes i_{G}^{\mathcal{R}})(C_{c}(G,A))$
is dense in $\lambda(\crossedprod)$. Together with Proposition \ref{prop:left-regular-is-topological-embedding}
this gives the two additional hypotheses alluded to before, under
which the following uniqueness theorem can now be established. As
mentioned earlier, this should be compared with Raeburn's result for
the crossed product of a $C^{*}$-algebra, see \cite{RaeburnOriginalUniversalPaper}
or \cite[Theorem 2.61]{Williams}.
\begin{thm}
\label{thm:universal-property}Let $(A,G,\alpha)$ be a Banach algebra
dynamical system, where $A$ has a bounded approximate left identity.
Let $\mathcal{R}$ be a uniformly bounded class of non-degenerate
continuous covariant representations of $(A,G,\alpha)$. Let $B$
be a Banach algebra with a bounded approximate left identity, such
that $\lambda:B\to\leftcent(B)$ is a topological embedding. Let $(k_{A},k_{G})$
be a non-degenerate $\mathcal{R}$-continuous covariant representation
of $(A,G,\alpha)$ on the Banach space $B$, such that
\begin{enumerate}
\item $k_{A}(A),k_{G}(G)\subseteq\leftcent(B)$
\item $(k_{A}\rtimes k_{G})(C_{c}(G,A))\subseteq\lambda(B)$
\item $(k_{A}\rtimes k_{G})(C_{c}(G,A))$ is dense in $\lambda(B)$.
\end{enumerate}
Suppose that, for each non-degenerate $\mathcal{R}$-continuous covariant
representation $(\pi,U)$ of $(A,G,\alpha)$ on a Banach space $X$,
there exists a non-degenerate bounded representation $T:B\to B(X)$
such that the non-degenerate bounded representation $\overline{T}:\leftcent(B)\to B(X)$
in the commuting diagram
\[
\xymatrix{B\ar[dr]^{\lambda}\ar[r]^{T} & B(X)\\
 & \leftcent(B)\ar[u]_{\overline{T}}
}
\]
generates $(\pi,U)$ as in Lemma \ref{lem:existence_of_generating_pair},
i.e., is such that $\overline{T}\circ k_{A}=\pi$ and $\overline{T}\circ k_{G}=U$.

Then there exists a unique topological isomorphism $\psi:\crossedprod\to B$,
such that the induced topological isomorphism $\psi_{l}:\leftcent(\crossedprod)\to\leftcent(B)$,
defined by $\psi_{l}(L):=\psi L\psi^{-1}$ for $L\in\leftcent(\crossedprod)$,
induces $(k_{A},k_{G})$ as in Lemma \ref{lem:generating-pair-translation}
from $(i_{A}^{\mathcal{R}},i_{G}^{\mathcal{R}})$, i.e., is such that
$k_{A}=\psi_{l}\circ i_{A}^{\mathcal{R}}$ and $k_{G}=\psi_{l}\circ i_{G}^{\mathcal{R}}$.\end{thm}
\begin{proof}
Proposition \ref{thm:isometric_realization} provides a non-degenerate
$\mathcal{R}$-continuous covariant representation $(\pi,U)$ such
that $(\pi\rtimes U)^{\mathcal{R}}$ is an isometric representation
of $\crossedprod$. If $T:B\to B(X)$ is a non-degenerate bounded
representation such that $\overline{T}\circ k_{A}=\pi$ and $\overline{T}\circ k_{G}=U$,
then Lemma \ref{lem:existence_of_generating_pair} shows that $(\overline{T}\circ k_{A})\rtimes(\overline{T}\circ k_{G})=\overline{T}\circ(k_{A}\rtimes k_{G})$,
i.e., that $\pi\rtimes U=\overline{T}\circ(k_{A}\rtimes k_{G})$.
Hence, for $f\in C_{c}(G,A)$:
\begin{eqnarray*}
\|q^{\mathcal{R}}(f)\|^{\mathcal{R}} & = & \|(\pi\rtimes U)^{\mathcal{R}}(q^{\mathcal{R}}(f))\|\\
 & = & \|\pi\rtimes U(f)\|\\
 & = & \|\overline{T}\circ(k_{A}\rtimes k_{G})(f)\|\\
 & \leq & \|\overline{T}\|\|k_{A}\rtimes k_{G}(f)\|\\
 & = & \|\overline{T}\|\|(k_{A}\rtimes k_{G})^{\mathcal{R}}(q^{\mathcal{R}}(f))\|.
\end{eqnarray*}
Since $(k_{A},k_{G})$ is $\mathcal{R}$-continuous, $\|(k_{A}\rtimes k_{G})^{\mathcal{R}}(q^{\mathcal{R}}(f))\|\leq\|(k_{A}\rtimes k_{G})^{\mathcal{R}}\|\|(q^{\mathcal{R}}(f))\|^{\mathcal{R}}$,
for all $f\in C_{c}(G,A)$, hence we can now conclude, using (2) and
(3) and the fact that $\lambda(B)$ is closed, that $(k_{A}\rtimes k_{G})^{\mathcal{R}}:\crossedprod\to\lambda(B)$
is a topological isomorphism of Banach algebras. Since $\lambda:B\to\leftcent(B)$
is assumed to be a topological embedding, 
\[
\psi:=\lambda^{-1}\circ(k_{A}\rtimes k_{G})^{\mathcal{R}}:\crossedprod\to B
\]
 is a topological isomorphism. 

We proceed to show that $\psi_{l}$ induces $k_{A}$ and $k_{G}$.
As a preparation, note that, since $(\pi\rtimes U)^{\mathcal{R}}$
is isometric and $(\pi\rtimes U)^{\mathcal{R}}=\overline{T}\circ(k_{A}\rtimes k_{G})^{\mathcal{R}}$,
the map $\overline{T}:\leftcent(B)\to B(X)$ is injective on $(k_{A}\rtimes k_{G})^{\mathcal{R}}(\crossedprod)=\lambda(B)$.
Now by \cite[Proposition 5.3]{2011arXiv1104.5151D}, for $a\in A$
and $f\in C_{c}(G,A),$ we have 
\begin{eqnarray*}
(\pi\rtimes U)^{\mathcal{R}}(i_{A}^{\mathcal{R}}(a)q^{\mathcal{R}}(f)) & = & \pi\rtimes U(i_{A}(a)f)\\
 & = & \pi(a)\pi\rtimes U(f)\\
 & = & \overline{T}\circ k_{A}(a)(\pi\rtimes U)^{\mathcal{R}}(q^{\mathcal{R}}(f))\\
 & = & \overline{T}\circ k_{A}(a)\overline{T}\circ(k_{A}\rtimes k_{G})^{\mathcal{R}}(q^{\mathcal{R}}(f))\\
 & = & \overline{T}\left(k_{A}(a)(k_{A}\rtimes k_{G})^{\mathcal{R}}(q^{\mathcal{R}}(f))\right).
\end{eqnarray*}
Since $\lambda(B)$ is a left ideal in $\leftcent(B)$ (as is the
case for every Banach algebra), we note that $k_{A}(a)(k_{A}\rtimes k_{G})^{\mathcal{R}}(q^{\mathcal{R}}(f))\in\lambda(B)$.
On the other hand, we also have 
\begin{eqnarray*}
(\pi\rtimes U)^{\mathcal{R}}(i_{A}^{\mathcal{R}}(a)q^{\mathcal{R}}(f)) & = & \overline{T}\left((k_{A}\rtimes k_{G})^{\mathcal{R}}(i_{A}^{\mathcal{R}}(a)q^{\mathcal{R}}(f))\right).
\end{eqnarray*}
Hence the injectivity of $\overline{T}$ on $\lambda(B)$ shows that
\[
k_{A}(a)(k_{A}\rtimes k_{G})^{\mathcal{R}}(q^{\mathcal{R}}(f))=(k_{A}\rtimes k_{G})^{\mathcal{R}}(i_{A}^{\mathcal{R}}(a)q^{\mathcal{R}}(f)).
\]
We now apply $\lambda^{-1}$ to both sides, and use that $\lambda^{-1}(L\circ\lambda(b))=L(b)$
for all $L\in\leftcent(B)$ and $b\in B$, to see that 
\begin{eqnarray*}
\psi(i_{A}^{\mathcal{R}}(a)q^{\mathcal{R}}(f)) & = & \lambda^{-1}\circ(k_{A}\rtimes k_{G})^{\mathcal{R}}(i_{A}^{\mathcal{R}}(a)q^{\mathcal{R}}(f))\\
 & = & \lambda^{-1}\left(k_{A}(a)(k_{A}\rtimes k_{G})^{\mathcal{R}}(q^{\mathcal{R}}(f))\right)\\
 & = & \lambda^{-1}\left(k_{A}(a)\circ\lambda\circ\lambda^{-1}\circ(k_{A}\rtimes k_{G})^{\mathcal{R}}(q^{\mathcal{R}}(f))\right)\\
 & = & \lambda^{-1}\left(k_{A}(a)\circ\lambda\circ\psi(q^{\mathcal{R}}(f))\right)\\
 & = & k_{A}(a)\psi(q^{\mathcal{R}}(f)).
\end{eqnarray*}
By density, we conclude that $\psi\circ i_{A}^{\mathcal{R}}(a)=k_{A}(a)\circ\psi$,
for all $a\in A$, i.e., that $k_{A}=\psi_{l}\circ i_{A}^{\mathcal{R}}$.
A similar argument yields $k_{G}=\psi_{l}\circ i_{G}^{\mathcal{R}}$.

As to uniqueness, suppose that $\phi:\crossedprod\to B$ is a topological
isomorphism such that $k_{A}=\phi_{l}\circ i_{A}^{\mathcal{R}}$ and
$k_{G}=\phi_{l}\circ i_{G}^{\mathcal{R}}$. Remembering that $(i_{A}^{\mathcal{R}}\rtimes i_{G}^{\mathcal{R}})^{\mathcal{R}}=\lambda$
\cite[Theorem 7.2]{2011arXiv1104.5151D}, this readily implies that,
for $f\in C_{c}(G,A)$,
\begin{eqnarray*}
\phi\circ(\lambda(q^{\mathcal{R}}(f))\circ\phi^{-1} & = & \phi\circ(i_{A}^{\mathcal{R}}\rtimes i_{G}^{\mathcal{R}})^{\mathcal{R}}(q^{\mathcal{R}}(f))\circ\phi^{-1}\\
 & = & (k_{A}\rtimes k_{G})^{\mathcal{R}}(q^{\mathcal{R}}(f)),
\end{eqnarray*}
hence 
\[
(k_{A}\rtimes k_{G})^{\mathcal{R}}(q^{\mathcal{R}}(f))\circ\phi=\phi\circ\lambda(q^{\mathcal{R}}(f)).
\]
Applying this to $q^{\mathcal{R}}(g)$, for $g\in C_{c}(G,A)$, we
find
\begin{eqnarray*}
(k_{A}\rtimes k_{G})^{\mathcal{R}}(q^{\mathcal{R}}(f))\phi(q^{\mathcal{R}}(g)) & = & \phi\left(\lambda(q^{\mathcal{R}}(f))q^{\mathcal{R}}(g)\right)\\
 & = & \phi(q^{\mathcal{R}}(f)q^{\mathcal{R}}(g))\\
 & = & \phi(q^{\mathcal{R}}(f))\phi(q^{\mathcal{R}}(g))\\
 & = & \lambda(\phi(q^{\mathcal{R}}(f)))\phi(q^{\mathcal{R}}(g)).
\end{eqnarray*}
By density, we conclude that $(k_{A}\rtimes k_{G})^{\mathcal{R}}(q^{\mathcal{R}}(f))=\lambda(\phi(q^{\mathcal{R}}(f))),$
for all $f\in C_{c}(G,A)$. Again by density, we conclude that
\[
\phi=\lambda^{-1}\circ(k_{A}\rtimes k_{G})^{\mathcal{R}}=\psi.
\]
\end{proof}

\section{\label{sec:Applications_L1_and_Beurling_algebras}Generalized Beurling
algebras as crossed products}

In this section we give sufficient conditions for a crossed product
of a Banach algebra to be topologically isomorphic to a generalized
Beurling algebra (see Definition \ref{def:generalized-Beurling}),
cf.\,Theorem \ref{thm:faithful-representation} and Corollary \ref{cor:crossedprod_is_Beulringalgebra}.
Since these conditions can always be satisfied, all generalized Beurling
algebras are topologically isomorphic to a crossed product (cf.\,Theorem
\ref{thm:Choosing-R-correctly-Crossed-Products-are-beurling}). Through
an application of the General Correspondence Theorem (Theorem \ref{thm:General-Correspondence-Theorem})
we then obtain a bijection between the non-degenerate bounded representations
of such a generalized Beurling algebra and the non-degenerate continuous
covariant representations of the Banach algebra dynamical system where
the group representation is bounded by a multiple of the weight, cf.\,Theorem
\ref{thm:continuous-non-deg-covars-are-R-continuous}. When the Banach
algebra in the Banach algebra dynamical system is taken to be the
scalars, and the weight on the group $G$ taken to be the constant
1 function, then Corollary \ref{cor:crossedprod_is_Beulringalgebra}
shows that $L^{1}(G)$ is isometrically isomorphic to a crossed product,
and Theorem \ref{thm:continuous-non-deg-covars-are-R-continuous}
reduces to the classical bijective correspondence between uniformly
bounded strongly continuous representations of $G$ and non-degenerate
bounded representations of $L^{1}(G)$, cf.\,Corollary \ref{cor:classical-L1-result}.

We start with the topological isomorphism between a generalized Beurling
algebra and a crossed product. 
\begin{defn}
For a locally compact group $G$, let $\omega:G\to[0,\infty)$ be
a non-zero submultiplicative Borel measurable function. Then $\omega$
is called a \emph{weight }on $G$. 
\end{defn}
Note that we do not assume that $\omega\geq1$, as is done in some
parts of the literature. The fact that $\omega$ is non-zero readily
implies that $\omega(e)\geq1$. More generally, if $K\subseteq G$
is a compact set, there exist $a,b>0$ such that $a\leq\omega(s)\leq b$
for all $s\in K$ \cite[Lemma 1.3.3]{Kaniuth}. 

Let $(A,G,\alpha)$ be a Banach algebra dynamical system, and $\mathcal{R}$
a uniformly bounded class of continuous covariant representations
of $(A,G,\alpha)$. Recall that $C^{\mathcal{R}}:=\sup_{(\pi,U)\in\mathcal{R}}\|\pi\|$
and $\nu^{\mathcal{R}}:G\to\mathbb{R}_{\geq0}$ is defined by $\nu^{\mathcal{R}}(r):=\sup_{(\pi,U)\in\mathcal{R}}\|U_{r}\|$
as in \cite[Definition 3.1]{2011arXiv1104.5151D}. We note that the
map $\nu^{\mathcal{R}}$ is a weight on $G$. It is clearly submultiplicative,
and, being the supremum of a family of continuous maps $\{r\mapsto\|U_{r}x\|:(\pi,U)\in\mathcal{R},\ x\in X_{\pi},\ \|x\|\leq1\}$,
the map $\nu^{\mathcal{R}}$ is lower semicontinuous, hence Borel
measurable. 

Let $\omega$ be a weight on $G$, such that $\nu^{\mathcal{R}}\leq\omega$.
Then, for all $f\in C_{c}(G,A)$, 
\begin{eqnarray}
\sigma^{\mathcal{R}}(f) & = & \sup_{(\pi,U)\in\mathcal{R}}\left\Vert \int_{G}\pi(f(s))U_{s}\, ds\right\Vert \nonumber \\
 & \leq & \sup_{(\pi,U)\in\mathcal{R}}\int_{G}\|\pi(f(s))\|\|U_{s}\|\, ds\nonumber \\
 & \leq & C^{\mathcal{R}}\int_{G}\|f(s)\|\nu^{\mathcal{R}}(s)\, ds\nonumber \\
 & \leq & C^{\mathcal{R}}\int_{G}\|f(s)\|\omega(s)\, ds\nonumber \\
 & = & C^{\mathcal{R}}\|f\|_{1,\omega},\label{eq:seminorm-bounded-by-beurlingnorm}
\end{eqnarray}
where $\|\cdot\|_{1,\omega}$ denotes the $\omega$-weighted $L^{1}$-norm.
In Theorem \ref{thm:faithful-representation}, we will give sufficient
conditions under which a reverse inequality holds. Then $\sigma^{\mathcal{R}}$
is actually a norm on $C_{c}(G,A)$ and is equivalent to a weighted
$L^{1}$-norm, so that $\crossedprod$ will be isomorphic to a generalized
Beurling algebra to be defined shortly. 
\begin{defn}
Let $X$ be a Banach space, $1\leq p<\infty$, and $\omega:G\to[0,\infty)$
a weight on $G$. We define the weighted $p$-norm on $C_{c}(G,X)$
by 
\[
\|h\|_{p,\omega}:=\left(\int_{G}\|h(s)\|^{p}\omega(s)\, ds\right)^{1/p},
\]
and define $L^{p}(G,X,\omega)$ as the completion of $C_{c}(G,X)$
with this norm.\end{defn}
\begin{rem}
By definition $L^{p}(G,A,\omega)$ with $1\leq p<\infty$ has $C_{c}(G,A)$
as a dense subspace. In view of the central role of $C_{c}(G,A)$
in our theory of crossed products of Banach algebras, this is clearly
desirable, but it would be unsatisfactory not to discuss the relation
with spaces of Bochner integrable functions. We will now address this
and explain that for $p=1$ (our main space of interest in the sequel),
$L^{1}(G,A,\omega)$ is (isometrically isomorphic to) a Bochner space. 

In most of the literature, the theory of the Bochner integral is developed
for finite (or at least $\sigma$-finite) measures, and sometimes
the Banach space in question is assumed to be separable. Since $\omega d\mu$
(where $\mu$ is the left Haar measure on $G$) need not be $\sigma$-finite
and $A$ need not be separable, this is not applicable to our situation.
In \cite[Appendix E]{Cohn}, however, the theory is developed for
an arbitrary measure $\mu$ on a $\sigma$-algebra $\mathcal{A}$
of subsets of a set $\Omega$, and functions $f:\Omega\to X$ with
values in an arbitrary Banach space $X$. Such a function $f$ is
called Bochner integrable if $f^{-1}(B)\in\mathcal{A}$ for every
Borel subset of $X$, $f(\Omega)$ is separable, and $\int_{\Omega}\|f(\xi)\|d\mu(\xi)<\infty$
(the measurability of $\xi\mapsto\|f(\xi)\|$ is an automatic consequence
of the Borel measurability of $f$). Identifying Bochner integrable
functions that are equal $\mu$-almost everywhere, one obtains a Banach
space $L^{1}(\Omega,\mathcal{A},\mu,X)$, where the norm is given
by $\|[f]\|=\int_{\Omega}\|f(\xi)\|\, d\mu(\xi)$ with $f$ any representative
of the equivalence class $[f]\in L^{1}(\Omega,\mathcal{A},\mu,X)$.
Although it is not stated as such, it is in fact proved \cite[p.\,352]{Cohn}
that the simple Bochner integrable functions (i.e., all functions
of the form $\sum_{i=1}^{n}\chi_{A_{i}}\otimes x_{i}$, where $A_{i}\in\mathcal{A}$,
$\mu(A_{i})<\infty$ and $x_{i}\in X$) form a dense subspace of $L^{1}(\Omega,\mathcal{A},\mu,X)$.

We claim that our space $L^{1}(G,A,\omega)$ is isometrically isomorphic
to $L^{1}(G,\mathcal{B},\omega d\mu,A)$, where $\mathcal{B}$ is
the Borel $\sigma$-algebra of $G$, and $\mu$ is the left Haar measure
on $G$ again. To start with, if $f\in C_{c}(G,A)$, then certainly
$f$ is Bochner integrable, so that we obtain a (clearly isometric)
inclusion map $j:C_{c}(G,A)\to L^{1}(G,\mathcal{B},\omega d\mu,A)$,
that can be extended to an isometric embedding of $L^{1}(G,A,\omega)$
into $L^{1}(G,\mathcal{B},\omega d\mu,A)$. To see that the image
is dense, it is, in view of the density of the simple Bochner integrable
functions in $L^{1}(G,\mathcal{B},\omega d\mu,A)$, sufficient to
approximate $\chi_{S}\otimes a$ with elements from $C_{c}(G,A),$
for arbitrary $a\in A$ and $S\in\mathcal{B}$ with $\int_{G}\chi_{S}\:\omega d\mu<\infty$.
Since $C_{c}(G)$ is dense in $L^{1}(G,\omega d\mu)$ \cite[Lemma 1.3.5]{Kaniuth},
we can choose a sequence $(f_{n})\subseteq C_{c}(G)$ such that $f_{n}\to\chi_{S}$
in $L^{1}(G,\omega d\mu)$, and then 
\[
\|f_{n}\otimes a-\chi_{S}\otimes a\|=\|a\|\int_{G}|f_{n}(r)-\chi_{S}(r)|\omega(r)\, d\mu(r)\to0.
\]
Hence the image of $j:C_{c}(G,A)\to L^{1}(G,\mathcal{B},\omega d\mu,A)$
is dense and our claim has been established. 

For the sake of completeness, we note that one cannot argue that $\omega d\mu$
is ``clearly'' a regular Borel measure on $G$, so that $C_{c}(G)$
is dense in $L^{1}(G,\omega d\mu)$ by the standard density result
\cite[Proposition 7.9]{Folland}. Indeed, although \cite[Exercises 7.2.7--9]{Folland}
give sufficient conditions for this to hold (none of which applies
in our general case), the regularity is not automatic, see \cite[Exercise 7.2.13]{Folland}.
The proof of the density of $C_{c}(G)$ in $L^{1}(G,\omega d\mu)$
in \cite[ Lemma 1.3.5]{Kaniuth} is direct and from first principles.
It uses in an essential manner that $\omega$ is bounded away from
zero on compact subsets of $G$, but not that the Haar measure should
be $\sigma$-finite or that $\omega$ should be integrable.
\end{rem}
{}

With $(A,G,\alpha)$ a Banach algebra dynamical system and $\omega$
a weight on $G$, if $\alpha$ is uniformly bounded, say $\|\alpha_{s}\|\leq C_{\alpha}$
for some $C_{\alpha}\geq0$ and all $s\in G$, then, using the submultiplicativity
of $\omega$, it is routine to verify that 
\[
\|f*g\|_{1,\omega}\leq C_{\alpha}\|f\|_{1,\omega}\|g\|_{1,\omega}\quad(f,g\in C_{c}(G,A)).
\]
Hence the Banach space $L^{1}(G,A,\omega)$ can be supplied with the
structure of an associative algebra, such that 

\[
\|u*v\|_{1,\omega}\leq C_{\alpha}\|u\|_{1,\omega}\|v\|_{1,\omega}\quad(u,v\in L^{1}(G,A,\omega)).
\]
If $C_{\alpha}=1$ (i.e., if $\alpha$ lets $G$ act as isometries
on $A$), then $L^{1}(G,A,\omega)$ is a Banach algebra, and when
$C_{\alpha}\neq1$, as is well known, there is an equivalent norm
on $L^{1}(G,A,\omega)$ such that it becomes a Banach algebra. We
will show below (cf.\,Theorem \ref{thm:Choosing-R-correctly-Crossed-Products-are-beurling})
that such a norm can always be obtained from a topological isomorphism
between $L^{1}(G,A,\omega)$ and the crossed product $\crossedprod$
for a suitable choice of $\mathcal{R}$.
\begin{defn}
\global\long\def\BeurlingTypeAlg{L^{1}(G,A,\omega;\alpha)}
\label{def:generalized-Beurling}Let $(A,G,\alpha)$ be a Banach algebra
dynamical system with $\alpha$ uniformly bounded and $\omega$ a
weight on $G$. The Banach space $L^{1}(G,A,\omega)$ endowed with
the continuous multiplication induced by the twisted convolution on
$C_{c}(G,A)$, given by 
\[
[f*g](s):=\int_{G}f(r)\alpha_{r}(g(r^{-1}s))\, dr\quad(f,g\in C_{c}(G,A),\ s\in G),
\]
will be denoted by $\BeurlingTypeAlg$ and called a \emph{generalized
Beurling algebra.}
\end{defn}
As a special case, we note that for $A=\mathbb{K}$, the generalized
Beurling algebra reduces to the classical Beurling algebra $L^{1}(G,\omega)$,
which is a true Banach algebra.

{}

Obtaining such a reverse inequality to (\ref{eq:seminorm-bounded-by-beurlingnorm})
rests on inducing a covariant representation of $(A,G,\alpha)$ from
the left regular representation $\lambda:A\to B(A)$ of $A$, analogous
to \cite[Example 2.14]{Williams}. The key result is Proposition \ref{pro:reverseinequality}
and we will now start working towards it.
\begin{defn}
\label{def:induced_rep_and_translation_rep}Let $(A,G,\alpha)$ be
a Banach algebra dynamical system and let $\pi:A\to B(X)$ be a bounded
representation of $A$ on a Banach space $X$. We define the induced
algebra representation $\tilde{\pi}$ and left regular group representation
$\Lambda$ on the space of all functions from $G$ to $X$ by the
formulae:
\begin{eqnarray*}
[\tilde{\pi}(a)h](s) & := & \pi(\alpha_{s}^{-1}(a))h(s),\\
(\Lambda_{r}h)(s) & := & h(r^{-1}s),
\end{eqnarray*}
where $h:G\to X$, $r,s\in G$ and $a\in A$. 
\end{defn}
A routine calculation, left to the reader, shows that $(\tilde{\pi},\Lambda)$
is covariant.

We need a number of lemmas in preparation for the proof of Proposition
\ref{pro:reverseinequality}. 

The following is clear.
\begin{lem}
\label{lem:tilde_lambda_bounded}If $(A,G,\alpha)$ is a Banach algebra
dynamical system with $\alpha$ uniformly bounded by a constant $C_{\alpha}$,
and $\omega:G\to[0,\infty)$ a weight, then for any bounded representation
$\pi:A\to B(X)$ on a Banach space $X$, both the maps $\tilde{\pi}:A\to B(C_{0}(G,X))$
\textup{(}as defined in Definition \ref{def:induced_rep_and_translation_rep}\textup{)}
and $\tilde{\pi}:A\to B(L^{p}(G,X,\omega))$ for $1\leq p<\infty$
\textup{(}the canonically induced representation $\tilde{\pi}$ of
$A$ on $L^{p}(G,X,\omega)$ as completion of $C_{c}(G,X)$ with the
$\|\cdot\|_{p,\omega}$-norm\textup{)} are representations with norms
bounded by $C_{\alpha}\|\pi\|$. Moreover, $C_{c}(G,X)$ is invariant
under both $A$-actions.
\end{lem}
{}
\begin{lem}
If $(A,G,\alpha)$ is a Banach algebra dynamical system and $X$ a
Banach space and $\omega$ a weight on $G$, then both the left regular
representations $\Lambda:G\to B(C_{0}(G,X))$ \textup{(}as defined
in Definition \ref{def:induced_rep_and_translation_rep}\textup{)},
and $\Lambda:G\to B(L^{p}(G,X,\omega))$ for $1\leq p<\infty$ \textup{(}the
canonically induced representation $\Lambda$ of $G$ on $L^{p}(G,X,\omega)$
as completion of $C_{c}(G,X)$ with the $\|\cdot\|_{p,\omega}$-norm\textup{)}
are strongly continuous group representations. The representation
$\Lambda:G\to B(C_{0}(G,X))$ acts as isometries, and $\Lambda:G\to B(L^{p}(G,X,\omega))$
is bounded by $\omega^{1/p}$. Moreover, $C_{c}(G,X)$ is invariant
under both $G$-actions.\end{lem}
\begin{proof}
That $\Lambda:G\to B(C_{0}(G,X))$ acts on $C_{0}(G,X)$ as isometries
is clear. 

That $\Lambda:G\to B(L^{p}(G,X,\omega))$ is bounded by $\omega^{1/p}$
follows from left invariance of the Haar measure and submultiplicativity
of $\omega$: For any $h\in C_{c}(G,X)$ and $s\in G$,
\begin{eqnarray*}
\|\Lambda_{s}h\|_{p,\omega}^{p} & = & \int_{G}\|[\Lambda_{s}h](t)\|^{p}\omega(t)\, dt\\
 & = & \int_{G}\|h(s^{-1}t)\|^{p}\omega(t)\, dt\\
 & = & \int_{G}\|h(t)\|^{p}\omega(st)\, dt\\
 & \leq & \omega(s)\int_{G}\|h(t)\|^{p}\omega(t)\, dt\\
 & = & \omega(s)\|h\|_{p,\omega}^{p}.
\end{eqnarray*}
Therefore $\Lambda_{s}$ induces a map on $L^{p}(G,X,\omega)$ with
the same norm, denoted by the same symbol, and $\|\Lambda_{s}\|\leq\omega(s)^{1/p}$.

To establish strong continuity of $\Lambda:G\to B(C_{0}(G,X))$ and
$\Lambda:G\to B(L^{p}(G,X,\omega))$, it is sufficient to establish
strong continuity at $e\in G$ on dense subsets of both $C_{0}(G,X)$
of $L^{p}(G,X,\omega)$ respectively \cite[Corollary 2.5]{2011arXiv1104.5151D}.
By the uniform continuity of elements in $C_{c}(G,X)$ \cite[Lemma 1.88]{Williams}
and the density of $C_{c}(G,X)$ in $C_{0}(G,X)$, the result follows
for $\Lambda:G\to B(C_{0}(G,X))$.

To establish the result for $L^{p}(G,X,\omega)$, let $\varepsilon>0$
and $h\in C_{c}(G,X)$ be arbitrary. Let $K:=\mbox{supp}(h)$ and
$W$ a precompact neighbourhood of $e\in G$. By uniform continuity
of $h$, there exists a symmetric neighbourhood $V\subseteq W$ of
$e\in G$ such that $\|\Lambda_{s}h-h\|_{\infty}^{p}<\varepsilon^{p}/\left(\sup_{r\in WK}\omega(r)\right)\mu(WK)$
for all $s\in V$. Then, for $s\in V$, 
\begin{eqnarray*}
\|\Lambda_{s}h-h\|_{p,\omega}^{p} & = & \int_{WK}\|h(s^{-1}r)-h(r)\|^{p}\omega(r)\, dr\\
 & \leq & \frac{\varepsilon^{p}}{\left(\sup_{r\in WK}\omega(r)^{p}\right)\mu(WK)}\int_{WK}\omega(r)\, dr\\
 & \leq & \varepsilon^{p}.
\end{eqnarray*}
By the density of $C_{c}(G,X)$ in $L^{p}(G,X,\omega)$, the result
follows.\end{proof}
\begin{lem}
Let $(A,G,\alpha)$ be a Banach algebra dynamical system where $\alpha$
is uniformly bounded by a constant $C_{\alpha}$ and $\omega$ a weight
on $G$. Let $\pi:A\to B(X)$ be a non-degenerate bounded representation
on a Banach space $X$. If $f\in C_{c}(G,X)$, then there exist a
compact subset $K$ of $G$, containing $\textup{supp}(f)$, and a
sequence $(f_{n})\subseteq\textup{span}\,(\tilde{\pi}(A)(C_{c}(G)\otimes X))$
such that $\textup{supp}(f_{n})\subseteq K$ for all $n$, and $(f_{n})$
converges uniformly to $f$ on $G$. Consequently the representations
$\tilde{\pi}:A\to B(C_{0}(G,X))$ and $\tilde{\pi}:A\to B(L^{p}(G,X,\omega))$
for $1\leq p<\infty$ \textup{(}as yielded by Definition \ref{def:induced_rep_and_translation_rep}\textup{)}
are then non-degenerate.\end{lem}
\begin{proof}
Let $f\in C_{c}(G,X)$ and $\varepsilon>0$ be arbitrary. Since $\mbox{supp}(f)$
is compact, we can fix some precompact open set $U_{f}$ containing
$\mbox{supp}(f)$. Since $\pi$ is non-degenerate, for every $s\in G$,
there exist finite sets $\{a_{i,s}\}_{i=1}^{n_{s}}$ and $\{x_{i,s}\}_{i=1}^{n_{s}}$
such that $\|f(s)-\sum_{i=1}^{n_{s}}\pi(a_{i,s})x_{i,s}\|<\varepsilon$.
Since $\alpha$ is strongly continuous, for each $s\in G$ and $i\in\{1,\ldots,n_{s}\}$,
there exists some precompact neighbourhood $W_{i,s}$ of $s$, such
that $t\in W_{i,s}$ implies $n_{s}\|\pi\|\|x_{i,s}\|\|a_{i,s}-\alpha_{t}^{-1}\circ\alpha_{s}(a_{i,s})\|<\varepsilon$.
Furthermore, for any $s\in G$, we can choose a precompact neighbourhood
$V_{s}$ of $s$ such that $t\in V_{s}$ implies $\|f(s)-f(t)\|<\varepsilon$.
Define $W_{s}:=\bigcap_{i=1}^{n_{s}}W_{i,s}\cap V_{s}\cap U_{f}$.
Now $\{W_{s}\}_{s\in G}$ is an open cover of $\mbox{supp}(f)$, hence
let $\{W_{s_{j}}\}_{j=1}^{m}$ be a finite subcover. Let $\{u_{j}\}_{j=1}^{m}\subseteq C_{c}(G)$
be a partition of unity such that, for all $j\in\{1,\ldots,m\}$,
$0\leq u_{j}(t)\leq1$ for $t\in G$, $\mbox{supp}(u_{j})\subseteq W_{s_{j}}$,
$\sum_{j=1}^{m}u_{j}(t)=1$ for $t\in\mbox{supp}(f)$, and $\sum_{j=1}^{m}u_{j}(t)\leq1$
for $t\in G$. Then, for $t\in G$, 
\begin{eqnarray*}
 &  & \left\Vert f(t)-\left(\sum_{j=1}^{m}\sum_{i=1}^{n_{s_{j}}}\tilde{\pi}(\alpha_{s_{j}}(a_{i,s_{j}}))u_{j}\otimes x_{i,s_{j}}\right)(t)\right\Vert \\
 & = & \left\Vert f(t)-\sum_{j=1}^{m}\sum_{i=1}^{n_{s_{j}}}u_{j}(t)\pi(\alpha_{t}^{-1}\circ\alpha_{s_{j}}(a_{i,s_{j}}))x_{i,s_{j}}\right\Vert \\
 & = & \left\Vert \sum_{j=1}^{m}u_{j}(t)f(t)-\sum_{j=1}^{m}u_{j}(t)f(s_{j})+\sum_{j=1}^{m}u_{j}(t)f(s_{j})-\sum_{j=1}^{m}u_{j}(t)\sum_{i=1}^{n_{s_{j}}}\pi(a_{i,s_{j}})x_{i,s_{j}}\right.\\
 &  & \quad\left.+\sum_{j=1}^{m}u_{j}(t)\sum_{i=1}^{n_{s_{j}}}\pi(a_{i,s_{j}})x_{i,s_{j}}-\sum_{j=1}^{m}\sum_{i=1}^{n_{s_{j}}}u_{j}(t)\pi(\alpha_{t}^{-1}\circ\alpha_{s_{j}}(a_{i,s_{j}}))x_{i,s_{j}}\right\Vert \\
 & \leq & \sum_{j=1}^{m}u_{j}(t)\left\Vert f(t)-f(s_{j})\right\Vert +\sum_{j=1}^{m}u_{j}(t)\left\Vert f(s_{j})-\sum_{i=1}^{n_{s_{j}}}\pi(a_{i,s_{j}})x_{i,s_{j}}\right\Vert \\
 &  & \quad+\left\Vert \sum_{j=1}^{m}u_{j}(t)\sum_{i=1}^{n_{s_{j}}}\pi(a_{i,s_{j}})x_{i,s_{j}}-\sum_{j=1}^{m}\sum_{i=1}^{n_{s_{j}}}u_{j}(t)\pi(\alpha_{t}^{-1}\circ\alpha_{s_{j}}(a_{i,s_{j}}))x_{i,s_{j}}\right\Vert \\
 & \leq & \varepsilon+\varepsilon+\left\Vert \sum_{j=1}^{m}u_{j}(t)\sum_{i=1}^{n_{s_{j}}}\pi(a_{i,s_{j}})x_{i,s_{j}}-\sum_{j=1}^{m}\sum_{i=1}^{n_{s_{j}}}u_{j}(t)\pi(\alpha_{t}^{-1}\circ\alpha_{s_{j}}(a_{i,s_{j}}))x_{i,s_{j}}\right\Vert \\
 & \leq & \varepsilon+\varepsilon+\sum_{j=1}^{m}u_{j}(t)\sum_{i=1}^{n_{s_{j}}}\left\Vert \pi(a_{i,s_{j}})x_{i,s_{j}}-\pi(\alpha_{t}^{-1}\circ\alpha_{s_{j}}(a_{i,s_{j}}))x_{i,s_{j}}\right\Vert \\
 & \leq & \varepsilon+\varepsilon+\sum_{j=1}^{m}u_{j}(t)\sum_{i=1}^{n_{s_{j}}}\|\pi\|\|x_{i,s_{j}}\|\|a_{i,s_{j}}-\alpha_{t}^{-1}\circ\alpha_{s_{j}}(a_{i,s_{j}})\|\\
 & \leq & \varepsilon+\varepsilon+\sum_{j=1}^{m}u_{j}(t)\sum_{i=1}^{n_{s_{j}}}\frac{\varepsilon}{n_{s_{j}}}\\
 & \leq & \varepsilon+\varepsilon+\varepsilon.
\end{eqnarray*}
Since 
\[
\sum_{j=1}^{m}\sum_{i=1}^{n_{s_{j}}}\tilde{\pi}(\alpha_{s_{j}}(a_{i,s_{j}}))u_{j}\otimes x_{i,s_{j}}
\]
is supported in the fixed compact set $\overline{U_{f}}$, the result
follows.\end{proof}
Combining the previous three lemmas yields:
\begin{cor}
\label{cor:pitilde-Lambda-covariant-non-deg}If $(A,G,\alpha)$ is
a Banach algebra dynamical system with $\alpha$ uniformly bounded
by a constant $C_{\alpha}$, $\omega$ a weight on $G$ and $\pi:A\to B(X)$
a bounded representation on a Banach space $X$, then the pair $(\tilde{\pi},\Lambda)$
\textup{(}as yielded by Definition \ref{def:induced_rep_and_translation_rep}\textup{)}
is a continuous covariant representation of $(A,G,\alpha)$ on $C_{0}(G,X)$
or \textup{$L^{p}(G,X,\omega)$ }for $1\leq p<\infty$ respectively.
Moreover: 
\begin{enumerate}
\item Both representations $\tilde{\pi}:A\to B(C_{0}(G,X))$ and $\tilde{\pi}:A\to B(L^{p}(G,X,\omega))$
satisfy $\|\tilde{\pi}\|\leq C_{\alpha}\|\pi\|$.
\item The left regular group representation $\Lambda:G\to B(C_{0}(G,X))$
acts as isometries on $C_{0}(G,X)$, and the left regular group representation
$\Lambda:G\to B(L^{p}(G,X,\omega))$ is bounded by $\omega^{1/p}$
on $G$.
\item The space $C_{c}(G,X)$, seen as a subspace of $C_{0}(G,X)$ or $L^{p}(G,X,\omega)$,
is invariant under actions of both $A$ and $G$ on $C_{0}(G,X)$
or $L^{p}(G,X,\omega)$ through the representations $\tilde{\pi}:A\to B(C_{0}(G,X))$
and $\Lambda:G\to B(C_{0}(G,X))$, or $\tilde{\pi}:A\to B(L^{p}(G,X,\omega))$
and $\Lambda:G\to B(L^{p}(G,X,\omega))$, respectively.
\item If $\pi:A\to B(X)$ is non-degenerate, then so are the representations
$\tilde{\pi}:A\to B(C_{0}(G,X))$ and $\tilde{\pi}:A\to B(L^{p}(G,X,\omega))$.
\end{enumerate}
\end{cor}
If $\alpha$ is uniformly bounded by $C_{\alpha}\geq0$, Corollary
\ref{cor:pitilde-Lambda-covariant-non-deg} shows that the left regular
representation $\lambda:A\to B(A)$ of $A$ is such that the covariant
representation $(\tilde{\lambda},\Lambda)$ of $(A,G,\alpha)$ on
$L^{1}(G,A,\omega)$ (as yielded by Definition \ref{def:induced_rep_and_translation_rep})
is continuous with $\|\tilde{\lambda}\|\leq C_{\alpha}$ and $\|\Lambda_{s}\|\leq\omega(s)$.
Moreover, if $A$ has a bounded left or right  approximate identity,
then $\lambda$ is non-degenerate, and hence $(\tilde{\lambda},\Lambda)$
is non-degenerate.

We need two more results before Proposition \ref{pro:reverseinequality}
can be established.
\begin{lem}
\label{lem:point-evaluations-pulled-through-integral}Let $(A,G,\alpha)$
be a Banach algebra dynamical system with $\alpha$ uniformly bounded.
Let $\omega$ be a weight on $G$, and $\lambda:A\to B(A)$ the left
regular representation of $A$. Let $(\tilde{\lambda},\Lambda)$ be
the continuous covariant representation of $(A,G,\alpha)$ on $L^{1}(G,A,\omega)$
\textup{(}as yielded by Definition \ref{def:induced_rep_and_translation_rep}\textup{)}.
Then, for all $f\in C_{c}(G,A)$, $\tilde{\lambda}\rtimes\Lambda(f)\in B(L^{1}(G,A,\omega))$
leaves the subspace $C_{c}(G,A)$ of $L^{1}(G,A,\omega)$ invariant.
In fact, if $h\in C_{c}(G,A)\subseteq L^{1}(G,A,\omega)$, then $\tilde{\lambda}\rtimes\Lambda(f)h\in L^{1}(G,A,\omega)$
is given by the pointwise formula 
\[
[\tilde{\lambda}\rtimes\Lambda(f)h](s)=\int_{G}\alpha_{s}^{-1}(f(r))h(r^{-1}s)\, dr\quad(s\in G).
\]
\end{lem}
\begin{proof}
We proceed indirectly, via $C_{0}(G,A)$, and write $(\tilde{\lambda}_{0},\Lambda_{0})$
and $(\tilde{\lambda}_{1},\Lambda_{1})$ for the continuous covariant
representations of $(A,G,\alpha)$ on $C_{0}(G,A)$ and $L^{1}(G,A,\omega)$,
respectively. Let $f,h\in C_{c}(G,A)$ and consider the integral 
\[
\tilde{\lambda}_{1}\rtimes\Lambda_{1}(f)h=\int_{G}\tilde{\lambda}_{1}(f(r))\Lambda_{1,r}h\, dr\in L^{1}(G,A,\omega).
\]
Let $K:=\mbox{supp}(f)\cdot\mbox{supp}(h)$, and put $C_{0}(G,A)_{K}:=\{g\in C_{0}(G,A):\mbox{supp}(g)\subseteq K\}$.
Then $C_{0}(G,A)_{K}$ is a closed subspace of $C_{0}(G,A)$ and the
inclusion $j_{K}:C_{0}(G,A)_{K}\to L^{1}(G,A,\omega)$ is bounded,
since $\omega$ is bounded on compact sets. Define $\psi:G\to C_{0}(G,A)_{K}$
by $\psi(r):=\tilde{\lambda}_{0}(f(r))\Lambda_{0,r}h$ for all $r\in G$.
Then $\psi$ is continuous and supported on the compact set $\mbox{supp}(f)\subseteq G$.
Now, by the boundedness of $j_{K}$, 
\[
\int_{G}\tilde{\lambda}_{1}(f(r))\Lambda_{1,r}h\, dr=\int_{G}j_{K}\circ\psi(r)\, dr=j_{K}\left(\int_{G}\psi(r)\, dr\right).
\]
Since $\int_{G}\psi(r)\, dr\in C_{0}(G,A)_{K}$, we conclude that
$\tilde{\lambda}_{1}\rtimes\Lambda_{1}(f)h\in C_{c}(G,A)$.

Since the evaluations $\mbox{ev}_{s}:C_{0}(G,A)_{K}\to A$, sending
$g\in C_{0}(G,A)_{K}$ to $g(s)\in A$, are bounded for all $s\in G$,
we find that, for all $s\in G$,
\begin{eqnarray*}
\left(\int_{G}\psi(r)\, dr\right)(s) & = & \mbox{ev}_{s}\left(\int_{G}\psi(r)\, dr\right)\\
 & = & \int_{G}\mbox{ev}_{s}\circ\psi(r)\, dr\\
 & = & \int_{G}\psi(r)(s)\, dr\\
 & = & \int_{G}\alpha_{s}^{-1}(f(r))h(r^{-1}s)\, dr.
\end{eqnarray*}
Therefore $[\tilde{\lambda}_{1}\rtimes\Lambda_{1}(f)h](s)=\int_{G}\alpha_{s}^{-1}(f(r))h(r^{-1}s)\, dr$.\end{proof}
\begin{lem}
\label{lem:norm-approximation-from-below}Let $A$ be a Banach algebra
with bounded approximate right identity $(u_{i})$ and let $K\subseteq A$
be compact. Then, for any $\varepsilon>0$, there exists an index
$i_{0}$ such that $\|au_{i}\|\geq\|a\|-\varepsilon$ for all $a\in K$
and all $i\geq i_{0}$.\end{lem}
\begin{proof}
Let $M\geq1$ be an upper bound for $(u_{i})$ and $\varepsilon>0$
be arbitrary. By compactness of $K$, there exist $a_{1},\ldots,a_{n}\in K$
such that for all $a\in K$ there exists an index $k\in\{1,\ldots,n\}$
with $\|a-a_{k}\|<\varepsilon/3M\leq\varepsilon/3$. Let $i_{0}$
be such that $\|a_{k}u_{i}-a_{k}\|<\varepsilon/3$ for all $k\in\{1,\ldots,n\}$
and all $i\geq i_{0}$. 

Now, for $a\in K$ arbitrary, let $k_{0}\in\{1,\ldots,n\}$ be such
that $\|a-a_{k_{0}}\|<\varepsilon/3$. For any $i\geq i_{0}$, 
\begin{eqnarray*}
\|au_{i}\| & \geq & \|a\|-\|au_{i}-a_{k_{0}}u_{i}\|-\|a_{k_{0}}u_{i}-a_{k_{0}}\|-\|a_{k_{0}}-a\|\\
 & > & \|a\|-\frac{\varepsilon}{3M}M-\frac{\varepsilon}{3}-\frac{\varepsilon}{3}\\
 & = & \|a\|-\varepsilon.
\end{eqnarray*}
\end{proof}
Finally, we combine Lemmas \ref{lem:tilde_lambda_bounded}--\ref{lem:norm-approximation-from-below}
to obtain the following:
\begin{prop}
\label{pro:reverseinequality}Let $(A,G,\alpha)$ be a Banach algebra
dynamical system where $A$ has an $M$-bounded approximate right
identity and $\alpha$ is uniformly bounded by a constant $C_{\alpha}$.
Let $\omega$ be a weight on $G$, and $\lambda:A\to B(A)$ the left
regular representation of $A$. Let $W\subseteq G$ be a precompact
neighbourhood of $e\in G$. Then the non-degenerate continuous covariant
representation $(\tilde{\lambda},\Lambda)$ of $(A,G,\alpha)$ on
$L^{1}(G,A,\omega)$ \textup{(}as yielded by Definition \ref{def:induced_rep_and_translation_rep}\textup{)}
satisfies 
\[
\|\tilde{\lambda}\rtimes\Lambda(f)\|\geq\frac{1}{C_{\alpha}M\sup_{s\in W}\omega(s)}\|f\|_{1,\omega}
\]
for all $f\in C_{c}(G,A)$. Consequently, the representation $\tilde{\lambda}\rtimes\Lambda:C_{c}(G,A)\to B(L^{1}(G,A,\omega))$
is faithful.\end{prop}
\begin{proof}
Let $(u_{i})$ be an $M$-bounded approximate right identity of $A$
and $W\subseteq G$ any precompact neighbourhood of $e\in G$. Let
$f\in C_{c}(G,A)$ and $\varepsilon>0$ be arbitrary. By the uniform
continuity of $f$, there exists a symmetric neighbourhood $V\subseteq W$
of $e\in G$ such that $\|f(r)-f(rs)\|<\varepsilon/2C_{\alpha}M$
for all $s\in V$ and $r\in G$. By continuity of all maps involved
and the assumption that $f$ is compactly supported, the set $\{\alpha_{s}^{-1}(f(s)):s\in G\}\subseteq A$
is compact. Lemma \ref{lem:norm-approximation-from-below} then asserts
the existence of an index $i_{0}$, such that $\|au_{i_{0}}\|\geq\|a\|-\varepsilon/2$
for all $a\in\{\alpha_{s}^{-1}(f(s)):s\in G\}$.

By Urysohn's Lemma, let $h_{0}:G\to[0,1]$ be continuous with $h_{0}(e)=1$
and $\textup{supp}(h_{0})\subseteq V$, so that $h_{0}\in C_{c}(G)$.
We may assume $h_{0}(r)=h_{0}(r^{-1})$ for all $r\in G$, by replacing
$h_{0}$ with $r\mapsto\max\{h_{0}(r),h_{0}(r^{-1})\}$. Define 
\[
h:=\left(\int_{G}h_{0}(t)\, dt\right)^{-1}h_{0}\otimes u_{i_{0}}\in C_{c}(G,A).
\]
Then 
\begin{eqnarray*}
\|h\|_{1,\omega} & = & \left(\int_{G}h_{0}(t)\, dt\right)^{-1}\int_{G}h_{0}(r)\|u_{i_{0}}\|\omega(r)\, dr\\
 & \leq & M\sup_{r\in V}\omega(r)\\
 & \leq & M\sup_{r\in W}\omega(r).
\end{eqnarray*}

For every $s\in G$, we find, using the reverse triangle inequality,
noting that $\|f(s)\|=\|\alpha_{s}\circ\alpha_{s^{-1}}(f(s))\|\leq C_{\alpha}\|\alpha_{s^{-1}}(f(s))\|$,
remembering that $h_{0}$ is supported in $V$, and applying Lemma
\ref{lem:point-evaluations-pulled-through-integral}, that

\begin{eqnarray*}
 &  & \|[\tilde{\lambda}\rtimes\Lambda(f)h](s)\|\\
 & = & \left\Vert \int_{G}\alpha_{s}^{-1}(f(r))h(r^{-1}s)\, dr\right\Vert \\
 & = & \left\Vert \int_{G}\alpha_{s}^{-1}(f(sr))h(r^{-1})\, dr\right\Vert \\
 & = & \left(\int_{G}h_{0}(t)\, dt\right)^{-1}\left\Vert \int_{G}h_{0}(r^{-1})\alpha_{s}^{-1}(f(sr))u_{i_{0}}\, dr\right\Vert \\
 & \geq & \left(\int_{G}h_{0}(t)\, dt\right)^{-1}\left\Vert \int_{G}h_{0}(r^{-1})\alpha_{s}^{-1}(f(s))u_{i_{0}}\, dr\right\Vert \\
 &  & \quad-\left(\int_{G}h_{0}(t)\, dt\right)^{-1}\left\Vert \int_{G}h_{0}(r^{-1})\alpha_{s}^{-1}(f(s)-f(sr))u_{i_{0}}\, dr\right\Vert \\
 & \geq & \left(\int_{G}h_{0}(t)\, dt\right)^{-1}\left(\int_{G}h_{0}(r)\, dr\right)\left\Vert \alpha_{s}^{-1}(f(s))u_{i_{0}}\right\Vert \\
 &  & \quad-\left(\int_{G}h_{0}(t)\, dt\right)^{-1}\frac{\varepsilon C_{\alpha}M\left(\int_{G}h_{0}(r)\, dr\right)}{2C_{\alpha}M}\\
 & \geq & \|\alpha_{s^{-1}}(f(s))\|-\frac{\varepsilon}{2}-\frac{\varepsilon}{2}\\
 & \geq & \frac{1}{C_{\alpha}}\|f(s)\|-\varepsilon.
\end{eqnarray*}
Hence, with $L:=\sup_{s\in\mbox{supp}(f)}\omega(s)$, which is finite
since $\omega$ is bounded on compact sets, 
\begin{eqnarray*}
\|\tilde{\lambda}\rtimes\Lambda(f)h\|_{1,\omega} & \geq & \int_{\mbox{supp}(f)}\|[\tilde{\lambda}\rtimes\Lambda(f)h](s)\|\omega(s)\, ds\\
 & \geq & \int_{\mbox{supp}(f)}\left(\frac{1}{C_{\alpha}}\|f(s)\|-\varepsilon\right)\omega(s)\, ds\\
 & \geq & \frac{1}{C_{\alpha}}\|f\|_{1,\omega}-\varepsilon\mu(\mbox{supp}(f))L.
\end{eqnarray*}
Now, since $\|h\|_{1,\omega}\leq M\sup_{r\in W}\omega(r)$, we obtain
\begin{eqnarray*}
\|\tilde{\lambda}\rtimes\Lambda(f)\| & \geq & \frac{1}{C_{\alpha}M\sup_{r\in W}\omega(r))}\|f\|_{1,\omega}-\frac{\varepsilon L}{M\sup_{r\in W}\omega(r)}\mu(\mbox{supp}(f)).
\end{eqnarray*}
Because $\varepsilon>0$ was chosen arbitrarily, $\|\tilde{\lambda}\rtimes\Lambda(f)\|\geq(C_{\alpha}M\sup_{r\in W}\omega(r))^{-1}\|f\|_{1,\omega}$
now follows. \end{proof}
We now combine our previous results, notably (\ref{eq:seminorm-bounded-by-beurlingnorm})
and Proposition \ref{pro:reverseinequality}, to obtain sufficient
conditions for a crossed product $\crossedprod$ to be isomorphic
to a generalized Beurling algebra, and also collect a number of direct
consequences in the following result. The desired reverse inequality
to (\ref{eq:seminorm-bounded-by-beurlingnorm}) is a consequence of
Proposition \ref{pro:reverseinequality}, supplying the first inequality
in (\ref{eq:inequalitychain}) and the second inequality in (\ref{eq:inequalitychain}),
which follows from the assumption that $(\tilde{\lambda},\Lambda)$
is $\mathcal{R}$-continuous.
\begin{thm}
\label{thm:faithful-representation}Let $(A,G,\alpha)$ be a Banach
algebra dynamical system where $A$ has an $M$-bounded approximate
right identity and $\alpha$ is uniformly bounded by a constant $C_{\alpha}$.
Let $\omega$ be a weight on $G$. Let $\mathcal{R}$ be a uniformly
bounded class of continuous covariant representations of $(A,G,\alpha)$
with $C^{\mathcal{R}}=\sup_{(\pi,U)\in\mathcal{R}}\|\pi\|<\infty$
and satisfying $\nu^{\mathcal{R}}(r)=\sup_{(\pi,U)\in\mathcal{R}}\|U_{r}\|\leq\omega(r)$
for all $r\in G$. Let $\lambda$ be the left regular representation
of $A$, and suppose that the non-degenerate continuous covariant
representation $(\tilde{\lambda},\Lambda)$ of $(A,G,\alpha)$ on
$L^{1}(G,A,\omega)$ \textup{(}as yielded by Definition \ref{def:induced_rep_and_translation_rep}\textup{)}
is $\mathcal{R}$-continuous. Then, for all $f\in C_{c}(G,A)$, with
$\mathcal{Z}$ denoting a neighbourhood base of $e\in G$ of which
all elements are contained in a fixed compact set, 
\begin{align}
\left(\frac{1}{C_{\alpha}M\inf_{W\in\mathcal{Z}}\sup_{r\in W}\omega(r)}\right)\|f\|_{1,\omega}\leq\|\tilde{\lambda}\rtimes\Lambda(f)\|\quad\quad\quad\quad\quad\quad\label{eq:inequalitychain}\\
\leq\|\tilde{\lambda}\rtimes\Lambda\|\sigma^{\mathcal{R}}(f)\leq\|\tilde{\lambda}\rtimes\Lambda\|C^{\mathcal{R}}\|f\|_{1,\omega}.\nonumber 
\end{align}
In particular, $\sigma^{\mathcal{R}}$ is a norm on $C_{c}(G,A)$,
so that $C_{c}(G,A)$ can be identified with a subspace of $\crossedprod$.
Since the norms $\sigma^{\mathcal{R}}$ and $\|\cdot\|_{1,\omega}$
on $C_{c}(G,A)$ are equivalent, there exists a topological isomorphism
between the Banach algebra $\crossedprod$ and the generalized Beurling
algebra $\BeurlingTypeAlg$ that is the identity on $C_{c}(G,A)$. 

The multiplication on the common dense subspace $C_{c}(G,A)$ of the
spaces  $\crossedprod$ and\textup{ $\BeurlingTypeAlg$} is given
by 
\[
[f*g](s):=\int_{G}f(r)\alpha_{r}(g(r^{-1}s))\, dr\quad(f,g\in C_{c}(G,A),\ s\in G).
\]

The faithful representation $\tilde{\lambda}\rtimes\Lambda:C_{c}(G,A)\to B(L^{1}(G,A,\omega))$
extends to a topological embedding $(\tilde{\lambda}\rtimes\Lambda)^{\mathcal{R}}:\crossedprod\to B(L^{1}(G,A,\omega))$
of the Banach algebra $\crossedprod$ into $B(L^{1}(G,A,\omega))$.
\end{thm}
Using Corollary \ref{cor:pitilde-Lambda-covariant-non-deg}, we have
the following consequence of Theorem \ref{thm:faithful-representation},
where the isomorphism between $\crossedprod$ and $\BeurlingTypeAlg$
is isometric.
\begin{cor}
\label{cor:crossedprod_is_Beulringalgebra}Let $(A,G,\alpha)$ be
a Banach algebra dynamical system where $A$ has a $1$-bounded approximate
right identity and $\alpha$ lets $G$ act as isometries on $A$.
Let $\omega$ be a weight on $G$, and $\lambda$ the left regular
representation of $A$. Then the non-degenerate continuous covariant
representation $(\tilde{\lambda},\Lambda)$ on $L^{1}(G,A,\omega)$
\textup{(}as yielded by Definition \ref{def:induced_rep_and_translation_rep}\textup{)}
is such that $\tilde{\lambda}$ is contractive and $\Lambda$ is bounded
by $\omega$. 

Suppose furthermore that $\inf_{W\in\mathcal{Z}}\sup_{r\in W}\omega(r)=1$,
with $\mathcal{Z}$ denoting a neighbourhood base of $e\in G$ of
which all elements are contained in a fixed compact set, and that
$\mathcal{R}$ is a uniformly bounded class of continuous covariant
representations with $(\tilde{\lambda},\Lambda)\in\mathcal{R}$, and
satisfying 
\[
C^{\mathcal{R}}=\sup_{(\pi,U)\in\mathcal{R}}\|\pi\|\leq1,
\]
 and 
\[
\nu^{\mathcal{R}}(r)=\sup_{(\pi,U)\in\mathcal{R}}\|U_{r}\|\leq\omega(r)\quad(r\in G).
\]
Then $\sigma^{\mathcal{R}}(f)=\|f\|_{1,\omega}$ for $f\in C_{c}(G,A)$,
and hence $\crossedprod$ is isometrically isomorphic to the generalized
Beurling algebra $\BeurlingTypeAlg$. 

Moreover, \textup{$(\tilde{\lambda}\rtimes\Lambda)^{\mathcal{R}}:\crossedprod\to B(L^{1}(G,A,\omega))$}
is an isometric embedding as a Banach algebra.\end{cor}
\begin{proof}
Since $(\tilde{\lambda},\Lambda)\in\mathcal{R}$, we have $\|\tilde{\lambda}\rtimes\Lambda\|\leq1$,
and by hypothesis $C^{\mathcal{R}}\leq1$. Therefore, by Theorem \ref{thm:faithful-representation},
for every $f\in C_{c}(G,A)$, 
\[
\|f\|_{1,\omega}\leq\|\tilde{\lambda}\rtimes\Lambda(f)\|\leq\|\tilde{\lambda}\rtimes\Lambda\|\sigma^{\mathcal{R}}(f)\leq C^{\mathcal{R}}\|\tilde{\lambda}\rtimes\Lambda\|\|f\|_{1,\omega}\leq\|f\|_{1,\omega}.
\]
We conclude that $C^{\mathcal{R}}=\|\tilde{\lambda}\rtimes\Lambda\|=1$,
and the result now follows.\end{proof}
\begin{rem}
Certainly if the weight $\omega:G\to[0,\infty)$ is continuous in
$e\in G$ and $\omega(e)=1$, then $\inf_{W\in\mathcal{Z}}\sup_{r\in W}\omega(r)=1$,
for example if $\omega$ is taken to be a continuous positive character
of $G$.
\end{rem}
{}
\begin{rem}
\label{rem:induced-related-to-left-regular}We note that the representation
\[
(\tilde{\lambda}\rtimes\Lambda)^{\mathcal{R}}:\BeurlingTypeAlg\to B(\BeurlingTypeAlg)
\]
 does not equal the left regular representation of $\BeurlingTypeAlg$
in general, but they are always conjugate. To see this, define, for
$h\in C_{c}(G,A)$ and $s\in G$, $\check{h}(s):=\alpha_{s^{-1}}(h(s))$,
$\hat{h}(s):=\alpha_{s}(h(s))$. Then $\hat{\cdot}:C_{c}(G,A)\to C_{c}(G,A)$
and $\check{\cdot}:C_{c}(G,A)\to C_{c}(G,A)$ are mutual inverses
and, since $\alpha$ is uniformly bounded, extend to mutually inverse
Banach space isomorphisms of $\BeurlingTypeAlg$ onto itself. Then
$(\tilde{\lambda}\rtimes\Lambda)^{\mathcal{R}}$ and the left regular
representation $\lambda$ of $\BeurlingTypeAlg$ are conjugate under
$\hat{\cdot}$. Indeed by Lemma \ref{lem:point-evaluations-pulled-through-integral},
for $f,h\in C_{c}(G,A)$ and $s\in G$, 
\begin{eqnarray*}
\left(\tilde{\lambda}\rtimes\Lambda(f)\check{h}\right)^{\wedge}(s) & = & \alpha_{s}\left([\tilde{\lambda}\rtimes\Lambda(f)\check{h}](s)\right)\\
 & = & \alpha_{s}\left(\int_{G}\alpha_{s}^{-1}(f(r))\check{h}(r^{-1}s)\, dr\right)\\
 & = & \alpha_{s}\left(\int_{G}\alpha_{s}^{-1}(f(r))\alpha_{s^{-1}r}(h(r^{-1}s))\, dr\right)\\
 & = & \int_{G}f(r)\alpha_{r}(h(r^{-1}s))\, dr\\
 & = & [f*h](s)\\
 & = & [\lambda(f)h](s).
\end{eqnarray*}
Hence $(\tilde{\lambda}\rtimes\Lambda)^{\mathcal{R}}$ and the left
regular representation 
\[
\lambda:\BeurlingTypeAlg\to B(\BeurlingTypeAlg)
\]
 of $\BeurlingTypeAlg$ are conjugate as claimed. Note that $\hat{\cdot}$
is the identity if $\alpha=\textup{triv}$, hence in that case $(\tilde{\lambda}\rtimes\Lambda)^{\mathcal{R}}=\lambda$.
\end{rem}
We continue the main line with the following trivial but important
observation: If $(\tilde{\lambda},\Lambda)\in\mathcal{R}$, for example,
by taking $\mathcal{R}:=\{(\tilde{\lambda},\Lambda)\}$, then certainly
$(\tilde{\lambda},\Lambda)$ is $\mathcal{R}$-continuous, hence the
conclusions in Theorem \ref{thm:faithful-representation} hold, and
in particular the algebras $\crossedprod$ and $\BeurlingTypeAlg$
are topologically isomorphic. A similar remark is applicable to Corollary
\ref{cor:crossedprod_is_Beulringalgebra}, giving sufficient conditions
for the mentioned topological isomorphism to be isometric. Hence we
have the following: 
\begin{thm}
\label{thm:Choosing-R-correctly-Crossed-Products-are-beurling}Let
$(A,G,\alpha)$ be a Banach algebra dynamical system where $A$ has
a bounded approximate right identity and $\alpha$ is uniformly bounded.
Let $\omega$ be a weight on $G$ and let the non-degenerate continuous
covariant representation $(\tilde{\lambda},\Lambda)$ of $(A,G,\alpha)$
on $L^{1}(G,A,\omega)$ be as yielded by Definition \ref{def:induced_rep_and_translation_rep}\textup{.}
Then the generalized Beurling algebra $\BeurlingTypeAlg$ and the
crossed product $\crossedprod$ with $\mathcal{R}:=\{(\tilde{\lambda},\Lambda)\}$
are topologically isomorphic via an isomorphism that is the identity
on $C_{c}(G,A)$.

Furthermore, the map $\tilde{\lambda}\rtimes\Lambda:C_{c}(G,A)\to B(L^{1}(G,A,\omega))$
extends to a topological embedding of $\BeurlingTypeAlg$ into $B(L^{1}(G,A,\omega))$.

If $A$ has a $1$-bounded two-sided approximate identity, $\alpha$
lets $G$ act as isometries on $A$ and \textup{$\inf_{W\in\mathcal{Z}}\sup_{r\in W}\omega(r)=1$,
}with $\mathcal{Z}$ denoting a neighbourhood base of $e\in G$ of
which all elements are contained in a fixed compact set\textup{,}
then the isomorphism between $\crossedprod$ and $\BeurlingTypeAlg$
is an isometry, and the embedding of $\BeurlingTypeAlg$ into $B(L^{1}(G,A,\omega))$
is isometric.\end{thm}
\begin{rem}
As noted in Remark \ref{rem:induced-related-to-left-regular}, when
$\alpha=\textup{triv}$, then $(\tilde{\lambda}\rtimes\Lambda)^{\mathcal{R}}$
equals the left regular representation $\lambda:\BeurlingTypeAlg\to B(\BeurlingTypeAlg)$
of $\BeurlingTypeAlg$.
\end{rem}
{}
\begin{rem}
We note that, for $(A,G,\alpha)=(\mathbb{K},G,\textup{triv})$, the
second part of Theorem \ref{thm:Choosing-R-correctly-Crossed-Products-are-beurling}
asserts that $(\mathbb{K}\rtimes_{\textup{triv}}G)^{\mathcal{R}}$
is isometrically isomorphic to the classical Beurling algebra $L^{1}(G,\omega)$,
provided that $\inf_{W\in\mathcal{Z}}\sup_{r\in W}\omega(r)=1$ (which
is certainly true if $\omega$ is continuous at $e\in G$ and $\omega(e)=1$).
In particular $L^{1}(G)$ is isometrically isomorphic to a crossed
product. Under the condition $\inf_{W\in\mathcal{Z}}\sup_{r\in W}\omega(r)=1$,
combining Remark \ref{rem:induced-related-to-left-regular} and Theorem
\ref{thm:Choosing-R-correctly-Crossed-Products-are-beurling} also
shows that the left regular representation of $L^{1}(G,\omega)$ is
an isometric embedding of $L^{1}(G,\omega)$ into $B(L^{1}(G,\omega))$. 
\end{rem}
Hence, provided that $A$ has a bounded approximate right identity,
the generalized Beurling algebras $\BeurlingTypeAlg$, and in particular
(when $A=\mathbb{K}$) the classical Beurling algebras $L^{1}(G,\omega)$,
are isomorphic to a crossed product associated with a Banach algebra
dynamical system. Therefore, in the case where the algebra $A$ has
a two-sided identity, the General Correspondence Theorem (Theorem
\ref{thm:General-Correspondence-Theorem}) determines the non-degenerate
bounded representations of generalized Beurling algebras. This we
will elaborate on in the rest of the section. In cases where the algebra
is trivial, i.e., $A=\mathbb{K}$, we regain classical results on
the representation theory of $L^{1}(G)$ and other classical Beurling
algebras. 

Assume, in addition to the hypothesis in Theorem \ref{thm:faithful-representation},
that $A$ has an $M$-bounded two-sided approximate identity and that
all continuous covariant representations in $\mathcal{R}$ are non-degenerate.
In that case, we claim that the non-degenerate $\mathcal{R}$-continuous
covariant representations are precisely the non-degenerate continuous
covariant representations $(\pi,U)$ of $(A,G,\alpha)$, with no further
restriction on $\pi$, but with $U$ such that $\|U_{r}\|\leq C_{U}\omega(r)$
for all $r\in G$ and a $U$-dependent constant $C_{U}$. To see this,
we start by noting that, for $f\in C_{c}(G,A)$,
\begin{eqnarray*}
\|\pi\rtimes U(f)\| & \leq & \int_{G}\|\pi(f(r))\|\|U_{r}\|\, dr\\
 & \leq & \int_{G}\|\pi\|\|f(r)\|C_{U}\omega(r)\, dr\\
 & \leq & C_{U}\|\pi\|\int_{G}\|f(r)\|\omega(r)\, dr\\
 & = & C_{U}\|\pi\|\|f\|_{1,\omega}\\
 & \leq & C_{(\pi,U)}'\sigma^{\mathcal{R}}(f)
\end{eqnarray*}
for some $C_{(\pi,U)}'\geq0$, since $\|\cdot\|_{1,\omega}$ and $\sigma^{\mathcal{R}}$
are equivalent.

For the converse, we use that $A$ has a bounded approximate left
identity and that $\mathcal{R}$ consists of non-degenerate continuous
covariant representations. If $(\pi,U)$ is a non-degenerate $\mathcal{R}$-continuous
representation of $(A,G,\alpha)$, then the General Correspondence
Theorem (Theorem \ref{thm:General-Correspondence-Theorem}) asserts
that 
\[
(\pi,U)=(\overline{(\pi\rtimes U)^{\mathcal{R}}}\circ i_{A}^{\mathcal{R}},\overline{(\pi\rtimes U)^{\mathcal{R}}}\circ i_{G}^{\mathcal{R}}),
\]
where $\overline{(\pi\rtimes U)^{\mathcal{R}}}:\leftcent(\crossedprod)\to B(X_{\pi})$
is the non-degenerate bounded representation induced by the non-degenerate
bounded representation $(\pi\rtimes U)^{\mathcal{R}}:\crossedprod\to B(X_{\pi})$.
However if $T:\crossedprod\to B(X)$ is any non-degenerate bounded
representation, then \cite[Proposition 7.1]{2011arXiv1104.5151D}
asserts that there exists a constant $C_{T}:=M_{l}^{\mathcal{R}}\|T\|$,
with $M_{l}^{\mathcal{R}}$ a bound for a bounded approximate left
identity in $\crossedprod$, such that 
\begin{equation}
\|\overline{T}\circ i_{G}^{\mathcal{R}}(r)\|\leq C_{T}\nu^{\mathcal{R}}(r)\leq C_{T}\omega(r)\quad(r\in G).\label{eq:group-rep-bound}
\end{equation}
Therefore, $r\mapsto\|U_{r}\|$ is bounded by a multiple of $\omega$,
as claimed.

We now take $\mathcal{R}:=\{(\tilde{\lambda},\Lambda)\}$ as in Theorem
\ref{thm:Choosing-R-correctly-Crossed-Products-are-beurling}. Theorem
\ref{thm:Choosing-R-correctly-Crossed-Products-are-beurling} shows
that the non-degenerate bounded representations of $\BeurlingTypeAlg$
can be identified with those of $\crossedprod$. By the General Correspondence
Theorem (Theorem \ref{thm:General-Correspondence-Theorem}) the latter
are in natural bijection with the non-degenerate $\mathcal{R}$-continuous
covariant representations of $(A,G,\alpha)$ and these we have just
described. Hence the non-degenerate bounded representations of $\BeurlingTypeAlg$
are in natural bijection with pairs $(\pi,U)$ as above. Furthermore,
slightly simplified versions of \cite[Equations (8.1) and (8.2)]{2011arXiv1104.5151D}
(cf.\,Remark \ref{rem:eq-8.1and8.2-can-be-simplified-by-deleting-alpha})
give explicit formulas for retrieving $(\pi,U)$ from a non-degenerate
bounded representation $T$ of $\crossedprod\simeq\BeurlingTypeAlg$.
Combining all this, we obtain the following correspondence between
the non-degenerate continuous covariant representations of $(A,G,\alpha)$
and the non-degenerate bounded representations of the generalized
Beurling algebra $\BeurlingTypeAlg$:
\begin{thm}
\label{thm:continuous-non-deg-covars-are-R-continuous}Let $(A,G,\alpha)$
be a Banach algebra dynamical system where $A$ has a two-sided approximate
identit\textup{y} and $\alpha$ is uniformly bounded by $C_{\alpha}$.
Let $\omega$ be a weight on $G$. Then the following maps are mutual
inverses between the non-degenerate continuous covariant representations
$(\pi,U)$ of $(A,G,\alpha)$ on a Banach space\textup{ $X$}, satisfying
$\|U_{r}\|\leq C_{U}\omega(r)$ for some $C_{U}\geq0$ and all $r\in G$,
and the non-degenerate bounded representations $T:\BeurlingTypeAlg\to B(X)$
of the generalized Beurling algebra $\BeurlingTypeAlg$ on $X$: 
\[
(\pi,U)\mapsto\left(f\mapsto\int_{G}\pi(f(r))U_{r}\, dr\right)=:T^{(\pi,U)}\quad(f\in C_{c}(G,A)),
\]
determining a non-degenerate bounded representation $T^{(\pi,U)}$
of the generalized Beurling algebra $\BeurlingTypeAlg$, and, 
\[
T\mapsto\left(\begin{array}{l}
a\mapsto\textup{SOT-lim}_{(V,i)}T(z_{V}\otimes au_{i}),\\
s\mapsto\textup{SOT-lim}_{(V,i)}T(z_{V}(s^{-1}\cdot)\otimes u_{i})
\end{array}\right)=:(\pi^{T},U^{T}),
\]
where $\mathcal{Z}$ is a neighbourhood base of $e\in G$, of which
all elements are contained in a fixed compact subset of $G$, $z_{V}\in C_{c}(G,A)$
is chosen such that $z_{V}\geq0$, supported in $V\in\mathcal{Z}$,
$\int_{G}z_{V}(r)dr=1$, and $(u_{i})$ is any bounded approximate
left identity of $A$.

Furthermore, if $A$ has an $M$-bounded approximate left identity,
then the following bounds for $T^{(\pi,U)}$ and $(\pi^{T},U^{T})$
hold:
\begin{enumerate}
\item $\|T^{(\pi,U)}\|\leq C_{U}\|\pi\|$,
\item \textup{$\|\pi^{T}\|\leq\left(\inf_{V\in\mathcal{Z}}\sup_{r\in V}\omega(r)\right)\|T\|$,}
\item $\|U_{s}^{T}\|\leq M\left(\inf_{V\in\mathcal{Z}}\sup_{r\in V}\omega(r)\right)\|T\|\,\omega(s)\quad(s\in G)$.
\end{enumerate}
\end{thm}
\begin{proof}
Except for the claimed bounds for $\|T^{U}\|$, $\|\pi^{T}\|$ and
$\|U^{T}\|$, all statements have been proven preceding the statement
of the theorem. We will now establish these three bounds.

We prove (1). Let $(\pi,U)$ be a non-degenerate continuous covariant
representations of $(A,G,\alpha)$ on a Banach space $X$, satisfying
$\|U_{r}\|\leq C_{U}\omega(r)$ for some $C_{U}\geq0$ and all $r\in G$.
Then, for any $f\in C_{c}(G,A)$,

\begin{eqnarray*}
\|T^{(\pi,U)}(f)\| & = & \left\Vert \int_{G}\pi(f(r))U_{r}\, dr\right\Vert \\
 & \leq & \int_{G}\|\pi\|\|f(r)\|\|U_{r}\|\, dr\\
 & \leq & \|\pi\|C_{U}\int_{G}\|f(r)\|\omega(r)\, dr\\
 & = & \|\pi\|C_{U}\|f\|_{1,\omega}.
\end{eqnarray*}
Therefore $\|T^{(\pi,U)}\|\leq\|\pi\|C_{U}$. 

We prove (2). Let $T:\BeurlingTypeAlg\to B(X)$ be a non-degenerate
bounded representations of the generalized Beurling algebra $\BeurlingTypeAlg$
on $X$. Choose a bounded two-sided approximate identity $(u_{i})$
of $A$. Then, for any $a\in A$,

\begin{eqnarray*}
\|T(z_{V}\otimes au_{i})\| & \leq & \|T\|\|z_{V}\otimes au_{i}\|_{1,\omega}\\
 & \leq & \|T\|\int_{G}z_{V}(r)\|au_{i}\|\omega(r)\, dr\\
 & = & \|T\|\|au_{i}\|\int_{G}z_{V}(r)\omega(r)\, dr\\
 & \leq & \|T\|\|au_{i}\|\sup_{r\in V}\omega(r)\int_{G}z_{V}(r)\, dr\\
 & = & \|T\|\|au_{i}\|\sup_{r\in V}\omega(r).
\end{eqnarray*}
Since, in particular, $(u_{i})$ is an approximate right identity
of $A$, for any $\varepsilon_{1}>0$, there exists an index $i_{0}$
such that $i\geq i_{0}$ implies $\|au_{i}\|\leq\|a\|+\varepsilon_{1}$.
Also, for any $\varepsilon_{2}>0$, there exists some $V_{0}\in\mathcal{Z}$
such that $\sup_{r\in V_{0}}\omega(r)\leq\inf_{V\in\mathcal{Z}}\sup_{r\in V}\omega(r)+\varepsilon_{2}$.
Now, if $(V,i)\geq(V_{0},i_{0})$, then $V_{0}\supseteq V$ and $i\geq i_{0}$,
and hence 
\begin{eqnarray*}
\|T(z_{V}\otimes au_{i})\| & \leq & \|T\|\|au_{i}\|\sup_{r\in V}\omega(r)\\
 & \leq & \|T\|\|au_{i}\|\sup_{r\in V_{0}}\omega(r)\\
 & \leq & \|T\|(\|a\|+\varepsilon_{1})\left(\inf_{V\in\mathcal{Z}}\sup_{r\in V}\omega(r)+\varepsilon_{2}\right).
\end{eqnarray*}
Therefore, if $x\in X$, then 
\begin{eqnarray*}
\|\pi^{T}(a)x\| & = & \lim_{(V,i)}\|T(z_{V}\otimes au_{i})x\|\\
 & = & \lim_{(V,i)\geq(V_{0},i_{0})}\|T(z_{V}\otimes au_{i})x\|\\
 & \leq & \|T\|(\|a\|+\varepsilon_{1})\left(\inf_{V\in\mathcal{Z}}\sup_{r\in V}\omega(r)+\varepsilon_{2}\right)\|x\|.
\end{eqnarray*}
Since $\varepsilon_{1}$ and $\varepsilon_{2}$ we chosen arbitrarily,
$\|\pi^{T}\|\leq\|T\|\left(\inf_{V\in\mathcal{Z}}\sup_{r\in V}\omega(r)\right)$
now follows. 

We prove (3). Let $(u_{i})$ be an $M$-bounded approximate left identity
of $A$. Fix $s\in G$. Let $\varepsilon>0$ be arbitrary and let
$V_{0}\in\mathcal{Z}$ be such that $\sup_{r\in V_{0}}\omega(r)\leq\inf_{V\in\mathcal{Z}}\sup_{r\in V}\omega(r)+\varepsilon$.
Fix some index $i_{0}$, then, for every $(V,i)\geq(V_{0},i_{0})$, 

\begin{eqnarray*}
\|T(z_{V}(s^{-1}\cdot)\otimes u_{i})\| & \leq & \|T\|\|z_{V}(s^{-1}\cdot)\otimes u_{i}\|_{1,\omega}\\
 & = & \|T\|\int_{G}z_{V}(s^{-1}r)\|u_{i}\|\omega(r)\: dr\\
 & \leq & M\|T\|\int_{G}z_{V}(r)\omega(sr)\: dr\\
 & \leq & M\|T\|\int_{G}z_{V}(r)\omega(s)\omega(r)\: dr\\
 & = & M\|T\|\omega(s)\int_{V}z_{V}(r)\omega(r)\: dr\\
 & \leq & M\|T\|\left(\sup_{r\in V_{0}}\omega(r)\right)\omega(s)\int_{V}z_{V}(r)\: dr\\
 & \leq & M\|T\|\left(\inf_{V\in\mathcal{Z}}\sup_{r\in V}\omega(r)+\varepsilon\right)\omega(s).
\end{eqnarray*}
Therefore, if $x\in X$, then
\begin{eqnarray*}
\|U_{s}^{T}x\| & = & \lim_{(V,i)}\|T(z_{V}(s^{-1}\cdot)\otimes u_{i})x\|\\
 & = & \lim_{(V,i)\geq(V_{0},i_{0})}\|T(z_{V}(s^{-1}\cdot)\otimes u_{i})x\|\\
 & \leq & M\|T\|\left(\inf_{V\in\mathcal{Z}}\sup_{r\in V}\omega(r)+\varepsilon\right)\omega(s)\|x\|.
\end{eqnarray*}
Since $\varepsilon>0$ was chosen arbitrarily, $\|U_{r}^{T}\|\leq M\|T\|\left(\inf_{V\in\mathcal{Z}}\sup_{r\in V}\omega(r)\right)\omega(s)$
now follows.\end{proof}
\begin{rem}
\label{rem:eq-8.1and8.2-can-be-simplified-by-deleting-alpha}Our reconstruction
formulas in Theorem \ref{thm:continuous-non-deg-covars-are-R-continuous}
differs slightly from those given in \cite[Equations (8.1) and (8.2)]{2011arXiv1104.5151D},
where the reconstruction formula for $U^{T}$ is given as 
\begin{equation}
s\mapsto\textup{SOT-lim}_{(V,i)}T(z_{V}(s^{-1}\cdot)\otimes\alpha_{s}(u_{i})),\label{eq:CPI-group-reconstruction}
\end{equation}
with $(u_{i})$ any bounded approximate left identity of $A$. However,
if $(u_{i})$ is any bounded approximate left identity of $A$ and
$s\in G$ is fixed, then $(\alpha_{s^{-1}}(u_{i}))$ is also a bounded
approximate left identity of $A$, and using this particular choice
in (\ref{eq:CPI-group-reconstruction}) gives the formula in Theorem
\ref{thm:continuous-non-deg-covars-are-R-continuous}.
\end{rem}
For the Banach algebra dynamical system $(\mathbb{K},G,\textup{triv})$
and weight $\omega$ on $G$, Theorem \ref{thm:continuous-non-deg-covars-are-R-continuous}
simplifies. We collect the statements from Theorem \ref{thm:continuous-non-deg-covars-are-R-continuous}
concerning representations and some material from Remark \ref{rem:induced-related-to-left-regular},
Corollary \ref{cor:crossedprod_is_Beulringalgebra} in the following
result, which contains a few classical results as special cases: For
one-dimensional representations, the result reduces to the bijection
between $\omega$-bounded characters of $G$ and multiplicative functionals
of the Beurling algebra $L^{1}(G,\omega)$, see, e.g., \cite[Theorem 2.8.2]{Kaniuth}
(where, contrary to our general groups, $G$ is assumed to be abelian).
In the case where $\omega$ is the constant $1$, the result reduces
to the classical bijection between uniformly bounded strongly continuous
representations of $G$ and non-degenerate bounded representations
of $L^{1}(G)$, see, e.g., \cite[Assertion VI.1.32]{Helemski}.
\begin{cor}
\label{cor:classical-L1-result}Let $\omega$ be a weight on $G$.
With $(z_{V})$ as in Theorem \ref{thm:continuous-non-deg-covars-are-R-continuous},
the maps 
\[
U\mapsto\left(f\mapsto\int_{G}f(r)U_{r}\, dr\right)=:T^{U}\quad(f\in C_{c}(G)),
\]
determining a non-degenerate bounded representation $T^{U}$\!\!
of the Beurling algebra $L^{1}(G,\omega)$, and 
\[
T\mapsto\left(s\mapsto\textup{SOT-lim}_{V}T(z_{V}(s^{-1}\cdot))\right)=:U^{T}
\]
are mutual inverses between the strongly continuous group representations
$U$ of $G$ on a Banach space\textup{ $X$}, satisfying $\|U_{r}\|\leq C_{U}\omega(r)$,
for some $C_{U}\geq0$ and all $r\in G$, and the non-degenerate bounded
representations $T:L^{1}(G,\omega)\to B(X)$ of the Beurling algebra
$L^{1}(G,\omega)$ on $X$\textup{).}

If the weight satisfies $\inf_{W\in\mathcal{Z}}\sup_{r\in W}\omega(r)=1$,
where $\mathcal{Z}$ is a neighbourhood base of $e\in G$, of which
all elements are contained in a fixed compact subset of $G$, then
$\|T^{U}\|=\sup_{r\in G}\|U_{r}\|/\omega(r)$ and $\|U_{r}^{T}\|\leq\|T\|\omega(r)$
for all $r\in G$.\end{cor}
\begin{proof}
The only statement that does not follow directly from Theorem \ref{thm:continuous-non-deg-covars-are-R-continuous}
is that $\|T^{U}\|=\sup_{r\in G}\|U_{r}\|/\omega(r)$, when $\sup_{W\in\mathcal{Z}}(\sup_{r\in W}\omega(r))^{-1}=1$.

To establish this, we note that 
\[
\|U_{r}\|=\omega(r)\frac{\|U_{r}\|}{\omega(r)}\leq\left(\sup_{s\in G}\frac{\|U_{s}\|}{\omega(s)}\right)\omega(r).
\]
 Therefore, we can replace $C_{U}$ with $\sup_{r\in G}\|U_{r}\|/\omega(r)$,
and, by the bound (1) in Theorem \ref{thm:continuous-non-deg-covars-are-R-continuous},
$\|T^{U}\|\leq\sup_{r\in G}\|U_{r}\|/\omega(r)$. The reverse inequality
follows from (3) in Theorem \ref{thm:continuous-non-deg-covars-are-R-continuous},
when noting that the maps $U\mapsto T^{U}$ and $T\mapsto U^{T}$
are mutual inverses.\end{proof}
\begin{rem}
For one-dimensional representations, Corollary \ref{cor:classical-L1-result}
implies that continuous characters $\chi:G\to\mathbb{C}^{\times}$
of $G$, such that $|\chi(r)|\leq C_{\chi}\omega(r)$ for some $C_{\chi}$
and all $r\in G$, are in natural bijection with the one-dimensional
representations of $L^{1}(G,\omega)$. Since this is a Banach algebra,
such representations are contractive, and the final part of Corollary
\ref{cor:classical-L1-result} then asserts that one can actually
take $C_{\chi}=1$ (cf.\,\cite[Lemma 2.8.2]{Kaniuth} for abelian
$G$). One can also verify this directly by noting that, if there
exists some $s\in G$ for which $|\chi(s)|>\omega(s)$, then, for
all $n\in\mathbb{N}$, by submultiplicativity of $\omega$, 
\[
\left(\frac{|\chi(s)|}{\omega(s)}\right)^{n}=\frac{|\chi(s^{n})|}{\omega(s)^{n}}\leq C_{\chi}\frac{\omega(s^{n})}{\omega(s)^{n}}\leq C_{\chi}.
\]
Therefore, since $|\chi(s)|>\omega(s)$, we must have that $C_{\chi}=\infty$,
which is absurd. Hence $|\chi(r)|\leq\omega(r)$ for all $r\in G$.
\end{rem}

\section{Other types for $(\pi,U)$\label{sec:Other-types}}

For a given Banach algebra dynamical system $(A,G,\alpha)$ we have
thus far been concerned with a uniformly bounded class of pairs $(\pi,U)$,
where $\pi:A\to B(X)$ and $U:G\to B(X)$ are multiplicative representations,
$U$ is strongly continuous, and satisfy the covariance condition
\[
U_{r}\pi(a)U_{r}^{-1}=\pi(\alpha_{r}(a))
\]
 for all $r\in G$ and $a\in A$. On the other hand, in \cite[Proposition 6.5]{2011arXiv1104.5151D},
we have encountered an example of a pair $(\pi,U)$ where $\pi$ and
$U$ are both anti-multiplicative and satisfy the anti-covariance
condition
\[
U_{r}\pi(a)U_{r}^{-1}=\pi(\alpha_{r^{-1}}(a))
\]
 for all $r\in G$ and $a\in A$. Suppose one has a uniformly bounded
class $\mathcal{R}$ of such pairs $(\pi,U)$, with $U$ strongly
continuous, $\pi$ non-degenerate and that $A$ has a bounded ``appropriately
sided'' approximate identity, can one then find a Banach algebra of
crossed product type again, such that its non-degenerate bounded (perhaps
anti-) representations are in natural bijection with the $\mathcal{R}$-continuous
pairs $(\rho,V)$, satisfying the aforementioned requirements for
elements of $\mathcal{R}$? What about pairs $(\pi,U)$ where $\pi$
is multiplicative, $U$ is anti-multiplicative and a covariance condition
is satisfied? Can one, to ask a more fundamental question, expect
a meaningful theory to exist for such pairs?

In this section we address these matters. We start by determining
what appears to be the natural ``reasonable'' requirements in this
vein on $(\pi,U)$ for a meaningful theory to exist (and which are
not met in the second-mentioned example). There turn out to be four
cases. For each case we indicate a Banach algebra dynamical system
$(B,H,\beta)$ such that $B=A$ and $H=G$ as sets, and such that
the given maps $\pi:B\to B(X)$ and $U:H\to B(X)$ are now multiplicative
and satisfy a covariance condition. This brings us back into the realm
of the correspondence as in Theorem \ref{thm:General-Correspondence-Theorem}
or \cite[Theorem 8.1]{2011arXiv1104.5151D}, but we leave it to the
reader to formulate the resulting correspondence theorem for the other
three types of uniformly bounded classes of non-degenerate continuous
pairs $(\pi,U)$. 

After this, we turn to actions of $A$ and $G$ on $C_{c}(G,A)$.
While this is not, in general, a Banach space, several Banach spaces
are naturally obtained from $C_{c}(G,A)$ via quotients and/or completions,
hence it is for this space that we list sixteen canonical pairs of
actions, with each of the four ``reasonable'' properties occurring
four times. We then explain that, even though the formulas look quite
different, there is essentially only one pair, and the fifteen others
can be derived from it. We conclude with natural pairs $(\pi,U)$
of commuting actions on $C_{c}(G,A)$. 

This section is, in a sense, elementary and almost entirely algebraic
in nature. Nevertheless, we thought it worthwhile to make a systematic
inventorization, once and for all, of the ``reasonable'' properties
of pairs $(\pi,U)$, the natural actions on $A$-valued function spaces
on $G$, and the interrelations between the various formulas. A particular
case of the results in the present section will be instrumental in
Section \ref{sec:Beurling-Right-and-bimodules} where we explain how
non-degenerate right-- and bimodules over generalized Beurling algebras
fit into the general framework of crossed products of Banach algebras.

To start with, let $(A,G,\alpha)$ be a Banach algebra dynamical system.
What are the ``reasonable'' properties of $(\pi,U)$ that can lead
to a meaningful theory? Let us assume that $\pi:A\to B(X)$ is linear
and multiplicative or anti-multiplicative, that $U:G\to B(X)$ is
a multiplicative or anti-multiplicative map of $G$ into the group
of invertible elements of $B(X)$, and that 
\begin{equation}
U_{r}\pi(a)U_{r}^{-1}=\pi(\delta_{r}(a))\label{eq:anti?covariance}
\end{equation}
 for all $a\in A$ and $r\in G$, where $\delta$ is a multiplicative
or anti-multiplicative map from $G$ into the automorphisms or anti-automorphisms
of $A$. This is ``asking for the most general setup''. We start by
arguing that $\delta$ should map $G$ into the automorphisms of $A$.
Indeed, if $\pi$ is multiplicative, $r\in G$ and $a_{1},a_{2}\in A$,
then 
\begin{eqnarray*}
\pi(\delta_{r}(a_{1}a_{2})) & = & U_{r}\pi(a_{1}a_{2})U_{r}^{-1}\\
 & = & U_{r}\pi(a_{1})U_{r}^{-1}U_{r}\pi(a_{2})U_{r}^{-1}\\
 & = & \pi(\delta_{r}(a_{1}))\pi(\delta_{r}(a_{2}))\\
 & = & \pi(\delta_{r}(a_{1})\delta_{r}(a_{2})).
\end{eqnarray*}
If $\pi$ is anti-multiplicative, then again
\begin{eqnarray*}
\pi(\delta_{r}(a_{1}a_{2})) & = & U_{r}\pi(a_{1}a_{2})U_{r}^{-1}\\
 & = & U_{r}\pi(a_{2})U_{r}^{-1}U_{r}\pi(a_{1})U_{r}^{-1}\\
 & = & \pi(\delta_{r}(a_{2}))\pi(\delta_{r}(a_{1}))\\
 & = & \pi(\delta_{r}(a_{1})\delta_{r}(a_{2})).
\end{eqnarray*}
Hence one is led to assume that $\delta$ maps $G$ into $\mbox{Aut}(A)$,
still leaving open the possible choice of $\delta:G\to\mbox{Aut}(A)$
being multiplicative or anti-multiplicative.

To continue, if $U$ is anti-multiplicative, then (\ref{eq:anti?covariance})
implies, for $a\in A$ and $r_{1},r_{2}\in G$,
\begin{eqnarray*}
\pi(\delta_{r_{1}r_{2}}(a)) & = & U_{r_{1}r_{2}}\pi(a)U_{r_{1}r_{2}}^{-1}\\
 & = & U_{r_{2}}U_{r_{1}}\pi(a)U_{r_{1}}^{-1}U_{r_{2}}^{-1}\\
 & = & \pi(\delta_{r_{2}}\circ\delta_{r_{1}}(a)).
\end{eqnarray*}
Therefore, unless one imposes a further relation between $\pi$ and
$U$, it seems that only the possibility that $\delta$ is also anti-multiplicative
will lead to a meaningful theory. Likewise, the multiplicativity of
$U$ ``implies'' that $\delta$ should be multiplicative. Using that
$\delta_{r}$ is multiplicative on $A$ for $r\in G$, it is easily
seen that the covariance condition yields no implications on the nature
of $\pi$. 

With $(A,G,\alpha)$ given, the relevant non-trivial choice for a
multiplicative $\delta$ is $\alpha,$ and for an anti-multiplicative
$\delta$ it is $\alpha^{o}$ where $\alpha_{r}^{o}:=\alpha_{r^{-1}}$
for all $r\in G$; the reason for this notation will become clear
in a moment. We will consider these non-trivial choices for $\delta$
first, and return to $\delta=\mbox{triv}$ later. 

Hence we have to consider four meaningful possibilities for a pair
$(\pi,U)$ and the relation between $\pi$ and $U$. If we let, e.g.,
$(a,m)$ denote the case where $\pi$ is anti-multiplicative and $U$
is multiplicative, then, for $(m,m)$ and $(a,m)$, one should require
\[
U_{r}\pi(a)U_{r}^{-1}=\pi(\alpha_{r}(a)),
\]
 and for $(m,a)$ and $(a,a)$, one should require 
\[
U_{r}\pi(a)U_{r}^{-1}=\pi(\alpha_{r^{-1}}(a))=\pi(\alpha_{r}^{o}(a))
\]
for all $a\in A$ and $r\in G$.

Now note that, with $G^{o}$ denoting the opposite group, $\alpha^{o}:G^{o}\to\mbox{Aut}(A)$
is a multiplicative strongly continuous map if $\alpha$ is. Therefore,
if $(A,G,\alpha)$ is a Banach algebra dynamical system, then so is
$(A,G^{o},\alpha^{o})$. Furthermore, if $A^{o}$ is the opposite
algebra, then $\mbox{Aut}(A)=\mbox{Aut}(A^{o})$. Therefore, if $(A,G,\alpha)$
is a Banach algebra dynamical system, so is $(A^{o},G,\alpha)$. Combining
these two, a Banach algebra dynamical system has a third natural companion
Banach algebra dynamical system, namely $(A^{o},G^{o},\alpha^{o})$.
In each of these three cases, the Banach algebra is $A$ as a set,
and the group is $G$ as a set. Hence the given maps $\pi:A\to B(X)$
and $U:G\to B(X)$ can be viewed unaltered as maps for the new system,
denoted by $\tilde{\pi}$ and $\tilde{U}$. The crux is, then, that
anti-multiplicative representations of $A$ correspond to multiplicative
representations of $A^{o}$, and likewise for $G$ and $G^{o}$. Hence,
regardless of the type of $(\pi,U)$, one can always pass to a suitable
companion Banach algebra dynamical system to ensure that the same
pair of maps is a pair of type $(m,m)$ for the companion Banach algebra
dynamical system. For example, if $(\pi,U)$ is of type $(a,a)$ for
$(A,G,\alpha)$ and satisfies $U_{r}\pi(a)U_{r}^{-1}=\pi(\alpha_{r^{-1}}(a))$
for $a\in A$ and $r\in G$, then $\tilde{\pi}:A^{o}\to B(X)$ and
$\tilde{U}:G^{o}\to B(X)$ form a pair of type $(m,m)$ for $(A^{o},G^{o},\alpha^{o})$,
satisfying $\tilde{U}_{r}\tilde{\pi}(a)\tilde{U}_{r}^{-1}=\pi(\alpha_{r}^{o}(a))$
for $a\in A^{o}$ and $r\in G^{o}$. Hence, $(\tilde{\pi},\tilde{U})$
is a covariant pair of type $(m,m)$ for $(A^{o},G^{o},\alpha^{o})$,
and we are back at our original type of objects. One can argue similarly
for the types $(a,m)$ and $(m,a)$, and this leads to the following
table:

\begin{table}[h]
\begin{centering}
\begin{tabular}{|c|c|c|c|}
\hline 
Type of $(\pi,U)$  & Should require  & $(\tilde{\pi},\tilde{U})$ is type  & \tabularnewline
for $(A,G,\alpha)$ & that $U_{r}\pi(a)U_{r}^{-1}=$ & $(m,m)$ for & $\tilde{U}_{r}\tilde{\pi}(a)\tilde{U_{r}}^{-1}=$\tabularnewline
\hline 
\hline 
$(m,m)$ & $\pi(\alpha_{r}(a))$ & $(A,G,\alpha)$ & $\tilde{\pi}(\alpha_{r}(a))$\tabularnewline
\hline 
$(m,a)$ & $\pi(\alpha_{r^{-1}}(a))$ & $(A,G^{o},\alpha^{o})$ & $\tilde{\pi}(\alpha_{r}^{o}(a))$\tabularnewline
\hline 
$(a,m)$ & $\pi(\alpha_{r}(a))$ & $(A^{o},G,\alpha)$ & $\tilde{\pi}(\alpha_{r}(a))$\tabularnewline
\hline 
$(a,a)$ & $\pi(\alpha_{r^{-1}}(a))$ & $(A^{o},G^{o},\alpha^{o})$ & $\tilde{\pi}(\alpha_{r}^{o}(a))$\tabularnewline
\hline 
\end{tabular}
\par\end{centering}

\caption{}
\label{tab:table1}
\end{table}

We can now point out how classes of pairs $(\pi,U)$ of other types
than $(m,m)$ can be related to representations of a crossed product
of a Banach algebra. For example, suppose that $\mathcal{R}$ is a
uniformly bounded class (as in Section \ref{sec:Preliminaries-and-recapitulation})
of non-degenerate continuous pairs $(\pi,U)$ where $\pi:A\to B(X)$
and $U:G\to B(X)$ are both anti-multiplicative satisfying $U_{r}\pi(a)U_{r}^{-1}=\pi(\alpha_{r^{-1}}(a))$.
We pass to the system $(A^{o},G^{o},\alpha^{o})$ and consider the
class $\tilde{\mathcal{R}}$ consisting of all pairs $(\tilde{\pi},\tilde{U})=(\pi,U)$,
for $(\pi,U)\in\mathcal{R}$. Then $\tilde{\mathcal{R}}$ is a uniformly
bounded class of non-degenerate continuous covariant representations
of $(A^{o},G^{o},\alpha^{o})$, and the general correspondence theorem,
Theorem \ref{thm:General-Correspondence-Theorem} or \cite[Theorem 8.1]{2011arXiv1104.5151D}
furnishes a bijection between the non-degenerate bounded (multiplicative)
representations of $(A^{o}\rtimes_{\alpha^{o}}G^{o})^{\tilde{\mathcal{R}}}$
and the non-degenerate $\tilde{\mathcal{R}}$-continuous covariant
representations of $(A^{o},G^{o},\alpha^{o})$. It is then a matter
of routine, left to the reader, to reformulate the latter class as
pairs $(\pi,U)$ of type $(a,a)$ for $(A,G,\alpha)$ again, being
aware that the Haar measure for $G$ differs from that of $G^{o}$
by the modular function. The remaining types $(m,a)$ and $(a,m)$
can be treated similarly and bring the non-degenerate bounded (always
multiplicative) representations of $(A\rtimes_{\alpha^{o}}G^{o})^{\tilde{\mathcal{R}}}$
and $(A^{o}\rtimes_{\alpha}G)^{\tilde{\mathcal{R}}}$, respectively,
into play.

We now turn to what can perhaps be regarded as the sixteen canonical
types of actions of $A$ and $G$ on the linear space $C_{c}(G,A)$
(and hence on many natural Banach spaces). They are listed in Table
\ref{tab:table2} and were originally obtained by judiciously experimenting
with various candidate expressions. In this table $a\in A$, $r,s\in G$,
$f\in C_{c}(G,A)$ and $\chi:G\to\mathbb{C}^{\times}$ is a continuous
character. The possibility of inserting $\chi$ enables one to arrange,
by choosing the modular function, that the group actions as in the
lines 3, 8, 11 and 16 are isometric on $L^{p}$-type spaces for $1\leq p<\infty$.

\begin{table}[h]
\begin{centering}
\begin{tabular}{|c|c|c|c|c|}
\hline 
No. & $(\pi(a)f)(s)$ & $(U_{r}f)(s)$ & Type $(\pi,U)$ & $U_{r}\pi(a)U_{r}^{-1}$\tabularnewline
\hline 
\hline 
1 & $af(s)$ & $\chi_{r}\alpha_{r}(f(r^{-1}s))$ & $(m,m)$ & $\pi(\alpha_{r}(a))$\tabularnewline
\hline 
2 & $af(s)$ & $\chi_{r}\alpha_{r}(f(sr))$ & $(m,m)$ & $\pi(\alpha_{r}(a))$\tabularnewline
\hline 
3 & $\alpha_{s}(a)f(s)$ & $\chi_{r}f(sr)$ & $(m,m)$ & $\pi(\alpha_{r}(a))$\tabularnewline
\hline 
4 & $\alpha_{s^{-1}}(a)f(s)$ & $\chi_{r}f(r^{-1}s)$ & $(m,m)$ & $\pi(\alpha_{r}(a))$\tabularnewline
\hline 
5 & $af(s)$ & $\chi_{r}\alpha_{r^{-1}}(f(sr^{-1}))$ & $(m,a)$ & $\pi(\alpha_{r^{-1}}(a))$\tabularnewline
\hline 
6 & $af(s)$ & $\chi_{r}\alpha_{r^{-1}}(f(rs))$ & $(m,a)$ & $\pi(\alpha_{r^{-1}}(a))$\tabularnewline
\hline 
7 & $\alpha_{s^{-1}}(a)f(s)$ & $\chi_{r}f(rs)$ & $(m,a)$ & $\pi(\alpha_{r^{-1}}(a))$\tabularnewline
\hline 
8 & $\alpha_{s}(a)f(s)$ & $\chi_{r}f(sr^{-1})$ & $(m,a)$ & $\pi(\alpha_{r^{-1}}(a))$\tabularnewline
\hline 
9 & $f(s)a$ & $\chi_{r}\alpha_{r}(f(r^{-1}s))$ & $(a,m)$ & $\pi(\alpha_{r}(a))$\tabularnewline
\hline 
10 & $f(s)a$ & $\chi_{r}\alpha_{r}(f(sr))$ & $(a,m)$ & $\pi(\alpha_{r}(a))$\tabularnewline
\hline 
11 & $f(s)\alpha_{s}(a)$ & $\chi_{r}f(sr)$ & $(a,m)$ & $\pi(\alpha_{r}(a))$\tabularnewline
\hline 
12 & $f(s)\alpha_{s^{-1}}(a)$ & $\chi_{r}f(r^{-1}s)$ & $(a,m)$ & $\pi(\alpha_{r}(a))$\tabularnewline
\hline 
13 & $f(s)a$ & $\chi_{r}\alpha_{r^{-1}}(f(sr^{-1}))$ & $(a,a)$ & $\pi(\alpha_{r^{-1}}(a))$\tabularnewline
\hline 
14 & $f(s)a$ & $\chi_{r}\alpha_{r^{-1}}(f(rs))$ & $(a,a)$ & $\pi(\alpha_{r^{-1}}(a))$\tabularnewline
\hline 
15 & $f(s)\alpha_{s^{-1}}(a)$ & $\chi_{r}f(rs)$ & $(a,a)$ & $\pi(\alpha_{r^{-1}}(a))$\tabularnewline
\hline 
16 & $f(s)\alpha_{s}(a)$ & $\chi_{r}f(sr^{-1})$ & $(a,a)$ & $\pi(\alpha_{r^{-1}}(a))$\tabularnewline
\hline 
\end{tabular}
\par\end{centering}

\caption{}
\label{tab:table2}
\end{table}

We will now explain why, essentially, there is only one canonical
type of action from which all others can be derived. To start with,
note that the spaces $C_{c}(G,A)$, $C_{c}(G^{o},A)$, $C_{c}(G,A^{o})$
and $C_{c}(G^{o},A^{o})$ can all be identified. This can be put to
good use as follows: Suppose one has verified that the formulas in
line 1 yield a pair $(\pi,U)$ of type $(m,m)$ for any Banach algebra
dynamical system. Then one can apply this to $(A,G^{o},\alpha^{o})$
and view the resulting actions of $A$ and $G^{o}$ on $C_{c}(G^{o},A)$,
which are of type $(m,m)$, as actions of $A$ and $G$ on $C_{c}(G,A)$.
It is immediate that the resulting pair $(\pi,U)$ will be of type
$(m,a)$ for $(A,G,\alpha)$. In fact, it is line 5 in the table.
Likewise, line 1 for $(A^{o},G,\alpha)$ and for $(A^{o},G^{o},\alpha^{o})$
yields line 9 and 13 for $(A,G,\alpha)$, respectively. Similarly
line 2 yields the lines 6, 10 and 14, line 3 yields the lines 7, 11
and 15, and line 4 yields the lines 8, 12 and 16. Thus the actions
in lines 1 through 4 generate all others via passing to companion
Banach algebra dynamical systems. These four actions of $(A,G,\alpha)$
of type $(m,m)$ are, in turn, also essentially the same: They are,
in fact, equivalent under linear automorphisms of $C_{c}(G,A)$. In
order to see this, define, for a continuous character $\chi:G\to\mathbb{C}^{\times}$,
the linear order 2 automorphism $T_{\chi}:C_{c}(G,A)\to C_{c}(G,A)$
by 
\[
(T_{\chi}f)(s):=\chi_{s}f(s^{-1})
\]
for all $s\in G$ and $f\in C_{c}(G,A)$. Adding line numbers in brackets
in the obvious way, one then verifies that
\[
\pi_{(2)}(a)=T_{\chi(1)\chi(2)^{-1}}\pi_{(1)}(a)T_{\chi(1)\chi(2)^{-1}}^{-1}
\]
 for all $a\in A$, and
\[
U_{(2),r}=T_{\chi(1)\chi(2)^{-1}}U_{(1),r}T_{\chi(1)\chi(2)^{-1}}^{-1}
\]
 for all $r\in G$. Thus the actions in the lines 1 and 2 are equivalent.
Likewise,

\[
\pi_{(4)}(a)=T_{\chi(4)\chi(3)^{-1}}\pi_{(3)}(a)T_{\chi(4)\chi(3)^{-1}}^{-1}
\]
 for all $a\in A$, and
\[
U_{(4),r}=T_{\chi(4)\chi(3)^{-1}}U_{(3),r}T_{\chi(4)\chi(3)^{-1}}^{-1}
\]
for all $r\in G$. Hence the actions in the lines 3 and 4 are equivalent.
Furthermore, with $\chi:G\to\mathbb{C}^{\times}$ a continuous character
as before, we let $S_{\chi}:C_{c}(G,A)\to C_{c}(G,A)$ be defined
by 
\[
(S_{\chi}f)(s):=\chi_{s^{-1}}\alpha_{s^{-1}}(f(s))
\]
for all $s\in G$ and $f\in C_{c}(G,A)$. Then $S_{\chi}$ is a linear
automorphism of $C_{c}(G,A)$ and its inverse is given by 
\[
(S_{\chi}^{-1}f)(s)=\chi_{s}\alpha_{s}(f(s)).
\]
It is then straightforward to check that
\[
\pi_{(4)}(a)=S_{\chi(1)\chi(4)^{-1}}\pi_{(1)}(a)S_{\chi(1)\chi(4)^{-1}}^{-1}
\]
for all $a\in A$, and
\[
U_{(4),r}=S_{\chi(1)\chi(4)^{-1}}U_{(1),r}S_{\chi(1)\chi(4)^{-1}}^{-1}
\]
for all $r\in G$. Thus the actions in the lines 1 and 4 are equivalent,
and hence all actions of type $(m,m)$ in the lines 1 through 4 are
equivalent. Therefore, in spite of the different appearances, there
is essentially only one type of canonical action in Table \ref{tab:table2}.

We conclude this section with a discussion of the remaining case $\delta=\mbox{triv}$
in (\ref{eq:anti?covariance}), i.e., commuting actions of $A$ and
$G$. It is interesting to note that, given a Banach algebra dynamical
system $(A,G,\alpha)$, we have eight canonical commuting actions
of $A$ and $G$ on $C_{c}(G,A)$. They are listed in Table \ref{tab:table3},
with the same notational conventions as in Table \ref{tab:table2}.

\begin{table}[h]
\begin{centering}
\begin{tabular}{|c|c|c|c|c|}
\hline 
No. & $(\pi(a)f)(s)$ & $(U_{r}f)(s)$ & Type $(\pi,U)$ & $U_{r}\pi(a)U_{r}^{-1}$\tabularnewline
\hline 
\hline 
1 & $\alpha_{s}(a)f(s)$ & $\chi_{r}\alpha_{r}(f(r^{-1}s))$ & $(m,m)$ & $\pi(a)$\tabularnewline
\hline 
2 & $\alpha_{s^{-1}}(a)f(s)$ & $\chi_{r}\alpha_{r}(f(sr))$ & $(m,m)$ & $\pi(a)$\tabularnewline
\hline 
3 & $\alpha_{s^{-1}}(a)f(s)$ & $\chi_{r}\alpha_{r^{-1}}(f(sr^{-1}))$ & $(m,a)$ & $\pi(a)$\tabularnewline
\hline 
4 & $\alpha_{s}(a)f(s)$ & $\chi_{r}\alpha_{r^{-1}}(f(rs))$ & $(m,a)$ & $\pi(a)$\tabularnewline
\hline 
5 & $f(s)\alpha_{s}(a)$ & $\chi_{r}\alpha_{r}(f(r^{-1}s))$ & $(a,m)$ & $\pi(a)$\tabularnewline
\hline 
6 & $f(s)\alpha_{s^{-1}}(a)$ & $\chi_{r}\alpha_{r}(f(sr))$ & $(a,m)$ & $\pi(a)$\tabularnewline
\hline 
7 & $f(s)\alpha_{s^{-1}}(a)$ & $\chi_{r}\alpha_{r^{-1}}(f(sr^{-1}))$ & $(a,a)$ & $\pi(a)$\tabularnewline
\hline 
8 & $f(s)\alpha_{s}(a)$ & $\chi_{r}\alpha_{r^{-1}}(f(rs))$ & $(a,a)$ & $\pi(a)$\tabularnewline
\hline 
\end{tabular}
\par\end{centering}

\caption{}
\label{tab:table3}
\end{table}

We employ a similar mechanism as before. Indeed, suppose we have verified
that, for any Banach algebra dynamical system, the formulas in line
1 yield commuting actions of type $(m,m)$. Applying this to $(A,G^{o},\alpha^{o})$
one obtains a commuting pair of type $(m,a)$: line 3 in Table \ref{tab:table3}.
Likewise, line 1 for $(A^{o},G,\alpha)$ and for $(A^{o},G^{o},\alpha^{o})$
yield line 5 and line 7, respectively. Similarly line 2 yields the
lines 4, 6 and 8. Furthermore, with $\mathbf{1}:G\to\mathbb{C}^{\times}$
denoting the trivial character, one checks that 
\[
\pi_{(2)}(a)=T_{\mathbf{1}}\pi_{\text{(1)}}(a)T_{\mathbf{1}}^{-1}
\]
 for all $a\in A$, and
\[
U_{(2),r}=T_{\mathbf{1}}U_{(1),r}T_{\mathbf{1}}^{-1}
\]
 for all $r\in G$. Thus the actions in lines 1 and 2 are equivalent,
and again there is essentially only one pair of actions in Table \ref{tab:table3}.
In this case, one can even go a bit further: Define
\begin{eqnarray*}
(\tilde{\pi}(a)f)(s) & := & af(s)\\
(\tilde{U_{r}}f)(s) & := & f(r^{-1}s)
\end{eqnarray*}
for all $a\in A$, $r\in G$ and $f\in C_{c}(G,A)$. Then $(\tilde{\pi},\tilde{U})$
is ``the'' canonical covariant pair of type $(m,m)$ for $(A,G,\mbox{triv})$,
and one verifies that
\[
\pi_{(1)}(a)=S_{\chi(1)}^{-1}\tilde{\pi}(a)S_{\chi(1)}
\]
for all $a\in A$, and
\[
U_{(1),r}=S_{\chi(1)}^{-1}\tilde{U}_{r}S_{\chi(1)}
\]
for all $r\in G$. Hence all the commuting actions for $A$ and $G$
in Table \ref{tab:table3} essentially originate from the canonical
covariant pair $(\tilde{\pi},\tilde{U})$ for $(A,G,\mbox{triv})$.

\section{\label{sec:Several-BADS}Several Banach algebra dynamical systems
and classes }

\global\long\def\BADScrossedprodi{(A_{i}\rtimes_{\alpha_{i}}G_{i})^{\mathcal{R}_{i}}}

\global\long\def\BADSsystemi{(A_{i},G_{i},\alpha_{i})}

Suppose $\BADSsystemi$, with $i\in\{1,\ldots,n\}$, are finitely
many Banach algebra dynamical systems, and that $\mathcal{R}_{i}$
is a non-empty uniformly bounded class of non-degenerate continuous
covariant representations of $\BADSsystemi$. We will show (cf. Theorem
\ref{thm:sun-product-of-tensor-crossed-products}) that, for a Banach
space $X$, there is a natural bijection between the non-degenerate
bounded representations of the projective tensor product $\widehat{\bigotimes}_{i=1}^{n}\BADScrossedprodi$
on $X$ and the $n$-tuples $((\pi_{1},U_{1}),\ldots,(\pi_{n},U_{n}))$,
where, for each $i\in\{1,\ldots,n\}$, $(\pi_{i},U_{i})$ is a non-degenerate
$\mathcal{R}_{i}$-continuous covariant representation of $\BADSsystemi$
on $X$, and $(\pi_{i},U_{i})$ commutes (to be defined below) with
$(\pi_{j},U_{j})$ for all $i,j\in\{1,\ldots,n\}$ with $i\neq j,$.
Such situations are quite common. For example if $X$ is a $G$-bimodule
(i.e., $X$ is supplied with a left action $U$ of $G$ and a right
action $V$ of $G$ that commute), then this can be interpreted as
commuting non-degenerate continuous covariant representations $(\textup{id},U)$
and $(\textup{id},V)$ of $(\mathbb{K},G,\textup{triv})$ and $(\mathbb{K},G^{o},\textup{triv})$,
respectively (where $G^{o}$ denotes the opposite group of $G$).
In a similar vein, if $(\pi,U)$ is a non-degenerate continuous covariant
representation of $(A,G,\alpha)$ on $X$, and $(\rho,V)$ is a non-degenerate
continuous pair of type $(a,a)$ (in the terminology of Section \ref{sec:Other-types})
and $(\pi,U)$ and $(\rho,V)$ commute, then $(\pi,U)$ and $(\rho,V)$
can be interpreted as a pair of commuting non-degenerate continuous
covariant representations of $(A,G,\alpha)$ and $(A^{o},G^{o},\alpha^{o})$,
respectively (where $A^{o}$ and $G^{o}$ are, respectively, the opposite
Banach algebra and group of $A$ and $G$, with $\alpha_{r}^{o}:=\alpha_{r^{-1}}$
for all $r\in G$ as in Section \ref{sec:Other-types}). Theorem \ref{thm:sun-product-of-tensor-crossed-products}
explains, as a special case, how such a pair of commuting non-degenerate
covariant representations $(\pi,U)$ and $(\rho,V)$ can be related
to a non-degenerate bounded representation of $(A\rtimes_{\alpha}G)^{\mathcal{R}_{1}}\hat{\otimes}(A^{o}\rtimes_{\alpha^{o}}G^{o})^{\mathcal{R}_{2}}$,
where $\mathcal{R}_{1}$ and $\mathcal{R}_{2}$ are uniformly bounded
classes of non-degenerate continuous covariant representations of
$(A,G,\alpha)$ and $(A^{o},G^{o},\alpha^{o})$ respectively, and
$(\pi,U)$ and $(\rho,V)$ are respectively $\mathcal{R}_{1}$-continuous
and $\mathcal{R}_{2}$-continuous.

We will now proceed to establish Theorem \ref{thm:sun-product-of-tensor-crossed-products},
and start with a rather obvious definition.
\begin{defn}
\label{def:commuting-maps}Let $X$ be a Banach space and let $\varphi_{i}:S_{i}\to B(X)$
be maps from sets $S_{i}$ into $B(X)$ for $i\in\{1,2\}$. Then $\varphi_{1}$
and $\varphi_{2}$ are said to \emph{commute }if $\varphi_{1}(s_{1})\varphi_{2}(s_{2})=\varphi_{2}(s_{2})\varphi_{1}(s_{1})$
for all $s_{1}\in S_{1}$ and $s_{2}\in S_{2}$. 

Let $(A_{1},G_{1},\alpha_{1})$ and $(A_{2},G_{2},\alpha_{2})$ be
Banach algebra dynamical systems with $(\pi_{1},U_{1})$ and $(\pi_{2},U_{2})$
pairs of maps $\pi_{1}:A_{1}\to B(X)$, $U_{1}:G_{1}\to B(X)$ and
$\pi_{2}:A_{2}\to B(X)$, $U_{2}:G_{2}\to B(X)$. Then the pairs $(\pi_{1},U_{1})$
and $(\pi_{2},U_{2})$ are said to \emph{commute} if each of $\pi_{1}$
and $U_{1}$ commutes with both $\pi_{2}$ and $U_{2}$.
\end{defn}
We then have the following:
\begin{lem}
\label{commuting-covars-and-integrated-forms}Let $X$ be a Banach
space. For $i\in\{1,2\}$, let $\BADSsystemi$ be a Banach algebra
dynamical system and let $(\pi_{i},U_{i})$ be a non-degenerate continuous
covariant representation of $\BADSsystemi$ on $X$. Then the following
are equivalent:
\begin{enumerate}
\item $(\pi_{1},U_{1})$ and $(\pi_{2},U_{2})$ commute.
\item $\pi_{1}\rtimes U_{1}:C_{c}(G_{1},A_{1})\to B(X)$ and $\pi_{2}\rtimes U_{2}:C_{c}(G_{2},A_{2})\to B(X)$
commute.
\end{enumerate}
If, for $i\in\{1,2\}$, $\mathcal{R}_{i}$ is a non-empty class of
continuous covariant representations of $\BADSsystemi$, such that
$(\pi_{i},U_{i})$ is $\mathcal{R}_{i}$-continuous, then \textup{(1)}
and \textup{(2)} are also equivalent to 
\begin{enumerate}
\item [\three_hack]$(\pi_{1}\rtimes U_{1})^{\mathcal{R}_{1}}:(A_{1}\rtimes_{\alpha_{1}}G_{1})^{\mathcal{R}_{1}}\to B(X)$
and $(\pi_{2}\rtimes U_{2})^{\mathcal{R}_{2}}:(A_{2}\rtimes_{\alpha_{2}}G_{2})^{\mathcal{R}_{2}}\to B(X)$
commute.
\end{enumerate}
\end{lem}
\begin{proof}
That (1) implies (2) can be seen through repeated application of \cite[Proposition 5.5.iii]{2011arXiv1104.5151D}.
We note that non-degeneracy is not required in this step.

That (2) implies (1) follows again by repeated applications of \cite[Propositions 5.5.iii]{2011arXiv1104.5151D},
and relies on the non-degeneracy of $(\pi_{i},U_{i})$ for $i\in\{1,2\}$.

That (2) is equivalent to (3) follows from the density of $q^{\mathcal{R}_{i}}(C_{c}(G_{i},A_{i}))$
in $\BADScrossedprodi$ and the fact that $(\pi_{i}\rtimes U_{i})^{\mathcal{R}_{i}}(q^{\mathcal{R}_{i}}(f))=\pi_{i}\rtimes U_{i}(f)$
for all $f\in C_{c}(G_{i},A_{i})$, for $i\in\{1,2\}$. We again note
that non-degeneracy is not required in this step.\end{proof}
The next step is to investigate the bounded representations of the
projective tensor product $B_{1}\hat{\otimes}B_{2}$ of two Banach
algebras $B_{1}$ and $B_{2}$ (which will later be taken to be crossed
products). We refer to \cite[Section 1.5]{Kaniuth} for the details
concerning the (canonical) algebra structure on the underlying projective
tensor product $B_{1}\hat{\otimes}B_{2}$ of the Banach spaces $B_{1}$
and $B_{2}$, and start with a lemma.
\begin{lem}
\label{lem:sun-product-of-commuting-representations}Let $B_{1}$
and $B_{2}$ be Banach algebras with commuting bounded representations
$\pi_{1}:B_{1}\to B(X)$ and $\pi_{2}:B_{2}\to B(X)$ on the same
Banach space $X$. Then the map $\pi_{1}\odot\pi_{2}:B_{1}\otimes B_{2}\to B(X)$
given by 
\[
\pi_{1}\odot\pi_{2}\left(\sum_{i=1}^{n}b_{1}^{(i)}\otimes b_{2}^{(i)}\right):=\sum_{i=1}^{n}\pi_{1}(b_{1}^{(i)})\pi_{2}(b_{2}^{(i)}),
\]
where $b_{j}^{(i)}\in B_{j}$ for $j\in\{1,2\}$ and $i\in\{1,\ldots,n\}$,
is well defined and extends uniquely to a bounded representation $\pi_{1}\hat{\odot}\pi_{2}:B_{1}\hat{\otimes}B_{2}\to B(X)$. 

Furthermore,
\begin{enumerate}
\item $\|\pi_{1}\hat{\odot}\pi_{2}\|\leq\|\pi_{1}\|\|\pi_{2}\|$,
\item $\pi_{1}\hat{\odot}\pi_{2}:B_{1}\hat{\otimes}B_{2}\to B(X)$ is non-degenerate
if and only if $\pi_{1}:B_{1}\to B(X)$ and $\pi_{2}:B_{2}\to B(X)$
are non-degenerate.
\end{enumerate}
\end{lem}
\begin{proof}
It is routine to verify that $\pi_{1}\odot\pi_{2}$ is well defined
and that $\|\pi_{1}\odot\pi_{2}\|\leq\|\pi_{1}\|\|\pi_{2}\|$. The
fact that $\pi_{1}$ and $\pi_{2}$ commute implies that $\pi_{1}\odot\pi_{2}$
is a representation of $B_{1}\otimes B_{2}$, and then the existence
of $\pi_{1}\hat{\odot}\pi_{2}$ as a bounded representation of $B_{1}\hat{\otimes}B_{2}$
is clear, as is (1).

Since obviously $\overline{\textup{span}(\pi_{1}\hat{\odot}\pi_{2}(B_{1}\hat{\otimes}B_{2})X)}\subseteq\overline{\textup{span}(\pi_{i}(B_{i})X)}$
for $i\in\{1,2\}$, the non-degeneracy of $\pi_{1}\hat{\odot}\pi_{2}$
implies the non-degeneracy of both $\pi_{1}$ and $\pi_{2}$.

Conversely, assume that both $\pi_{1}$ and $\pi_{2}$ are non-degenerate,
and let $x\in X$ and $\varepsilon>0$ be arbitrary. Choose $b_{1}^{(j)}\in B_{1}$
and $x^{(j)}\in X$ with $j\in\{1,\ldots,n\}$ such that $\left\Vert x-\sum_{j=1}^{n}\pi_{1}(b_{1}^{(j)})x^{(j)}\right\Vert \leq\varepsilon/2$.
Next, choose $b_{2}^{(j,k)}\in B_{2}$ and $x^{(j,k)}\in X$ with
$j\in\{1,\ldots,n\}$ and $k\in\{1,\ldots,m_{j}\}$ such that $\|\pi_{1}(b_{1}^{(j)})\|\left\Vert x^{(j)}-\sum_{k=1}^{m_{j}}\pi_{2}(b_{2}^{(j,k)})x^{(j,k)}\right\Vert \leq\varepsilon/2n$
for all $j\in\{1,\ldots,n\}$. Then 
\begin{eqnarray*}
 &  & \left\Vert x-\sum_{j=1}^{n}\sum_{k=1}^{m_{j}}\pi_{1}\odot\pi_{2}(b_{1}^{(j)}\otimes b_{2}^{(j,k)})x^{(j,k)}\right\Vert \\
 & \leq & \left\Vert x-\sum_{j=1}^{n}\pi_{1}(b_{1}^{(j)})x^{(j)}\right\Vert +\left\Vert \sum_{j=1}^{n}\pi_{1}(b_{1}^{(j)})x^{(j)}-\sum_{j=1}^{n}\sum_{k=1}^{m_{j}}\pi_{1}\odot\pi_{2}(b_{1}^{(j)}\otimes b_{2}^{(j,k)})x^{(j,k)}\right\Vert \\
 & \leq & \frac{\varepsilon}{2}+\sum_{j=1}^{n}\|\pi_{1}(b_{1}^{(j)})\|\left\Vert x^{(j)}-\sum_{k=1}^{m_{j}}\pi_{2}(b_{2}^{(j,k)})x^{(j,k)}\right\Vert \\
 & < & \varepsilon.
\end{eqnarray*}
Hence $\pi_{1}\hat{\odot}\pi_{2}$ is non-degenerate.\end{proof}
If both $B_{1}$ and $B_{2}$ have a bounded approximate left identity,
then all non-degenerate bounded representations of $B_{1}\hat{\otimes}B_{2}$
arise in this fashion for unique (necessarily non-degenerate, in view
of Lemma \ref{lem:sun-product-of-commuting-representations}) bounded
$\pi_{1}$ and $\pi_{2}$. More precisely, we have the following result,
for which we have not been able to find a reference.
\begin{prop}
\label{prop:sun-products-are-unique}Let $B_{1}$ and $B_{2}$ be
Banach algebras both having a bounded approximate left identity, and
let $X$ be a Banach space. If $\pi_{1}:B_{1}\to B(X)$ and $\pi_{2}:B_{2}\to B(X)$
are commuting non-degenerate bounded representations, then $\pi_{1}\hat{\odot}\pi_{2}:B_{1}\hat{\otimes}B_{2}\to B(X)$
is a non-degenerate bounded representation, and all non-degenerate
bounded representations of $B_{1}\hat{\otimes}B_{2}$ are obtained
in this fashion, for unique non-degenerate bounded representations
$\pi_{1}$ and $\pi_{2}$. Then
\begin{enumerate}
\item $\|\pi_{1}\hat{\odot}\pi_{2}\|\leq\|\pi_{1}\|\|\pi_{2}\|$
\item If, for $i\in\{1,2\}$, $B_{i}$ has an $M_{i}$-bounded approximate
left identity, then $\|\pi_{i}\|\leq M_{1}M_{2}\|\lambda_{B_{i}}\|\|\pi_{1}\hat{\odot}\pi_{2}\|$,
with $\lambda_{B_{i}}:B_{i}\to B(B_{i})$ denoting the left regular
representation of $B_{i}$.
\end{enumerate}
\end{prop}
\begin{proof}
Part of the proposition, including (1), has already been established
in Lemma \ref{lem:sun-product-of-commuting-representations}. We start
from a given non-degenerate bounded representation $\pi:B_{1}\hat{\otimes}B_{2}\to B(X)$
and construct the non-degenerate bounded representations $\pi_{1}$
and $\pi_{2}$ such that $\pi=\pi_{1}\hat{\odot}\pi_{2}$. First,
we note that $B_{1}\hat{\otimes}B_{2}$ has an approximate left identity
bounded by $M_{1}M_{2}$ \cite[Lemma 1.5.3]{Kaniuth}. Therefore,
if we let $\overline{\pi}:\leftcent(B_{1}\hat{\otimes}B_{2})\to B(X)$
denote the non-degenerate bounded representations of $\leftcent(B_{1}\hat{\otimes}B_{2})$
such that the diagram 
\[
\xymatrix{B_{1}\hat{\otimes}B_{2}\ar[r]^{\pi}\ar[dr]^{\lambda} & B(X)\\
 & \leftcent(B_{1}\hat{\otimes}B_{2})\ar[u]_{\overline{\pi}}
}
\]
commutes, then $\|\overline{\pi}\|\leq M_{1}M_{2}\|\pi\|$ \cite[Theorem 4.1]{2009arXiv0904.3268D}.
We will now compose $\overline{\pi}$ with bounded homomorphisms of
$B_{1}$ and $B_{2}$ into $\leftcent(B_{1}\hat{\otimes}B_{2})$ to
obtain the sought representations $\pi_{1}$ and $\pi_{2}$. For $b_{1}\in B_{1}$
consider $\lambda_{B_{1}}(b_{1})\hat{\otimes}\textup{id}_{B_{2}}\in B(B_{1}\hat{\otimes}B_{2})$,
where $\lambda_{B_{1}}(b_{1})$ is the image under the left regular
representation $\lambda_{B_{1}}:B_{1}\to B(B_{1})$ of $B_{1}$. Clearly,
$\|\lambda_{B_{1}}(b_{1})\hat{\otimes}\textup{id}_{B_{2}}\|=\|\lambda_{B_{1}}(b_{1})\|\leq\|\lambda_{B_{1}}\|\|b_{1}\|$,
and one readily verifies that $\lambda_{B_{1}}(b_{1})\hat{\otimes}\textup{id}_{B_{2}}\in\leftcent(B_{1}\hat{\otimes}B_{2})$.
If we define $l_{1}:B_{1}\to\leftcent(B_{1}\hat{\otimes}B_{2})$ by
$l_{1}(b_{1}):=\lambda_{B_{1}}(b_{1})\hat{\otimes}\textup{id}_{B_{2}}$
for $b_{1}\in B_{1}$, then $l_{1}$ is a bounded homomorphism, and
$\|l_{1}\|\leq\|\lambda_{B_{1}}\|$. Likewise, $l_{2}:B_{2}\to\leftcent(B_{1}\hat{\otimes}B_{2}),$
defined by $l_{2}(b_{2}):=\textup{id}_{B_{1}}\hat{\otimes}\lambda_{B_{2}}(b_{2})$
for $b_{2}\in B_{2}$, is a bounded homomorphism, and $\|l_{2}\|\leq\|\lambda_{B_{2}}\|$.
Now, for $i\in\{1,2\}$, define $\pi_{i}:B_{i}\to B(X)$ as $\pi_{i}:=\overline{\pi}\circ l_{i}$.
We note that $\|\pi_{i}\|\leq\|\overline{\pi}\|\|l_{i}\|\leq M_{1}M_{2}\|\lambda_{B_{i}}\|\|\pi\|$.
Since $l_{1}$ and $l_{2}$ obviously commute, the same holds true
for $\pi_{1}$ and $\pi_{2}$. Therefore $\pi_{1}\hat{\odot}\pi_{2}:B_{1}\hat{\otimes}B_{2}\to B(X)$
is a bounded representation.

We will proceed to show that $\pi_{1}\hat{\odot}\pi_{2}=\pi$, and
that $\pi_{1}$ and $\pi_{2}$ are uniquely determined. We compute,
for $x\in X$, $b_{1}^{(1)},b_{1}^{(2)}\in B_{1}$ and $b_{2}^{(1)},b_{2}^{(2)}\in B_{2}$:
\begin{eqnarray*}
 &  & \pi_{1}\hat{\odot}\pi_{2}(b_{1}^{(1)}\otimes b_{2}^{(1)})\pi(b_{1}^{(2)}\otimes b_{2}^{(2)})x\\
 & = & \pi_{1}(b_{1}^{(1)})\pi_{2}(b_{2}^{(1)})\pi(b_{1}^{(2)}\otimes b_{2}^{(2)})x\\
 & = & \overline{\pi}(\lambda_{B_{1}}(b_{1}^{(1)})\hat{\otimes}\textup{id}_{B_{2}})\overline{\pi}(\textup{id}_{B_{1}}\hat{\otimes}\lambda_{B_{2}}(b_{2}^{(1)}))\pi(b_{1}^{(2)}\otimes b_{2}^{(2)})x\\
 & = & \overline{\pi}(\lambda_{B_{1}}(b_{1}^{(1)})\hat{\otimes}\textup{id}_{B_{2}})\pi(\textup{id}_{B_{1}}\hat{\otimes}\lambda_{B_{2}}(b_{2}^{(1)})(b_{1}^{(2)}\otimes b_{2}^{(2)}))x\\
 & = & \pi(\lambda_{B_{1}}(b_{1}^{(1)})\hat{\otimes}\textup{id}_{B_{2}}(b_{1}^{(2)}\otimes b_{2}^{(1)}b_{2}^{(2)}))x\\
 & = & \pi(b_{1}^{(1)}b_{1}^{(2)}\otimes b_{2}^{(1)}b_{2}^{(2)})x\\
 & = & \pi(b_{1}^{(1)}\otimes b_{2}^{(1)})\pi(b_{1}^{(2)}\otimes b_{2}^{(2)})x.
\end{eqnarray*}
Since $\pi$ is non-degenerate and $B_{1}\otimes B_{2}$ is dense
in $B_{1}\hat{\otimes}B_{2}$, the restriction of $\pi$ to $B_{1}\otimes B_{2}$
is also non-degenerate. Hence we conclude from the above that $\pi_{1}\hat{\odot}\pi_{2}(b_{1}\otimes b_{2})=\pi(b_{1}\otimes b_{2})$
for all $b_{1}\in B_{1}$ and $b_{2}\in B_{2}$, i.e., that $\pi_{1}\hat{\odot}\pi_{2}=\pi$.
It is now clear that $\|\pi_{i}\|\leq M_{1}M_{2}\|\lambda_{B_{i}}\|\|\pi_{1}\hat{\odot}\pi_{2}\|$.
As already mentioned preceding the proposition, $\pi_{1}$ and $\pi_{2}$
are necessarily non-degenerate. 

As to uniqueness, assume that $\rho_{1}:B_{1}\to B(X)$ and $\rho_{2}:B_{2}\to B(X)$
are commuting bounded representations such that $\rho_{1}\hat{\odot}\rho_{2}=\pi$.
Then, for $x\in X$, $b_{1},b_{1}'\in B$ and $b_{2}'\in B_{2}$,
\begin{eqnarray*}
\rho_{1}(b_{1})\pi(b_{1}'\otimes b_{2}')x & = & \rho_{1}(b_{1})\rho_{1}\hat{\odot}\rho_{2}(b_{1}'\otimes b_{2}')x\\
 & = & \rho_{1}(b_{1})\rho_{1}(b_{1}')\rho_{2}(b_{2}')x\\
 & = & \rho_{1}(b_{1}b_{1}')\rho_{2}(b_{2}')x\\
 & = & \rho_{1}\hat{\odot}\rho_{2}(b_{1}b_{1}'\otimes b_{2}')x\\
 & = & \pi(\lambda_{B_{1}}(b_{1})\hat{\otimes}\textup{id}_{B_{2}}(b_{1}'\otimes b_{2}'))x\\
 & = & \overline{\pi}(\lambda_{B_{1}}(b_{1})\hat{\otimes}\textup{id}_{B_{2}})\pi(b_{1}'\otimes b_{2}')x\\
 & = & \pi_{1}(b_{1})\pi(b_{1}'\otimes b_{2}')x.
\end{eqnarray*}
The non-degeneracy of $\pi$ then implies that necessarily $\rho_{1}=\pi_{1}$
and likewise that $\rho_{2}=\pi_{2}$.\end{proof}
The following is now simply a matter of combining the General Correspondence
Theorem (Theorem \ref{thm:General-Correspondence-Theorem}), Lemma
\ref{lem:sun-product-of-commuting-representations}, Proposition \ref{prop:sun-products-are-unique},
and an induction argument.
\begin{thm}
\label{thm:sun-product-of-tensor-crossed-products}For $i\in\{1,\ldots,n\}$,
let $\BADSsystemi$ be a Banach algebra dynamical system, where $A_{i}$
has a bounded approximate left identity, and $\mathcal{R}_{i}$ is
a non-empty uniformly bounded class of non-degenerate continuous covariant
representations of $\BADSsystemi$. Let $X$ be a Banach space. Let
$((\pi_{1},U_{1}),\ldots,(\pi_{n},U_{n}))$ be an $n$-tuple where,
for each $i\in\{1,\ldots,n\}$, the pair $(\pi_{i},U_{i})$ is a non-degenerate
$\mathcal{R}_{i}$-continuous covariant representation of $\BADSsystemi$
on $X$, and all $(\pi_{i},U_{i})$ and $(\pi_{j},U_{j})$ commute
for all $i,j\in\{1,\ldots,n\}$ with $i\neq j$. Then the map sending
$((\pi_{1},U_{1}),\ldots,(\pi_{n},U_{n}))$ to the representation
\[
\widehat{\bigodot}_{i=1}^{n}(\pi_{i}\rtimes U_{i})^{\mathcal{R}_{i}}:\widehat{\bigotimes}_{i=1}^{n}\BADScrossedprodi\to B(X),
\]
is a bijection between the set of all such $n$-tuples and the set
of all non-degenerate bounded representations of $\widehat{\bigotimes}_{i=1}^{n}\BADScrossedprodi$
on $X$.
\end{thm}
For the sake of completeness, we mention that the commutativity assumption
applies only to the non-degenerate $\mathcal{R}_{i}$-continuous covariant
representations $(\pi_{i},U_{i})$, not to the elements of $\mathcal{R}_{i}$.

In Remark \ref{rem:applying-sun-tensor-to-beurling-bimodules} we
will apply Theorem \ref{thm:sun-product-of-tensor-crossed-products}
to relate bimodules over generalized Beurling algebras to left modules
over the projective tensor product of the algebra acting on the left
and the opposite algebra of the one acting on the right.

\section{Right and bimodules over generalized Beurling algebras\label{sec:Beurling-Right-and-bimodules}}

\global\long\def\OppBeurlingTypeAlg{L^{1}(G^{o},A^{o},\omega^{o};\alpha^{o})}

\global\long\def\BBeurlingTypeAlg{L^{1}(H,B,\eta;\beta)}

\global\long\def\OppBBeurlingTypeAlg{L^{1}(H^{o},B^{o},\eta^{o};\beta^{o})}

Let $(A,G,\alpha)$ be a Banach algebra dynamical system, where $A$
has a bounded two-sided approximate identity and $\alpha$ is uniformly
bounded, and let $\omega$ be a weight on $G$. In Section \ref{sec:Applications_L1_and_Beurling_algebras}
we have seen that the Banach space $L^{1}(G,A,\omega)$ has the structure
of an associative algebra, denoted $\BeurlingTypeAlg$, with multiplication
continuous in both variables, determined by

\[
[f*_{\alpha}g](s):=\int_{G}f(r)\alpha_{r}(g(r^{-1}s))\, d\mu(r)\quad(f,g\in C_{c}(G,A),\: s\in G).
\]
Here we have written $*_{\alpha}$ rather than $*$ to indicate the
$\alpha$-dependence of the multiplication (twisted convolution) on
$C_{c}(G,A)$, as another multiplication will also appear. For the
same reason we have now also written $d\mu$ for the chosen left Haar
measure on $G$. Furthermore, we have seen in Section \ref{sec:Applications_L1_and_Beurling_algebras}
that $\BeurlingTypeAlg$ is isomorphic to the Banach algebra $\crossedprod,$
when $\mathcal{R}$ is chosen suitably. As a consequence of the General
Correspondence Theorem (Theorem \ref{thm:General-Correspondence-Theorem}),
it was then shown that if $(\pi,U)$ is a non-degenerate continuous
covariant representation of $(A,G,\alpha)$, such that $\|U_{r}\|\leq C_{U}\omega(r)$
for all $r\in G$, then $\pi\rtimes U(f)=\int\pi(f)U_{r}\, d\mu(r)$,
for $f\in C_{c}(G,A)$, determines a non-degenerate bounded representation
of $\BeurlingTypeAlg$, and that all non-degenerate bounded representations
of $\BeurlingTypeAlg$ are uniquely determined in this way by such
pairs $(\pi,U)$. 

In this section we will explain how the non-degenerate bounded anti-representations
of $\BeurlingTypeAlg$ (i.e.,\,the non-degenerate right $\BeurlingTypeAlg$-modules)
are in natural bijection with the pairs $(\pi,U)$, where $\pi:A\to B(X)$
is non-degenerate, bounded and anti-multiplicative, $U:G\to B(X)$
is strongly continuous and anti-multiplicative, satisfy 
\[
U_{r}\pi(a)U_{r}^{-1}=\pi(\alpha_{r^{-1}}(a))\quad(a\in A,\: r\in G),
\]
(i.e.,\,with the non-degenerate continuous pairs $(\pi,U)$ of type
$(a,a)$ as in Section~\ref{sec:Other-types}, called thrice ``flawed''
in the introduction) and are such that $\|U_{r}\|\leq C_{U}\omega(r),$
for some $C_{U}\geq0$ and all $r\in G$. This may look counterintuitive
to the idea of Section \ref{sec:Other-types}, where it was argued
that one can ``always'' reinterpret given data so as to end up with
pairs of type $(m,m)$ for a (companion) Banach algebra dynamical
system, and then formulate a General Correspondence Theorem involving
the non-degenerate bounded representations of a companion crossed
product: anti-representations of the resulting crossed product never
enter the picture. Yet this is precisely what we will do, but it is
only the first step. 

In this first step the relevant crossed product will, as in Section
\ref{sec:Applications_L1_and_Beurling_algebras}, turn out to be topologically
isomorphic to $\OppBeurlingTypeAlg$ (where $\omega^{o}$ equals $\omega$,
seen as a weight on $G^{o}$). As it happens, $\OppBeurlingTypeAlg$
is topologically anti-isomorphic to $\BeurlingTypeAlg$. Hence, in
the second step, the non-degenerate bounded representations of $\OppBeurlingTypeAlg$
are viewed as the non-degenerate bounded anti-representations of $\BeurlingTypeAlg$,
which are thus, in the end, related to pairs $(\pi,U)$ of type $(a,a)$
as above. For this result, therefore, one should not think of $\BeurlingTypeAlg$
as being topologically isomorphic to a crossed product as in Section
\ref{sec:Applications_L1_and_Beurling_algebras}. Although this is
also the case, its main feature here is that it is anti-isomorphic
to $\OppBeurlingTypeAlg$ which, in turn, is topologically isomorphic
to the crossed product that ``actually'' explains the situation.

Once this has been completed, we remind ourselves again that $\BeurlingTypeAlg$
itself is topologically isomorphic to a crossed product, and combine
the results in the first part of this section with those in Sections
\ref{sec:Applications_L1_and_Beurling_algebras} and \ref{sec:Several-BADS}
in Theorem \ref{thm:Beurling-bimodules}, to describe for two Banach
algebra dynamical systems $(A,G,\alpha)$ and $(B,H,\beta)$ the non-degenerate
simultaneously left $\BeurlingTypeAlg$-- and right $\BBeurlingTypeAlg$-modules,
and, in the special case where $(A,G,\alpha)=(B,H,\beta)$, the non-degenerate
$\BeurlingTypeAlg$-bimodules.

To start, recall that the canonical left invariant measure $\mu$
on the opposite group $G^{o}$ of $G$ is given by $\mu^{o}(E):=\mu(E^{-1})$,
for $E$ a Borel subset of $G$. Then, recalling that $\int_{G}f\, d\mu=\int_{G}f(r^{-1})\Delta(r^{-1})\, d\mu(r)$
\cite[Lemma 1.67]{Williams}, for $f\in C_{c}(G)$, we have 
\[
\int_{G^{o}}f(r)\, d\mu^{o}(r)=\int_{G}f(r^{-1})\, d\mu(r)=\int_{G}f(r)\Delta(r^{-1})\, d\mu(r).
\]
We recall from Section \ref{sec:Other-types} if $(A,G,\alpha)$ is
a Banach algebra dynamical system, then so is $(A^{o},G^{o},\alpha^{o})$,
where $A^{o}$ is the opposite algebra of $A$, $G^{o}$ is the opposite
group of $G$, and $\alpha^{o}:G^{o}\to\mbox{Aut}(A^{o})=\mbox{Aut}(A)$
is given by $\alpha_{s}^{o}=\alpha_{s^{-1}}$ for all $s\in G^{o}$.
The vector spaces $C_{c}(G,A)$ and $C_{c}(G^{o},A^{o})$ can be identified,
but there are two convolution structures on it. If $\circledcirc$
denotes the multiplication in $A^{o}$ and $G^{o}$, then 
\[
[f*_{\alpha}g](s)=\int_{G}f(r)\alpha_{r}(g(r^{-1}s))\, d\mu(r)\quad(f,g\in C_{c}(G,A),\ s\in G),
\]
 and 
\[
[f*_{\alpha^{o}}g](s)=\int_{G}f(r)\circledcirc\alpha_{r}^{o}(g(r^{-1}\circledcirc s))\, d\mu^{o}(r)\quad(f,g\in C_{c}(G^{o},A^{o}),\ s\in G^{o}).
\]
Hence we have two associative algebras: $C_{c}(G,A)$ with multiplication
$*_{\alpha}$, and $C_{c}(G^{o},A^{o})$ with multiplication $*_{\alpha^{o}}$,
having the same underlying vector space. The first observation we
need is then the following:
\begin{lem}
\label{lem:hat-is-anti-isomorphism}Let $(A,G,\alpha)$ be a Banach
algebra dynamical system with companion opposite system $(A^{o},G^{o},\alpha^{o})$,
and let $\chi:G\to\mathbb{C}^{\times}$ be a continuous character
of $G$. For $f\in C_{c}(G,A)$, define $\hat{f}\in C_{c}(G^{o},A^{o})$
by $\hat{f}(s):=\chi(s^{-1})\alpha_{s^{-1}}(f(s))$ for $s\in G^{o}$.
Then the map $f\mapsto\hat{f}$ is an anti-isomorphism of the associative
algebras $C_{c}(G,A)$ with multiplication $*_{\alpha}$, and $C_{c}(G^{o},A^{o})$
with multiplication $*_{\alpha^{o}}$. The inverse is given by $g\mapsto\check{g}$,
where $\check{g}(s):=\chi(s)\alpha_{s}(g(s))$ for $g\in C_{c}(G^{o},A^{o})$
and $s\in G$.\end{lem}
\begin{proof}
It is clear that $\hat{\cdot}$ and $\check{\cdot}$ are mutually
inverse linear bijections. As to the multiplicative structures, we
compute, for $f,g\in C_{c}(G,A)$ and $s\in G^{o}$,
\begin{eqnarray*}
[\hat{f}*_{\alpha^{o}}\hat{g}](s) & = & \int_{G^{o}}\hat{f}(r)\circledcirc\alpha_{r^{-1}}^{o}(\hat{g}(r^{-1}\circledcirc s))\, d\mu^{o}(r)\\
 & = & \int_{G}\hat{f}(r^{-1})\circledcirc\alpha_{r^{-1}}^{o}(\hat{g}(r\circledcirc s))\, d\mu(r)\\
 & = & \int_{G}\alpha_{r}(\hat{g}(sr))\hat{f}(r^{-1})\, d\mu(r)\\
 & = & \int_{G}\alpha_{r}(\chi((sr)^{-1})\alpha_{(sr)^{-1}}g(sr))\chi(r)\alpha_{r}(f(r^{-1}))\, d\mu(r)\\
 & = & \chi(s^{-1})\int_{G}\alpha_{s^{-1}}(g(sr))\alpha_{r}(f(r^{-1}))\, d\mu(r)\\
 & = & \chi(s^{-1})\alpha_{s^{-1}}\left(\int_{G}g(sr))\alpha_{sr}(f(r^{-1}))\, d\mu(r)\right)\\
 & = & \chi(s^{-1})\alpha_{s^{-1}}\left(\int_{G}g(r))\alpha_{r}(f(r^{-1}s))\, d\mu(r)\right)\\
 & = & (g*_{\alpha}f)^{\wedge}(s).
\end{eqnarray*}
\end{proof}
Choosing $\chi$ suitably, we obtain a topological isomorphism in
the next result.
\begin{prop}
\label{prop:anti-isomorphism-between-beurling-and-opposite-beurling}Let
$(A,G,\alpha)$ be a Banach algebra dynamical system, where $\alpha$
is uniformly bounded. Let $\omega$ be a weight on $G$ and view $\omega^{o}:=\omega$
also as a weight on $G^{o}$. Then the map $f\mapsto\hat{f}$, where
$\hat{f}(s):=\Delta(s)\alpha_{s^{-1}}(f(s))$ for $f\in C_{c}(G,A)$
and $s\in G^{o}$ defines a topological anti-isomorphism between $\BeurlingTypeAlg$
and $\OppBeurlingTypeAlg$. The inverse map is determined by $g\mapsto\check{g}$
where $\check{g}(s):=\Delta(s^{-1})\alpha_{s}(g(s))$ for $g\in C_{c}(G^{o},A^{o})$
and $s\in G$. \end{prop}
\begin{proof}
In view of Lemma \ref{lem:hat-is-anti-isomorphism}, we need only
show that $\hat{\cdot}$ and $\check{\cdot}$ are isomorphisms between
the normed spaces $(C_{c}(G,A),\|\cdot\|_{1,\omega})$ and $(C_{c}(G^{o},A^{o}),\|\cdot\|_{1,\omega^{o}})$.
Let $\alpha$ be uniformly bounded by $C_{\alpha}$. If $f\in C_{c}(G,A)$,
then
\begin{eqnarray*}
\|\hat{f}\|_{1,\omega^{o}} & = & \int_{G^{o}}\|\hat{f}(r)\|\omega^{o}(r)\, d\mu^{o}(r)\\
 & = & \int_{G^{o}}\|\Delta(r)\alpha_{r^{-1}}(f(r))\|\omega(r)\, d\mu^{o}(r)\\
 & \leq & C_{\alpha}\int_{G^{o}}\|f(r)\|\omega(r)\Delta(r)\, d\mu^{o}(r)\\
 & = & C_{\alpha}\int_{G}\|f(r^{-1})\|\omega(r^{-1})\Delta(r^{-1})\, d\mu(r)\\
 & = & C_{\alpha}\int_{G}\|f(r)\|\omega(r)\, d\mu(r)\\
 & = & C_{\alpha}\|f\|_{1,\omega}.
\end{eqnarray*}
Similarly $\|\check{f}\|_{1,\omega}\leq C_{\alpha}\|f\|_{1,\omega^{o}}$
for all $f\in C_{c}(G^{o},A^{o})$.\end{proof}
It is now an easy matter to combine the ideas of Sections \ref{sec:Applications_L1_and_Beurling_algebras}
and \ref{sec:Other-types} with the above Proposition \ref{prop:anti-isomorphism-between-beurling-and-opposite-beurling}.

Let $X$ be a Banach space and let $(A,G,\alpha)$ be a Banach algebra
dynamical system, where $A$ has a bounded two-sided approximate identity
and $\alpha$ is uniformly bounded. As in Section \ref{sec:Other-types},
the pairs $(\pi,U)$, where $\pi:A\to B(X)$ is non-degenerate, bounded
and anti-multiplicative, $U:G\to B(X)$ is strongly continuous and
anti-multiplicative, and $U_{r}^{-1}\pi(a)U_{r}=\pi(\alpha_{r^{-1}}(a))$
for $a\in A$ and $r\in G$, can be identified with the pairs $(\pi^{o},U^{o})$,
where $\pi^{o}:A^{o}\to B(X)$, with $\pi^{o}(a):=\pi(a)$ for $a\in A$,
is non-degenerate, bounded and multiplicative, $U^{o}:G^{o}\to B(X)$,
with $U_{r}^{o}=U_{r}$ for all $r\in G^{o}$, is strongly continuous
and multiplicative, and $U_{r}^{o}\pi^{o}(a)U_{r}^{o-1}=\pi^{o}(\alpha_{r}^{o}(a))$
for $a\in A^{o}$ and $r\in G^{o}$. Furthermore, if $\omega$ is
a weight on $G$, also viewed as a weight $\omega^{o}:=\omega$ on
$G^{o}$, then there exists a constant $C_{U}$ such that $\|U_{r}\|\leq C_{U}\omega(r)$
for all $r\in G$ if and only if there exists a constant $C_{U^{o}}$
such that $\|U_{r}^{o}\|\leq C_{U^{o}}\omega^{o}(r)$ for all $r\in G^{o}$:
take the same constant. Now the collection of all such pairs $(\pi^{o},U^{o})$
is, in view of Theorem \ref{thm:continuous-non-deg-covars-are-R-continuous},
in natural bijection with the collection of all non-degenerate bounded
representations of $\OppBeurlingTypeAlg$ on $X$. As a consequence
of Proposition \ref{prop:anti-isomorphism-between-beurling-and-opposite-beurling},
this can in turn be viewed as the collection of all non-degenerate
bounded anti-representations of $\BeurlingTypeAlg$ on $X$. Combining
these three bijections, we can let pairs $(\pi,U)$ as described above
correspond bijectively to the non-degenerate bounded anti-representations
of $\BeurlingTypeAlg$ on $X$: If $(\pi,U)$ is such a pair, we associate
with it the non-degenerate bounded anti-representation of $\BeurlingTypeAlg$
determined by sending $f\in C_{c}(G,A)$ to $\pi^{o}\rtimes U^{o}(\hat{f})$.
Explicitly, for $f\in C_{c}(G,A)$, 
\begin{eqnarray*}
\pi^{o}\rtimes U^{o}(\hat{f}) & = & \int_{G^{o}}\pi^{o}(\hat{f}(r))U_{r}^{o}\, d\mu^{o}(r)\\
 & = & \int_{G^{o}}\pi(\Delta(r)\alpha_{r^{-1}}(f(r)))U_{r}\, d\mu^{o}(r)\\
 & = & \int_{G}\pi(\alpha_{r}(f(r^{-1})))U_{r^{-1}}\Delta(r^{-1})\, d\mu(r)\\
 & = & \int_{G}\pi(\alpha_{r^{-1}}(f(r)))U_{r}\, d\mu(r)\\
 & = & \int_{G}U_{r}U_{r}^{-1}\pi(\alpha_{r^{-1}}(f(r)))U_{r}\, d\mu(r)\\
 & = & \int_{G}U_{r}\pi(\alpha_{r}\circ\alpha_{r^{-1}}(f(r)))\, d\mu(r)\\
 & = & \int_{G}U_{r}\pi(f(r))\, d\mu(r).
\end{eqnarray*}

To retrieve $(\pi,U)$ from a given non-degenerate bounded anti-representation
$T$ of $\BeurlingTypeAlg$, we note that, by Proposition \ref{prop:anti-isomorphism-between-beurling-and-opposite-beurling},
$T\circ\check{\cdot}$ is a non-degenerate bounded representation
of $\OppBeurlingTypeAlg$, and hence, we can apply \cite[Equations (8.1) and (8.2)]{2011arXiv1104.5151D}
to $T\circ\check{\cdot}$. A bounded approximate left identity of
$A^{o}$ is then needed, and for this we take a bounded approximate
right identity $(u_{i})$ of $A$. Furthermore, if $V$ runs through
a neighbourhood base $\mathcal{Z}$ of $e\in G$, of which all elements
are contained in a fixed compact set of $G$, and $z_{V}\in C_{c}(G)$
is positive, supported in $V$, and $\int_{G}z_{V}(r^{-1})\, d\mu(r)=\int_{G^{o}}z_{V}(r)\, d\mu^{o}(r)=1$,
then the $z_{V}\in C_{c}(G)$ are as required for \cite[Equations (8.1) and (8.2)]{2011arXiv1104.5151D}.
Hence, again taking Remark \ref{rem:eq-8.1and8.2-can-be-simplified-by-deleting-alpha}
into account, we have, for $a\in A$,
\begin{eqnarray*}
\pi(a)=\pi^{o}(a) & = & \textup{SOT-lim}_{(V,i)}T((z_{V}\otimes a\circledcirc u_{i})^{\vee})\\
 & = & \textup{SOT-lim}_{(V,i)}T((z_{V}\otimes u_{i}a)^{\vee}),
\end{eqnarray*}
where $(z_{V}\otimes u_{i}a)^{\vee}(r)=\Delta(r^{-1})z_{V}(r)\alpha_{r}(au_{i})$
for $r\in G$, and, for $s\in G$,
\begin{eqnarray*}
U_{s}=U_{s}^{o} & = & \textup{SOT-lim}_{(V,i)}T((z_{V}(s^{-1}\circledcirc\cdot)\otimes u_{i})^{\vee})\\
 & = & \textup{SOT-lim}_{(V,i)}T((z_{V}(\cdot s^{-1})\otimes u_{i})^{\vee}),
\end{eqnarray*}
where $(z_{V}(\cdot s^{-1})\otimes u_{i})^{\vee}(r)=\Delta(r^{-1})z_{V}(rs^{-1})\alpha_{r}(u_{i})$
for $r\in G$.

All in all, we have the following result in analogy to Theorem \ref{thm:continuous-non-deg-covars-are-R-continuous}:
\begin{thm}
\label{thm:anti-correspondence}Let $(A,G,\alpha)$ be a Banach algebra
dynamical system where $A$ has a two-sided approximate identit\textup{y}
and $\alpha$ is uniformly bounded by a constant $C_{\alpha}$, and
let $\omega$ be a weight on $G$. Let $X$ be a Banach space. Let
the pair $(\pi,U)$ be such that $\pi:A\to B(X)$ is a non-degenerate
bounded anti-representation, $U:G\to B(X)$ is a strongly continuous
anti-representation satisfying $U_{r}\pi(\alpha)U_{r}^{-1}=\pi(\alpha_{r^{-1}}(a))$
for all $a\in A$ and $r\in G$, and with $C_{U}$ a constant such
that $\|U_{r}\|\leq C_{U}\omega(r)$ for all $r\in G$. Let $T:\BeurlingTypeAlg\to B(X)$
be a non-degenerate bounded anti-representation of $\BeurlingTypeAlg$
on $X$. Then the following maps are mutual inverses between all such
pairs $(\pi,U)$ and the non-degenerate bounded anti-representations
$T$ of $\BeurlingTypeAlg$: 
\[
(\pi,U)\mapsto\left(f\mapsto\int_{G}U_{r}\pi(f(r))\, dr\right)=:T^{(\pi,U)}\quad(f\in C_{c}(G,A)),
\]
determining a non-degenerate bounded anti-representation $T^{(\pi,U)}$
of $\BeurlingTypeAlg$, and, 
\[
T\mapsto\left(\begin{array}{l}
a\mapsto\textup{SOT-lim}_{(V,i)}T((z_{V}\otimes u_{i}a)^{\vee}),\\
s\mapsto\textup{SOT-lim}_{(V,i)}T((z_{V}(\cdot s^{-1})\otimes u_{i})^{\vee})
\end{array}\right)=:(\pi^{T},U^{T}),
\]
where $\mathcal{Z}$ is a neighbourhood base of $e\in G$, of which
all elements are contained in a fixed compact subset of $G$, $z_{V}\in C_{c}(G)$
is chosen such that $z_{V}\geq0$, supported in $V\in\mathcal{Z}$,
$\int_{G}z_{V}(r^{-1})\, dr=1$, and $(u_{i})$ is any bounded approximate
right identity of $A$.

Furthermore, if $A$ has an $M$-bounded approximate right identity,
then the following bounds for $T^{(\pi,U)}$ and $(\pi^{T},U^{T})$
hold:
\begin{enumerate}
\item $\|T^{(\pi,U)}\|\leq C_{U}\|\pi\|$,
\item \textup{$\|\pi^{T}\|\leq\left(\inf_{V\in\mathcal{Z}}\sup_{r\in V}\omega(r)\right)\|T\|$,}
\item $\|U_{s}^{T}\|\leq M\left(\inf_{V\in\mathcal{Z}}\sup_{r\in V}\omega(r)\right)\|T\|\,\omega(s)\quad(s\in G)$.
\end{enumerate}
\end{thm}
\begin{proof}
Except for the bounds, all statements were proven in the discussion
preceding the statement of the theorem. Establishing the bound (1)
proceeds as in Theorem \ref{thm:continuous-non-deg-covars-are-R-continuous}. 

To establish (2), we choose a bounded two-sided approximate identity
$(u_{i})$ of $A.$ Let $a\in A$ and $\varepsilon_{1},\varepsilon_{2},\varepsilon_{3}>0$
be arbitrary. There exists an index $i_{0}$ such that $\|u_{i}a\|\leq\|a\|+\varepsilon_{1}$
for all $i\geq i_{0}$. There exists some $W_{1}\in\mathcal{Z}$ such
that $\sup_{r\in W_{1}}\omega(r)\leq\inf_{V\in\mathcal{Z}}\sup_{r\in V}\omega(r)+\varepsilon_{2}$.
Since $r\mapsto\|\alpha_{r}\|$ is lower semicontinuous and $\|\alpha_{e}\|=1$,
there exists some $W_{2}\in\mathcal{Z}$ such that $\|\alpha_{r}\|\leq1+\varepsilon_{3}$
for all $r\in W_{2}$. Let $V_{0}\in\mathcal{Z}$ be such that $V_{0}\subseteq W_{1}\cap W_{2}$.
If $(V,i)\geq(V_{0},i_{0})$, then $V\subseteq V_{0}$ and $i\geq i_{0},$
hence
\begin{eqnarray*}
\|T((z_{V}\otimes u_{i}a)^{\vee})\| & \leq & \|T\|\left\Vert (z_{V}\otimes u_{i}a)^{\vee}\right\Vert _{1,\omega}\\
 & = & \|T\|\int_{G}\|(z_{V}\otimes u_{i}a)^{\vee}(r)\|\omega(r)\, dr\\
 & = & \|T\|\int_{G}\Delta(r^{-1})z_{V}(r)\|\alpha_{r}(au_{i})\|\omega(r)\, dr\\
 & \leq & \|T\|\|au_{i}\|(1+\varepsilon_{3})\left(\sup_{r\in V}\omega(r)\right)\int_{G}\Delta(r^{-1})z_{V}(r)\, dr\\
 & \leq & \|T\|(\|a\|+\varepsilon_{1})(1+\varepsilon_{3})\left(\sup_{r\in V_{0}}\omega(r)\right)\int_{G}z_{V}(r^{-1})\, dr\\
 & \leq & \|T\|(\|a\|+\varepsilon_{1})(1+\varepsilon_{3})\left(\inf_{V\in\mathcal{Z}}\sup_{r\in V}\omega(r)+\varepsilon_{2}\right).
\end{eqnarray*}
From this, the bound in (2) now follows as in the proof of Theorem
\ref{thm:continuous-non-deg-covars-are-R-continuous}.

As to (3), we fix $s\in G$. The operator $U_{s}^{T}=\textup{SOT-lim}_{(V,i)}T((z_{V}(\cdot s^{-1})\otimes u_{i})^{\vee})$
does not depend on the particular choice of the bounded approximate
right identity $(u_{i})$ (see Remark \ref{rem:eq-8.1and8.2-can-be-simplified-by-deleting-alpha}).
If $(u_{i})$ is an $M$-bounded approximate right identity of $A$,
then $(\alpha_{s^{-1}}(u_{i}))$ is also a bounded approximate right
identity of $A$, and therefore $U_{s}^{T}=\textup{SOT-lim}_{(V,i)}T((z_{V}(\cdot s^{-1})\otimes\alpha_{s^{-1}}(u_{i}))^{\vee})$.
Let $\varepsilon_{1},\varepsilon_{2}>0$ be arbitrary. Choose $W_{1}\in\mathcal{Z}$
such that $\|\alpha_{r}\|\leq1+\varepsilon_{1}$ for all $r\in W_{1}$,
and $W_{2}\in\mathcal{Z}$ such that $\sup_{r\in W_{2}}\omega(r)\leq\inf_{V\in\mathcal{Z}}\sup_{r\in V}\omega(r)+\varepsilon_{2}$.
Let $V_{0}\in\mathcal{Z}$ be such that $V_{0}\subseteq W_{1}\cap W_{2}$.
Fix some index $i_{0}$, then, for every $(V,i)\geq(V_{0},i_{0})$,
\begin{eqnarray*}
 &  & \|T((z_{V}(\cdot s^{-1})\otimes\alpha_{s^{-1}}(u_{i}))^{\vee})\| \\
 & \leq & \|T\|\left\Vert (z_{V}(\cdot s^{-1})\otimes\alpha_{s^{-1}}(u_{i}))^{\vee}\right\Vert _{1,\omega}\\
 & = & \|T\|\int_{G}\|(z_{V}(\cdot s^{-1})\otimes\alpha_{s^{-1}}(u_{i}))^{\vee}(r)\|\omega(r)\, dr\\
 & = & \|T\|\int_{G}\Delta(r^{-1})z_{V}(rs^{-1})\|\alpha_{rs^{-1}}(u_{i})\|\omega(r)\, dr\\
 & = & \|T\|\int_{G}z_{V}(r^{-1}s^{-1})\|\alpha_{r^{-1}s^{-1}}(u_{i})\|\omega(r^{-1})\, dr\\
 & = & \|T\|\int_{G}z_{V}(r^{-1})\|\alpha_{r^{-1}}(u_{i})\|\omega(r^{-1}s)\, dr\\
 & \leq & \|T\|\int_{G}z_{V}(r^{-1})\|\alpha_{r^{-1}}(u_{i})\|\omega(r^{-1})\omega(s)\, dr\\
 & \leq & \|T\|\int_{G}z_{V}(r^{-1})(1+\varepsilon_{1})\|u_{i}\|\left(\sup_{r\in V^{-1}}\omega(r^{-1})\right)\omega(s)\, dr\\
 & \leq & \|T\|(1+\varepsilon_{1})M\left(\sup_{r\in V^{-1}}\omega(r^{-1})\right)\omega(s)\int_{G}z_{V}(r^{-1})\, dr\\
 & \leq & \|T\|(1+\varepsilon_{1})M\left(\sup_{r\in V_{0}}\omega(r)\right)\omega(s)\\
 & \leq & \|T\|(1+\varepsilon_{1})M\left(\inf_{V\in\mathcal{Z}}\sup_{r\in V}\omega(r)+\varepsilon_{2}\right)\omega(s).
\end{eqnarray*}
Once again, the bound in (3) now follows as in the proof of Theorem
\ref{thm:continuous-non-deg-covars-are-R-continuous}.\end{proof}
We will now describe the non-degenerate bimodules over generalized
Beurling algebras as a special case of a more general result. Let
$(A,G,\alpha)$ and $(B,H,\beta)$ be Banach algebra dynamical systems,
where $A$ and $B$ have bounded two-sided approximate identities,
and both $\alpha$ and $\beta$ are uniformly bounded. Let $\omega$
be a weight on $G$, and $\eta$ a weight on $H$. Remembering that
$\BeurlingTypeAlg$ and $\BBeurlingTypeAlg$ are themselves also (isomorphic
to) a crossed product of a Banach algebra dynamical system, Theorem
\ref{thm:faithful-representation}, it is now easy to describe the
non-degenerate simultaneously left $\BeurlingTypeAlg$-- and right
$\BBeurlingTypeAlg$-modules, as follows: Let $X$ be a Banach space.
Suppose that $T^{m}:\BeurlingTypeAlg\to B(X)$ is a non-degenerate
bounded representation of $\BeurlingTypeAlg$ on $X$, and $T^{a}:\BBeurlingTypeAlg\to B(X)$
is a non-degenerate bounded anti-representation, such that $T^{m}$
and $T^{a}$ commute. We know from Theorem \ref{thm:continuous-non-deg-covars-are-R-continuous}
and Theorem \ref{thm:anti-correspondence} that $T^{m}$ and $T^{a}$
correspond to pairs $(\pi^{m},U^{m})$ and $(\pi^{a},U^{a})$, respectively,
each with the appropriate properties. But then $(\pi^{m},U^{m})$
and $(\pi^{a},U^{a})$ must also commute in the sense of Definition
\ref{def:commuting-maps}. Indeed, $(\pi^{a},U^{a})$ corresponds
to $T^{a}$ as being the pair such that the integrated form of $(\pi^{a,o},U^{a,o})$
gives rise to the non-degenerate bounded representation $T^{a}$ of
$\OppBBeurlingTypeAlg$ on $X$. But since $\OppBBeurlingTypeAlg$
is (isomorphic to) a crossed product, and likewise for $\BeurlingTypeAlg$,
the fact that $(\pi^{m},U^{m})$ and $(\pi^{a,o},U^{a,o})$ commute
then follows from Lemma \ref{commuting-covars-and-integrated-forms}
and the fact that $T^{m}$ and $T^{a}$ commute. Since $\pi^{a,o}=\pi^{a}$
and $U^{a,o}=U^{a}$ as set-theoretic maps, $(\pi^{m},U^{m})$ and
$(\pi^{a},U^{a})$ also commute. The same kind of arguments show that
the converse is equally true. 

Combining these results, we obtain the following following description
of the non-degenerate simultaneously left $\BeurlingTypeAlg$-- and
right $\BBeurlingTypeAlg$-modules. When $(A,G,\alpha)=(B,G,\beta)$
and $\omega=\eta$ it describes the non-degenerate $\BeurlingTypeAlg$-bimodules.
\begin{thm}
\label{thm:Beurling-bimodules}Let $(A,G,\alpha)$ and $(B,H,\beta)$
be a Banach algebra dynamical systems, where $A$ and $B$ have bounded
two-sided approximate identities, and both $\alpha$ and $\beta$
are uniformly bounded. Let $\omega$ be a weight on $G$, and $\eta$
a weight on $H$. Let $X$ be a Banach space.

Suppose that $(\pi^{m},U^{m})$ is a non-degenerate continuous covariant
representation of $(A,G,\alpha)$ on $X$ such that $\|U_{r}^{m}\|\leq C_{U^{m}}\omega(r)$
for some constant $C_{U^{m}}$ and all $r\in G$. Suppose that the
pair $(\pi^{a},U^{a})$ is such that $\pi^{a}:B\to B(X)$ is a non-degenerate
bounded anti-representation, that $U^{a}:H\to B(X)$ is a strongly
continuous anti-representation, such that $U_{s}^{a}\pi^{a}(b)U_{s}^{a-1}=\pi^{a}(\alpha_{s^{-1}}(b))$
for all $b\in B$ and $s\in H$, and $\|U_{s}^{a}\|\leq C_{U^{a}}\eta(s)$
for some constant $C_{U^{a}}$ and all $s\in H$. Furthermore, let
$(\pi^{m},U^{m})$ and $(\pi^{a},U^{a})$ commute.

Then the map 
\[
T^{m}(f):=\int_{G}\pi^{m}(f(r))U_{r}^{m}\, d\mu_{G}(r)\quad(f\in C_{c}(G,A))
\]
determines a non-degenerate bounded representation of $\BeurlingTypeAlg$
on $X$, and the map 
\[
T^{a}(g):=\int_{H}U_{s}^{a}\pi^{a}(g(s))\, d\mu_{H}(s)\quad(g\in C_{c}(H,B))
\]
determines a non-degenerate bounded anti-representation of $\BBeurlingTypeAlg$
on $X$. Moreover, $T^{m}:\BeurlingTypeAlg\to B(X)$ and $T^{a}:\BBeurlingTypeAlg\to B(X)$
commute.

All pairs $(T^{m},T^{a})$, where $T^{m}$ and $T^{a}$ commute, are
non-degenerate, bounded, $T^{m}$ is a representation of $\BeurlingTypeAlg$
on $X$, and $T^{a}$ is an anti-representation of $\BBeurlingTypeAlg$
on $X$, are obtained in this fashion from unique\textup{ (}necessarily
commuting\textup{)} pairs $(\pi^{m},U^{m})$ and $(\pi^{a},U^{a})$
with the above properties. 
\end{thm}
For reasons of space, we do not repeat the formulas in Theorem \ref{thm:continuous-non-deg-covars-are-R-continuous}
and Theorem \ref{thm:anti-correspondence} retrieving $(\pi^{m},U^{m})$
from $T^{m}$ and $(\pi^{a},U^{a})$ from $T^{a}$, or the upper bounds
therein.
\begin{rem}
\label{rem:applying-sun-tensor-to-beurling-bimodules}The results
of Section \ref{sec:Other-types} make it possible to establish a
bijection between the commuting pairs $(\pi^{m},U^{m})$ and $(\pi^{a},U^{a})$
as in Theorem \ref{thm:Beurling-bimodules} and the non-degenerate
bounded representations of one single algebra (rather than two). To
see this, note that, though $\BeurlingTypeAlg$ and $\OppBBeurlingTypeAlg$
are not Banach algebras in general, the continuity of the multiplication
still implies that $\BeurlingTypeAlg\hat{\otimes}\OppBBeurlingTypeAlg$
can be supplied with the structure of an associative algebra such
that multiplication is continuous. If $\BeurlingTypeAlg\simeq C_{1}$
and $\OppBBeurlingTypeAlg\simeq C_{2}$ as topological algebras, where
$C_{1}$ and $C_{2}$ are crossed products of the relevant Banach
algebra dynamical systems as in Section \ref{sec:Applications_L1_and_Beurling_algebras},
then clearly
\[
\BeurlingTypeAlg\hat{\otimes}\BBeurlingTypeAlg^{o}\simeq\BeurlingTypeAlg\hat{\otimes}\OppBBeurlingTypeAlg\simeq C_{1}\hat{\otimes}C_{2}
\]
where Proposition \ref{thm:anti-correspondence} was used in the first
step. From Theorem \ref{thm:sun-product-of-tensor-crossed-products}
we know what the non-degenerate bounded representations of $C_{1}\hat{\otimes}C_{2}$
are. Hence, combining all information, we see that the commuting pairs
$(\pi^{m},U^{m})$ and $(\pi^{a},U^{a})$ as in Theorem \ref{thm:Beurling-bimodules}
are in bijection with the non-degenerate bounded representations of
$\BeurlingTypeAlg\hat{\otimes}\BBeurlingTypeAlg^{o}$, by letting
$(\pi^{m},U^{m})$ and $(\pi^{a},U^{a})$ correspond to the non-degenerate
bounded representation $T^{m}\odot T^{a}$, where $T^{m}$ and $T^{a}$
are as in Theorem \ref{thm:Beurling-bimodules} (the latter now viewed
as a non-degenerate bounded representation of $\BBeurlingTypeAlg^{o}$).
Our notation is slightly imprecise here, since $\BeurlingTypeAlg$
and $\BBeurlingTypeAlg^{o}$ are not Banach algebras in general, but
it is easily seen that Lemma \ref{prop:sun-products-are-unique} is
equally valid when the norm need not be submultiplicative, but multiplication
is still continuous.
\end{rem}
To conclude, we note that the special case where $(A,G,\alpha)=(B,H,\beta)=(\mathbb{K},G,\textup{triv})$
in Theorem \ref{thm:Beurling-bimodules} states that the non-degenerate
bimodules over $L^{1}(G,\omega)$ correspond naturally to the $G$-bimodules
determined by a pair $(U^{m},U^{a})$ of commuting maps $U^{m}$ and
$U^{a}$, where $U^{m}:G\to B(X)$ is a strongly continuous representation,
$U^{a}:G\to B(X)$ is a strongly continuous anti-representation, and
$\|U_{r}^{m}\|\leq C_{U^{m}}\omega(r)$ and $\|U_{r}^{a}\|\leq C_{U^{a}}\omega(r)$
for some constants $C_{U^{m}}$ and $C_{U^{a}}$ and all $r\in G$.
Specializing further by taking $\omega=1$, we see that the non-degenerate
bimodules over $L^{1}(G)$ correspond naturally to the $G$-bimodules
determined by a commuting pair $(U^{m},U^{a})$ as above, with now
each of $U^{m}$ and $U^{a}$ uniformly bounded. This is a classical
result, cf.\,\cite[Proposition 2.1]{Johnson}.